\documentclass[11pt]{amsart}

\title{Higher Spherical Scissors Congruence II: Comodules}

\author[Klang]{Inbar Klang}
\address[Klang]{
Department of Mathematics,
Vrije Universiteit Amsterdam - Faculty of Science,
De Boelelaan 1111,
1081 HV Amsterdam,
The Netherlands}
\email{i.klang@vu.nl}

\author[Kuijper]{Josefien Kuijper}
\address[Kuijper]{
Department of Mathematics, University of Toronto
Bahen Centre, Room 6290
40 St. George St.,
Toronto, Ontario M5S 2E4,
Canada}
\email{josefien.kuijper@utoronto.ca}

\author[Malkiewich]{Cary Malkiewich}
\address[Malkiewich]{
Department of Mathematics, Binghamton University,
PO Box 6000, 
Binghamton, New York 13902,
USA}
\email{cmalkiew@binghamton.edu}

\author[Mehrle]{David Mehrle}
\address[Mehrle]{
Department of Mathematics, 
Colorado State University, 
Louis R. Weber Building,
841 Oval Drive, 
Fort Collins, Colorado 80526,
USA}
\email{david.mehrle@colostate.edu}

\author[Wittich]{Thor Wittich}
\address[Wittich]{
Department of Mathematics, Universit\"{a}t Osnabr\"{u}ck, 
Albrechtstrasse 28a, 49076 Osnabr\"{u}ck,
Germany}
\email{thor.wittich@uni-osnabrueck.de}

\usepackage{adjustbox}
\usepackage{amssymb,amsthm,amsmath} 
\usepackage{bbm} 
\usepackage{enumerate} 
\usepackage[margin=1.25in]{geometry} 
\usepackage[pagebackref,
			colorlinks,
			citecolor=OI3,
			linkcolor=OI3,
			urlcolor=OI5]{hyperref} 
\usepackage{mathrsfs} 
\usepackage{thmtools} 
\usepackage{tikz} 
\usepackage{tikz-cd} 
\usepackage{todonotes} 
\usepackage{xcolor} 
\usepackage[all]{xy}
\usepackage{relsize}
\usepackage{float}
			
\usepackage[nameinlink,capitalize]{cleveref} 

\definecolor{OI1}{RGB}{230,159,0}	
\definecolor{OI2}{RGB}{86,180,233}	
\definecolor{OI3}{RGB}{0,158,115}	
\definecolor{OI4}{RGB}{240,228,66}	
\definecolor{OI5}{RGB}{0,114,178}	
\definecolor{OI6}{RGB}{213,94,0}	
\definecolor{OI7}{RGB}{204,121,167}	

\newcommand{\Q}{\mathbb{Q}}
\newcommand{\R}{\mathbb{R}}
\newcommand{\Sph}{\mathbb{S}}
\newcommand{\Z}{\mathbb{Z}}

\DeclareMathAlphabet\mathbfcal{OMS}{cmsy}{b}{n} 

\newcommand{\sma}{\wedge} 

\makeatletter
\newcommand*{\boxwedge}{\mathbin{\mathpalette\@boxwedge{}}}
\newcommand*{\@boxwedge}[2]{%
  \sbox0{$#1\boxtimes\m@th$}%
  \dimen2=.5\dimexpr\wd0-\ht0-\dp0\relax 
  \dimen@=\dimexpr\ht0+\dp0\relax
  \def\lw{.07} 
  \kern\dimen2 
  \tikz[
    line width=\lw\dimen@,
    line join=round,
    x=\dimen@,
    y=\dimen@,
    baseline=0
  ]
  \draw
    (\lw/2,0) rectangle (1-\lw,1-\lw)
    (\lw,0) -- (.5,1-\lw-\lw/2) -- (1-\lw-\lw/2,0)
  ;%
  \kern\dimen2 
}
\makeatother

\newcommand{\apt}{\mathrm{apt}} 
\DeclareMathOperator{\Ls}{LS} 
\DeclareMathOperator{\PT}{PT} 
\DeclareMathOperator{\Pt}{Pt} 
\newcommand{\Sah}{\mathcal{S}} 
\newcommand{\funSah}{\mathfrak{S}} 
\DeclareMathOperator{\ST}{ST} 
\DeclareMathOperator{\tPt}{\widetilde{Pt}} 
\DeclareMathOperator{\T}{T} 
\DeclareMathOperator{\GL}{GL} 
\DeclareMathOperator{\Spc}{\mathbf{Spc}}

\DeclareMathOperator{\bhocolim}{hocolim^{(b)}}
\DeclareMathOperator{\hocolim}{hocolim} 
\newcommand{\id}{\mathrm{id}} 
\DeclareMathOperator{\pr}{pr} 
\DeclareMathOperator{\ob}{ob} 
\newcommand{\sd}{\operatorname{sd}} 
\newcommand{\Step}{\textup{Step}} 
\newcommand{\Sub}{\textup{Sub}} 

\DeclareMathOperator{\uhocolim}{hocolim^{(u)}}
\newcommand{\oplustimes}{
  \mathbin{\mathpalette\oplustimesaux\relax}%
}
\newcommand{\oplustimesaux}[2]{%
  \ooalign{\hfil$#1\oplus$\hfil\cr$#1\otimes$}%
}

\DeclareMathOperator{\op}{\textnormal{op}}

\DeclareMathOperator{\Ab}{Ab} 
\DeclareMathOperator{\Ch}{Ch} 
\DeclareMathOperator{\Alg}{Alg} 
\DeclareMathOperator{\Assoc}{Assoc} 
\DeclareMathOperator{\CAlg}{CAlg} 

\DeclareMathOperator{\CMon}{CMon} 
\DeclareMathOperator{\Comm}{Comm} 
\DeclareMathOperator{\RM}{RM}
\DeclareMathOperator{\Dip}{Dip} 
\DeclareMathOperator{\Luna}{Luna}

\newcommand{\Fun}{\mathrm{Fun}} 
\newcommand{\Pol}[2]{\mathcal{P}^{#1}_{#2}} 
\DeclareMathOperator{\Sp}{Sp} 
\DeclareMathOperator{\Spp}{\mathbf{Sp}} 
\newcommand{\Top}{\mathrm{Top}} 
\newcommand{\topp}{\mathrm{Top}_*} 
\DeclareMathOperator{\cof}{cof}

\newcommand{\C}{\mathcal{C}}

\numberwithin{equation}{section}
\numberwithin{figure}{section}

\newtheorem{corollary}[equation]{Corollary}
\newtheorem{lemma}[equation]{Lemma}
\newtheorem{proposition}[equation]{Proposition}
\newtheorem{theorem}[equation]{Theorem}

\newtheorem{lettertheorem}{Theorem} 
\newtheorem{letterconjecture}[lettertheorem]{Conjecture}

\theoremstyle{definition}
\newtheorem{conjecture}[equation]{Conjecture}
\newtheorem{construction}[equation]{Construction}
\newtheorem{definition}[equation]{Definition}
\newtheorem{example}[equation]{Example}

\newtheorem{remark}[equation]{Remark}

\crefname{equation}{}{} 
\crefname{figure}{figure}{figures} 

\AddToHook{env/conjecture/begin}{\crefalias{equation}{conjecture}}
\AddToHook{env/construction/begin}{\crefalias{equation}{construction}}
\AddToHook{env/corollary/begin}{\crefalias{equation}{corollary}}
\AddToHook{env/definition/begin}{\crefalias{equation}{definition}}
\AddToHook{env/example/begin}{\crefalias{equation}{example}}
\AddToHook{env/lemma/begin}{\crefalias{equation}{lemma}}
\AddToHook{env/lettercorollary/begin}{\crefalias{lettertheorem}{corollary}}
\AddToHook{env/lettertheorem/begin}{\crefalias{lettertheorem}{theorem}}
\AddToHook{env/notation/begin}{\crefalias{equation}{notation}}
\AddToHook{env/proposition/begin}{\crefalias{equation}{proposition}}
\AddToHook{env/question/begin}{\crefalias{equation}{question}}
\AddToHook{env/remark/begin}{\crefalias{equation}{remark}}
\AddToHook{env/theorem/begin}{\crefalias{equation}{theorem}}

\begin{document}
\begin{abstract}
	We define the higher Dehn invariant on the Euclidean, hyperbolic, and spherical variants of scissors congruence $K$-theory, and prove that it makes these spectra into comodules over the spectral Sah algebra. This allows us to make the first formulation of the conjecture that the higher volume and Dehn invariants are jointly injective.
\end{abstract}

\maketitle

\begingroup%
\setcounter{secnumdepth}{1}
\setcounter{tocdepth}{1}

\parskip 0em 
\tableofcontents
\endgroup%

\parskip 1em 

\section{Introduction}

Scissors congruence is the study of polytopes in Euclidean geometry $E^n$, hyperbolic geometry $H^n$, or spherical geometry $S^n$, up to the operation of cutting into finitely many smaller polytopes and rearranging. Two polytopes that are equivalent in this way are called \emph{equidecomposable}, or \emph{scissors congruent}. The modern side of this subject is \emph{higher scissors congruence}, which is concerned with the entire group of scissors congruences from a polytope to itself. The primary algebraic invariant in higher scissors congruence is the \emph{scissors congruence $K$-theory spectrum} $K(\Pol{X^n}{I(X^n)})$, or just $K(\mathcal X^n)$ for short, which classifies in a universal way all higher additive invariants of polytopes in the geometry $X^n$ under the action of the isometry group $I(X^n)$ \cite{inna-assem,scissors_thom}.

The lowest homotopy group $K_0(\mathcal X^n)$ is the classical scissors congruence group. The calculation of this group is equivalent to finding a complete set of invariants for polytopes up to scissors congruence. For instance, Hilbert's Third Problem is equivalent to computing $K_0(\mathcal E^3)$, the lowest $K$-group for three-dimensional Euclidean geometry. This problem is partly solved by Dehn \cite{dehn}, who shows in our language that the volume homomorphism
\[ K_0(\Pol{E^3}{I(E^3)}) \to \R \]
is not an isomorphism. This is because there is another homomorphism called the \emph{Dehn invariant}
\[ K_0(\Pol{E^3}{I(E^3)}) \to \R \otimes_\Z \R/\pi\Z. \]
The solution is then finished in celebrated work of Sydler and Jessen \cite{Sydler,jessen68}, who show that the volume and Dehn invariant together define an injective homomorphism, and furthermore the cokernel may be computed. Jessen also obtains a similar result in dimension 4 \cite{jessen72}, but since then the calculation of $K_0(\Pol{E^5}{I(E^5)})$ has been a tantalizing open question.

The calculation of the higher homotopy groups of $K(\Pol{X^n}{I(X^n)})$ is therefore the higher version of Hilbert's Third Problem. This higher version has seen a great deal of recent progress: a formula in terms of group homology that is correct rationally in general and integrally for Euclidean geometry \cite{scissors_thom}, a description in terms of groups of scissors automorphisms \cite{KLMMS-1}, a complete calculation in the one-dimensional case \cite{scissors_thom}, and a simplification in the two-dimensional Euclidean case \cite{KLMMS-3} that allows us to prove that $K_2$ and above are infinite \cite{holley_thesis} and to compute the $K$-groups explicitly for rational vertices \cite{KLMMS-3}. However, as the dimension of the geometry increases, computational descriptions get more complicated, and it is natural to search for more structural results.

One natural conjecture is that the volume and Dehn invariant should be jointly injective, just like they are for $K_0$ of $E^3$ and $E^4$. (And also for $K_0$ of $E^1$ and $E^2$, where there is no Dehn invariant, but the volume homomorphism is injective.) However, to even formulate this conjecture one needs the higher version of the Dehn invariant, defined on the scissors congruence $K$-theory spectrum $K(\mathcal X^n)$ and not just on the lowest homotopy group $K_0(\mathcal X^n)$.

In this paper we give the first precise definition of this higher Dehn invariant. The case that this is the correct definition rests on several points. First, it agrees with the classical Dehn invariant on $K_0$. Second, in the special case of the full isometry group and after rationalization, it agrees with the comultiplication on homology defined by Cathelineau \cite{cathelineau_sc}. Third, on homology our Dehn invariant agrees with the rigid Dehn invariants of Campbell and Zakharevich \cite{cz-hilbert}. And fourth, this Dehn invariant has the coassociativity and compatibility we expect with various products at the level of the $\infty$-category of spectra, not just at the chain level.

At the level of homology, our definition of the Dehn invariant admits a relatively short description, which suffices for many computations and applications (\cref{pt_dehn}). The Dehn invariant as a map of spectra is much more involved, requiring a combination of geometric configuration-space arguments, formal arguments with operads, and $\infty$-categorical machinery. Although it is a great deal of work, there are several payoffs: first, the fact that the construction works is further evidence that this is the correct higher version of the Dehn invariant. Second, many scissors congruence spectra are not rational or not known to be rational, and in those cases the map is not determined by the map induced on homology. And finally, the configuration-space description comes with beautiful geometric pictures that give further insight into the structure of the higher Dehn invariant.

To state the theorem, let us take the following wedge sums of scissors congruence $K$-theory spectra, one for each dimension of a given classical geometry:
\[ \mathcal{E} \simeq \bigvee_{n \geq 0} K(\Pol{E^n}{I(E^n)}), \quad \mathcal{H} \simeq \bigvee_{n \geq 0} K(\Pol{H^n}{I(H^n)}), \quad \mathcal{S} \simeq \bigvee_{n \geq 0} K(\Pol{S^{n-1}}{O(n)}), \quad \widetilde{\mathcal{S}} \simeq \bigvee_{n \geq 0} \widetilde{K}(\Pol{S^{n-1}}{O(n)}). \]
See \cref{spectral_sah_module} for a more precise definition. At the level of $\pi_0$, these give the classical scissors congruence groups of each geometry, summed over all dimensions $n \geq 0$.

\begin{lettertheorem}\label{mainthm}
    In the $\infty$-category of spectra, 
    \begin{enumerate}
    \vspace{-\topsep}
    \vspace{-\partopsep}
    \vspace{-\parskip}
    \item[(i)] the spectrum $\widetilde{\mathcal{S}}$ is a commutative Hopf algebra;
    \item[(ii)] the spectra $\mathcal{E}$ and $\mathcal{S}$ are comodule algebras over $\widetilde{\mathcal{S}}$; and
    \item[(iii)] the spectrum $\mathcal{H}$ is a comodule over $\widetilde{\mathcal{S}}$.
\end{enumerate}
\end{lettertheorem}
Furthermore, we show that the induced structure at the level of $\pi_0$ is the correct one:
\begin{lettertheorem}\label{mainthm_pi0}
    The induced comultiplications on $\pi_0(\mathcal{E})$, $\pi_0(\mathcal{H})$, and $\pi_0(\mathcal{S})$ are identified with the generalized Dehn invariants in each of these geometries. The induced product on $\pi_0(\mathcal{E})$ is the Cartesian product of polytopes $P \times Q$, and the product on $\pi_0(\mathcal{S})$ is the (orthogonal) join product of spherical polytopes $P * Q$.
\end{lettertheorem}

We also show that in the case of translational scissors congruence on Euclidean space, our higher Dehn invariant gives a higher version of the Hadwiger invariant of \cite{hadwiger_invts,sah_79}. This is striking, because it implies that these apparently different classical scissors congruence invariants were really special cases of the same general construction.

With this construction in place, we can now give a precise formulation of the conjecture that the volume and higher Dehn invariant are jointly injective:

\begin{letterconjecture}[Higher Dehn--Sydler--Jessen Conjecture]\label{higher_dsj}
    The composite of the higher Dehn invariant of this paper and the volume trace of \cite{bgmmz} gives an injective homomorphism on the homotopy groups
    \begin{eqnarray*}
        K_m(\Pol{X^n}{I(X^n)}) &\to& \bigoplus_{0 \leq c \leq n} \bigoplus_{i+j=m} K_i(\Pol{X^{n-c}}{I(X^{n-c})}) \otimes \widetilde{K_j}(\Pol{S^{c-1}}{O(c)}) \\
        &\to& \bigoplus_{0 \leq c \leq n} \bigoplus_{i+j=m} H_i(I(X^{n-c});\R) \otimes \widetilde{K_j}(\Pol{S^{c-1}}{O(c)}).
    \end{eqnarray*}
\end{letterconjecture}

We can therefore begin to investigate this conjecture.

Surprisingly, some work has already been done on this question, because in $E^1$ and $E^2$ there is no nonzero Dehn invariant. So in these cases the conjecture is equivalent to the simpler claim that the volume gives injective homomorphisms
\[ K_m(\Pol{E^1}{I(E^1)}) \to H_m(I(E^1);\R), \qquad K_m(\Pol{E^2}{I(E^2)}) \to H_m(I(E^2);\R). \]
In fact, recent work of Lemann \cite{lemann} shows that this is correct for one-dimensional Euclidean geometry, and therefore also for one-dimensional hyperbolic geometry. Forthcoming work of Fangji Liu and the third author also proves this for the weight-two summands in two-dimensional Euclidean geometry. And of course at $K_0$ this becomes the classical open question:
\begin{letterconjecture}\label{dsj}
    The volume and Dehn invariants are jointly injective on $K_0(\mathcal E^n)$, $K_0(\mathcal H^n)$, and $K_0(\mathcal S^n)$.
\end{letterconjecture}
By a theorem of Zylev \cite{zylev}, this is equivalent to saying that the volume and the classical Dehn invariants give a complete set of obstructions for two polytopes to be scissors congruent. At the time of writing this is open already in the case of $E^5$, $H^3$, or $S^4$.

In future work we intend to show that \cref{higher_dsj} is true rationally for spherical geometry $S^n$ in all dimensions $n \geq 0$. Therefore the difficult part of the calculation is the computation of \emph{reduced} spherical scissors congruence $K$-theory, which already for $S^3$ is quite mysterious. (At the level of $K_0$ this is also a theorem of Sah \cite{sah_79}; in other words, \cref{dsj} is true rationally for spherical geometry of any dimension.)

\subsection*{AI Declaration}

The mathematics of this paper was carried out, and the paper itself was written and edited, without the use of AI.

\subsection*{Acknowledgments.}
    JK was supported by the Knut and Alice Wallenberg grant KAW 2023.0416. 
	CM was partially supported by the National Science Foundation (NSF) grant DMS-2506430, the Simons Fellowship [SFI-MPS-SFM-00011827, CM] and a travel support grant from Simons Foundation International [SFI-MPS-TSM-00013425, CM].
	DM was partially supported by the NSF grant DMS-2135884.
    The authors thank the University of Toronto for their hospitality during a week long collaborative visit in May 2026.

\section{Geometries and vector spaces}\label{sec:geometries}

We begin by explicitly defining the geometries we will study, and the structure of their geometric subspaces and orthogonal complements. We organize this data in a coordinate-free way into categories with various multiplicative structures, which we will use to index our multiplications and comultiplications in scissors congruence.

\subsection{Three geometries}
We are interested in the three classical geometries with constant curvature.

\begin{definition}
Let $n \geq 0$. A \textit{geometry} $X$ is one of the following:
\begin{enumerate}[(i)]
    \item spherical geometry 
    \[
        S^{n-1} = \left\{(x_1,\ldots,x_n) \in \R^n\  \middle\vert\ \sum_{i=1}^{n} x_i^2 = 1\right\},
    \]
    \item Euclidean geometry
    \[
        E^n = \bigg\{(x_0,\ldots,x_{n-1},x_n) \in \R \oplus \R^n\  \bigg\vert\  x_0 = 1 \bigg\}, \text{ or} 
    \]
    \item hyperbolic geometry 
    \[
        H^n = \left\{(x_0,\ldots,x_{n-1},x_n) \in \R \oplus \R^n \  \middle\vert\  -x_0^2+\sum_{i=1}^{n} x_i^2 = -1 \text{ and }x_0>0\right\}.
    \]
\end{enumerate}
\end{definition}
In each of these cases, the larger vector space $\R^n$ or $\R \oplus \R^n$ is called the \emph{ambient vector space}. Note that when $n = 0$, the spherical geometry $S^{-1}$ is empty, while the other two geometries $E^0$ and $H^0$ have a single point.

We choose to define Euclidean spaces as affine subspaces of $\R \oplus \R^n$ in order to obtain the expected isometry group, as a subgroup of $\GL(\R \oplus \R^n)$.
    
The convention that Euclidean and hyperbolic geometry are $n$-dimensional, while spherical geometry is $(n-1)$-dimensional, is convenient for the current work because it captures all of the dimensions of faces that can possibly occur when we take the Dehn invariant: an empty spherical polytope can appear, but an empty Euclidean or hyperbolic polytope cannot.

\begin{definition}\label{df:geometric_subspace}
   Let $X$ be a geometry. A \emph{geometric subspace of $X$} is any subset of $X$ of the form $Y = X \cap U$, where $U$ is a vector subspace of the ambient vector space $\R \oplus \R^n$ or $\R^n$ associated with $X$. 
\end{definition}

Note that if $U$ and $U'$ are two vector subspaces of $\R^n$ such that $U \cap S^{n-1} = U' \cap S^{n-1}$, then $U = U'$. In Euclidean and hyperbolic geometry $X$, the same fact is true only for nonempty geometric subspaces, i.e., if $U \cap X = U' \cap X$ is nonempty, then $U = U'$.  

\begin{definition}
    Let $X$ be a geometry. 
    The \emph{isometry group of $X$}  is the subgroup $I(X)$ of $\GL(\R \oplus \R^n)$ or $\GL(\R^n)$ defined as follows:
	\begin{enumerate}[(i)]
		\item The group $I(S^{n-1}) = O(n)$ consists of linear transformations preserving the quadratic form $\sum_{i=1}^{n} x_i^2$.
		\item The group $I(E^n)$ consists of linear transformations preserving the form $\sum_{i=1}^n x_i^2$ on $\R^n$ and the coordinate $x_0$. It is a semidirect product of the group of translations with the orthogonal group:
		\[ 
            I(E^n) \cong \R^{n} \rtimes \mathrm{O}(n). 
        \]
		\item The group $I(H^n)$ consists of linear transformations preserving the quadratic form $-x_0^2+\sum_{i=1}^n x_i^2$ and the sign of $x_0$. It is an index-two subgroup of the indefinite orthogonal group $O(1,n)$ given by those matrices $A \in O(1,n)$ with $A_{11} > 0$. Some people call $I(H^n)$ the \textit{orthochronous Lorentz group}. 
	\end{enumerate}
    These are all considered as discrete groups. In what follows we will focus on using the full isometry group $I(X)$, but in \cref{sec:restrict_isometry_group} we will restrict attention to various subgroups, such as the translation subgroup $\R^n \leq I(E^n)$.
\end{definition}

\begin{remark}
    If we were to equip the three geometries with their natural metrics, then the above isometry groups would literally be the isometry groups, i.e., the groups of bijective distance-preserving maps. However, we prefer to use a broader notion of ``geometry'' which is not necessarily equipped with a metric, only with a symmetry group, as above.
\end{remark}

\subsection{Coordinate-free geometries}

Since geometric subspaces of a geometry $X$ correspond to certain vector subspaces of the corresponding ambient vector space,
we may reinterpret the geometries above in a coordinate-free manner, referencing only intrinsic data of the ambient vector space. We use this idea to define PROPs $\Dip^S$, $\Dip^E$, and $\Dip^H$, which we will use to define the comodule spectra of scissors congruence. The category $\Dip^S$ was already used in \cite{KKMMW1} to define the commutative Hopf algebra structure of the spherical Sah algebra. 

\begin{definition} 
\label{coordinate-free-dip}
    We define three categories.
    \begin{enumerate}[(a)]
        \item Let $\Dip^S$ be the category with objects $(V,q)$ where $V$ is a finite-dimensional real vector space and $q$ is a positive definite quadratic form on $V$. A morphism in $\Dip^S$ is an isomorphism preserving the form $q$. Note that the automorphism group of $(V,q) \in \Dip^S$ is $O(V) \cong I(S^{\dim V - 1})$. 
        \item Let $\Dip^E$ be the category with objects $(V,x_0,q)$ where $V$ is a finite-dimensional real vector space of \textit{positive} dimension, $x_0 \colon V \to \R$ is a rank one linear functional with kernel $V_0 \subsetneq V$, and $q$ is a positive definite quadratic form \textit{defined only on} $V_0$. Morphisms in $\Dip^E$ are isomorphisms sending one functional to the other and preserving the quadratic form on the level of the kernels. The automorphism group of $(V,x_0,q) \in \Dip^E$ is isomorphic to $I(E^{\dim(V) -1})$.
        \item Let $\Dip^H$ be the category with objects $(V,q,\varepsilon)$ where $V$ is a finite-dimensional real vector space of \textit{positive} dimension and $q$ is a nondegenerate quadratic form on $V$ with exactly one negative eigenvalue and the rest of the eigenvalues positive\footnote{
Recall that every quadratic form on $\R^n$ is isomorphic to one of the form $\sum_{i=1}^a -x_i^2 + \sum_{i=a+1}^{a+b} x_i^2$, with $a$ negative eigenvalues and $b$ positive eigenvalues. So a quadratic form with one negative eigenvalue is always of the form $-x_0^2 + \sum_{i=1}^n x_i^2$ in some basis of $\R^{1+n}$.}, and $\varepsilon$ is a choice of orientation\footnote{The negative eigenspace is not well-defined as a subspace of $V$, so by ``orientation'' we mean that the set $\{v \in V \mid q(v) < 0\}$ has two connected components, and we select one as the ``positive" component. Along any line where $q(v) < 0$ for all nonzero vectors $v$, this induces an orientation of that line, which selects a ``positive'' direction.}
for the negative eigenspace of $q$. A morphism in $\Dip^H$ is an isomorphism of vector spaces preserving the form $q$ and the orientation of its negative eigenspace. The automorphism group of $(V,q,\varepsilon) \in \Dip^H$ is isomorphic to $I(H^{\dim(V) - 1})$. 
    \end{enumerate}
\end{definition}

We use $\Dip^X$ to refer generically to all three of these examples. Though each of the morphism sets in $\Dip^X$ has a natural topology, we ignore it and consider the morphisms to be discrete spaces. This is consistent with our earlier convention that the automorphism groups $I(X^n)$ should be discrete groups.

Note that the objects of $\Dip^X$ are always vector spaces with extra structure. We will often refer to objects of $\Dip^X$ simply by the vector space, leaving the rest of the structure implicit. We can always recover the geometry as follows:
\begin{itemize}
    \item If $V \in \Dip^S$, the corresponding geometry is the unit sphere where $q = 1$.
    \item If $V \in \Dip^E$, the corresponding geometry is the affine subspace where $x_0 = 1$.
    \item If $V \in \Dip^H$, the corresponding geometry is the ``positive'' component of the hypersurface where $q = -1$.
\end{itemize}

We also have a notion of a subobject $U$ for each $V \in \Dip^X$, which corresponds to a geometric subspace of the appropriate geometry.
\begin{definition}\label{dip-subobject}
A \emph{subobject} of $V \in \Dip^X$ is any vector subspace $U \subseteq V$ that defines another object of $\Dip^X$ in a natural way. To be more specific:
\begin{enumerate}[(a)]
    \item If $V \in \Dip^S$, a subobject of $V$ is any vector subspace $U \subseteq V$, with inner product given by restricting $q$ from $V$ to $U$.
    \item If $V \in \Dip^E$, a subobject of $V$ is any vector subspace $U \subseteq V$ such that the restriction $x_0|_U:U\to \R$ is still rank one. We restrict $q$ to $U_0 = V_0\cap U$ to define the quadratic form on $U_0$.
    \item If $V \in \Dip^H$, a subobject of $V$ is any vector subspace $U \subseteq V$ that contains at least one vector $u \in U$ such that $q(u) < 0$, with quadratic form given by restricting $q$ to $U$. The vector $u$ guarantees that $q$ has a nontrivial negative eigenspace. 
\end{enumerate}
\end{definition}
Note that these are not subobjects in the usual category-theoretic sense, because the category $\Dip^X$ is a groupoid, so any subobject of $V$ in the category-theoretic sense must be isomorphic to $V$.

\begin{lemma}
    In $E^n$ and $H^n$, these subobjects correspond bijectively to nonempty geometric subspaces of the corresponding geometry. In $S^{n-1}$ they correspond to all geometric subspaces, including the empty subspace.
\end{lemma}

\begin{proof}
    A geometric subspace is an equivalence class of vector subspaces $U \subseteq V$, where two are equivalent if they have the same intersection with the geometry $X \subseteq V$. In the spherical case, the intersection with the unit sphere completely determines the vector space, so we are left with all vector subspaces of $V$. In the Euclidean case, the geometric subspace $U \cap X$ is nonempty when it contains a vector $v$ with $x_0(v) = 1$, and then the remaining basis vectors may be taken in the kernel of $x_0$; this shows that the linear subspace $U$ is uniquely identified by its intersection with $X = \{v : x_0(v) = 1\}$. In the hyperbolic case, the geometric subspace $U \cap X$ is nonempty when it contains a vector $v$ with $q(v) = -1$. Up to coordinate change this vector is $(1,0,\ldots,0)$, and then the remainder of $U \cap \R^n$ is determined from $U \cap X$ as the tangent space at $(1,0,\ldots,0)$.
\end{proof}

\subsection{Direct sums and differences}
In order to define our comultiplication maps, we will need to take direct sums of objects $V \in \Dip^X$ for $X = S, E$, or $H$ with inner product spaces $W \in \Dip^S$.
\begin{definition}\label{dip-direct-plus}
The \emph{direct sum} functor
\[ \oplus\colon \Dip^X \times \Dip^S \to \Dip^X \]
is defined for $V \in \Dip^X$ and $W \in \Dip^S$ as follows.
\begin{itemize}
    \item If $V \in \Dip^S$, the direct sum $V \oplus W \in \Dip^S$ is the usual direct sum of inner product spaces, with $q((v,w)) = q_V(v) + q_W(w)$, so that $V$ and $W$ are orthogonal complements in the resulting space.
    \item If $V \in \Dip^E$, the direct sum $V \oplus W \in \Dip^E$ is the usual direct sum of vector spaces, mapping to $\R$ by projecting to $V$ and applying $x_0 \colon V \to \R$. The kernel of this map is the direct sum $V_0 \oplus W$, and we define the inner product on this by $q((v,w)) = q_{V_0}(v) + q_W(w)$ as above.
    \item If $V \in \Dip^H$, the direct sum $V \oplus W \in \Dip^H$ is the usual direct sum of vector spaces, with indefinite quadratic form $q((v,w)) = q_V(v) + q_W(w)$. Furthermore we declare a vector of the form $(v,0)$ with $q(v) < 0$ to have positive orientation if and only if $v$ has positive orientation as a vector in $V$.
\end{itemize}
\end{definition}
Of course, we can just as easily define the direct sum $W \oplus V$ when $W \in \Dip^S$ and $V \in \Dip^X$ in the same way.

We can also invert this operation by assigning each subobject $U \subseteq V$ in $\Dip^X$ to an orthogonal complement $V \ominus U$ in $\Dip^S$. The complement is always an inner product space, no matter which $X$ we started with.
\begin{definition}\label{dip-direct-minus}
The \emph{direct difference} is the partially defined functor
\[ \ominus\colon \Dip^X \times \Dip^X \to \Dip^S \]
defined for all pairs $(U,V)$ in which $U$ is a subobject of $V$, as follows.
\begin{itemize}
    \item If $X = S$ or $X = H$, the direct difference $V \ominus U$ is the orthogonal complement of $U$ in $V$. In other words
    \[ V \ominus U := \{ v \in V : b(v,u) = 0 \ \forall u \in U \}. \]
    Here $b(-,-)$ is the associated bilinear form to the quadratic form $q$.
    \item If $X = E$, the direct difference $V \ominus U$ is the orthogonal complement of $U_0 = \ker(x_0\colon U \to \R)$ in $V_0 = \ker(x_0\colon V \to \R)$:
    \[ V \ominus U := \{ v \in V_0 : b(v,u) = 0 \ \forall u \in U_0 \}. \]
    This is because the bilinear form $b(-,-)$ is only defined on $V_0$.
\end{itemize}
In all of these cases, $V \ominus U$ inherits a positive definite quadratic form from $V$ (or $V_0$ in the Euclidean case), so that $V \ominus U \in \Dip^S$.
\end{definition}

\begin{example}
    In hyperbolic geometry, the notion of orthogonality can be counterintuitive. Let $V$ be $\R^2$ equipped with the quadratic form $-x_0^2 + x_1^2$, and let $U = \langle (2,1) \rangle$ be the vector subspace of $V$ spanned by $(2,1)$. Then $V \ominus U$ is the span of $(1,2)$, which is not visually perpendicular to the line $U$ in the Euclidian sense, as illustrated below.  
    
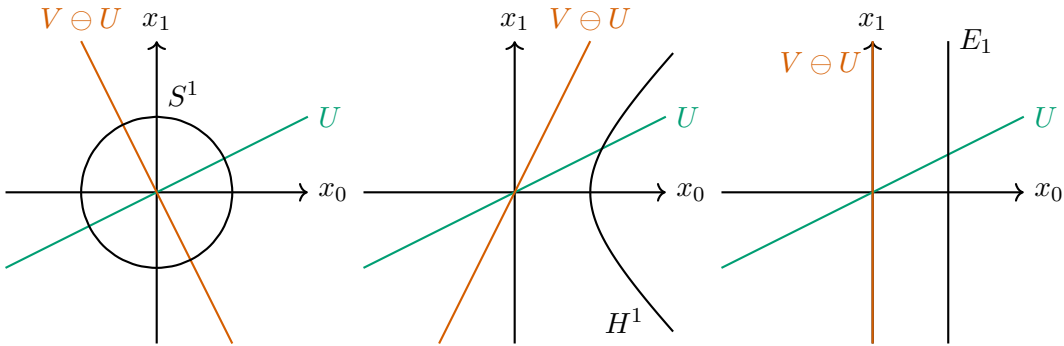
\[
         \begin{tikzpicture}   
            \draw[thick,->] (-2,0) -- (2,0) node[right]{$x_0$}; 
            \draw[thick,->] (0,-2) -- (0,2) node[above]{$x_1$};
            \draw[thick,OI3] (-2,-1) -- (2,1) node[right]{$U$};
            \draw[thick,OI6] (1,-2) -- (-1,2) node[above]{$V \ominus U$};
            \draw[thick,domain=90:360+90,variable=\u,smooth] plot ({cos(\u)},{sin(\u)})
            node[above right]{$S^1$};
        \end{tikzpicture}\,
        \begin{tikzpicture}   
            \draw[thick,->] (-2,0) -- (2,0) node[right]{$x_0$}; 
            \draw[thick,->] (0,-2) -- (0,2) node[above]{$x_1$};
            \draw[thick,OI3] (-2,-1) -- (2,1) node[right]{$U$};
            \draw[thick,OI6] (-1,-2) -- (1,2) node[above]{$V \ominus U$};
            \draw[thick,domain=-pi/2+0.2:pi/2-0.2,variable=\u,smooth] plot  ({cosh(-\u)},{sinh(-\u)}); 
            \node at (1.45,-1.7) {$H^1$};
        \end{tikzpicture}\,
                 \begin{tikzpicture}   
            \draw[thick,->] (-2,0) -- (2,0) node[right]{$x_0$}; 
            \draw[thick,->] (0,-2) -- (0,2) node[above]{$x_1$};
            \draw[thick,OI3] (-2,-1) -- (2,1) node[right]{$U$};
            \draw[thick,OI6] (0,-2) -- (0,2) node[ below left]{$V \ominus U$};
            \draw[thick] (1,-2) -- (1,2) node[right]{$E_1$};
        \end{tikzpicture}
    \]
\end{example}
\begin{remark}
    If we were to give each geometry $X$ its standard Riemannian metric, take any geometric subspace $U \cap X$, take the orthogonal complement to the tangent space $T_u (U \cap X) \subseteq T_u X$ for each point $u \in U \cap X$ and consider it as a subspace of the ambient vector space $V$, the result is \emph{always} $V \ominus U$ for every point $u \in U \cap X$. In spherical and Euclidean geometry, this is easy to see. In hyperbolic geometry, it follows from the fact that the tangent space to $H^n$ at $u \in U \cap H^n$ consists of those vectors $v$ such that $-2u_0v_0 + \sum_{i=1}^n 2u_iv_i = 0$, whereas $v \in V \ominus U$ iff $-u_0v_0 + \sum_{i=1}^n u_iv_i = 0$ for \emph{all} $u \in U$. These are the same equation, so $V \ominus U \subseteq T_u H^n$ for all $u \in U \cap H^n$. On the other hand, the Riemannian metric on $T_u H^n$ turns out to be given by the indefinite form $q$ restricted to $T_u H^n \subseteq V$, and so in $T_u H^n$ the subspace $V \ominus U$ is indeed the orthogonal complement of $T_u (U \cap H^n)$.
\end{remark}

\subsection{Symmetric monoidal structures}
The direct sum makes the category of vector spaces into a symmetric monoidal category, as well as $\Dip^S$. Forgetting the extra structure, the above direct sum operations therefore have isomorphisms with the same coherences between them. However, since the indexing category changes, it's convenient to build larger categories which are symmetric monoidal with this structure. The Euclidean version of this category also incorporates a symmetric monoidal structure on $\Dip^E$:

\begin{definition}\label{dip-cartesian}
The \emph{Cartesian product} functor
\[ \times\colon \Dip^E \times \Dip^E \to \Dip^E \]
is defined for vector spaces $V,W$ by by taking the fiber product $V \times_\R W$, and taking the direct sum of their inner products in the kernel space $V_0 \oplus W_0$.
\end{definition}

\begin{lemma}\label{dip-cartesian-coherent}
    This makes $\Dip^E$ into a symmetric monoidal category.
\end{lemma}

\begin{proof}
    This follows from the fact that the fiber product $\times_\R$ is unital, associative, and commutative up to coherent isomorphism.
\end{proof}

We can now define a larger category $\Luna^X$ for $X = S,E$, and $H$.

\begin{definition}\label{def-Luna}
    For a geometry $X$ we define the Linear sqUare-Null Alteration category, or $\Luna^X$ for short, as follows: 
    \begin{itemize}
        \item $\Luna^S = \Dip^S$,
        \item $\Luna^E = \Dip^E \vee \Dip^S$ where the wedge sum identifies one zero-dimensional Euclidean geometry given by $V = \R$, $x_0 = \id_\R$ in $\Dip^E$ with one zero-dimensional vector space in $\Dip^S$, and
        \item $\Luna^H = {*} \sqcup \Dip^H \sqcup \Dip^S$ where $*$ is the trivial one-point category with only the identity morphism.
    \end{itemize}
    
    We define these differently because the multiplicative structure that exists is different in each case: $\Dip^S$ has a direct sum functor, while $\Dip^E$ has a Cartesian product and a direct sum with $\Dip^S$. On the other hand, $\Dip^H$ does not have a Cartesian product functor, only a direct sum with $\Dip^S$. So we use the $*$ as a ``trash can'' to receive the product of two objects in $\Dip^H$.

    These products define the following bifunctors on $\Luna^X$:

    \begin{itemize}
        \item $\oplus: \Luna^S \times \Luna^S \to \Luna^S$ is $\oplus: \Dip^S \times \Dip^S \to \Dip^S$, which is direct sum of inner product spaces.
        \item $\oplustimes: \Luna^E \times \Luna^E \to \Luna^E$ is given by the functors
        \begin{itemize}
            \item $\oplus: \Dip^S \times \Dip^S \to \Dip^S$,
            \item $\oplus: \Dip^E \times \Dip^S \to \Dip^E$,
            \item $\oplus: \Dip^S \times \Dip^E \to \Dip^E$, and
            \item $\times: \Dip^E \times \Dip^E \to \Dip^E$.
        \end{itemize}
        \item $\oplus: \Luna^H \times \Luna^H \to \Luna^H$ is given by the functors
        \begin{itemize}
            \item $\oplus: \Dip^S \times \Dip^S \to \Dip^S$,
            \item $\oplus: \Dip^H \times \Dip^S \to \Dip^H$,
            \item $\oplus: \Dip^S \times \Dip^H \to \Dip^H$, and 
            \item ${*}: (\Dip^H \times \Dip^H) \sqcup (* \times \Luna^H) \sqcup (\Luna^H \times *) \to {*}$.
        \end{itemize}
    \end{itemize}
    
\end{definition}

\[
\begin{tikzpicture}
    \node at (-1,4) {$\Luna^E$};
    \node (top) at (0,4) {$\times$};
    \node at (1,4) {$\Luna^E$};
    \node (A2) at (-2.05,3) {$\Dip^E$};
    \node at (-1.4,2.97) {$\vee$};
    \node (A3) at (-0.8,3) {$\Dip^S$};
    \node (middle) at (0,3) {$\times$};
    \node (B2) at (0.9,3) {$\Dip^E$};
    \node at (1.55, 2.97) {$\vee$};
    \node (B3) at (2.15,3) {$\Dip^S$};
    \node (C2) at (-0.65,0) {$\Dip^E$};
    \node at (0,-0.03) {$\vee$};
    \node (C3) at (0.6,0) {$\Dip^S$};
    \draw[double equal sign distance] (top) -- (middle);
    \draw[->,thick,OI3] (A2.270) -- (C2.135);
    \draw[->,thick,OI3] (B2.-70) -- (C2.45);
    \draw[OI1,->,thick] (A3) -- (C3);
    \draw[OI1,->,thick] (B3) -- (C3);
    \draw[OI5,->,thick] (A2) -- (C2);
    \draw[OI5,->,thick] (B3) -- (C2.30);
    \draw[OI2,->,thick] (A3) -- (C2);
    \draw[OI2,->,thick] (B2.250) -- (C2.70);
\end{tikzpicture}
\hspace*{2cm}
\begin{tikzpicture}
    \node at (-1,4) {$\Luna^H$};
    \node (top) at (0,4) {$\times$};
    \node at (1,4) {$\Luna^H$};
    \node (A1) at (-3,3) {$\ast$};
    \node at (-2.7,2.97) {$\sqcup$};
    \node (A2) at (-2.05,3) {$\Dip^H$};
    \node at (-1.4,2.97) {$\sqcup$};
    \node (A3) at (-0.8,3) {$\Dip^S$};
    \node (middle) at (0,3) {$\times$};
    \node (B1) at (0.5,3) {$\ast$};
    \node at (0.8,2.97) {$\sqcup$};
    \node (B2) at (1.45,3) {$\Dip^H$};
    \node at (2.1, 2.97) {$\sqcup$};
    \node (B3) at (2.7,3) {$\Dip^S$};
    \node (C1) at (-1,0) {$\ast$};
    \node at (-0.7,-0.03) {$\sqcup$};
    \node (C2) at (0.05,0) {$\Dip^H$};
    \node at (0.7,-0.03) {$\sqcup$};
    \node (C3) at (1.3,0) {$\Dip^S$};
    \draw[double equal sign distance] (top) -- (middle);
    \draw[OI3,->,thick] (A2) -- (C1);
    \draw[OI3,->,thick] (B2) -- (C1);
    \draw[OI1,->,thick] (A3) -- (C3);
    \draw[OI1,->,thick] (B3) -- (C3);
    \draw[OI2,->,thick] (A2) -- (C2);
    \draw[OI2,->,thick] (B3) -- (C2);
    \draw[OI5,->,thick] (A3) -- (C2);
    \draw[OI5,->,thick] (B2) -- (C2);
\end{tikzpicture}
\]

\begin{lemma}
    The usual isomorphisms of vector spaces $0 \oplus V \cong V$, $(V_1 \oplus V_2) \oplus V_3 \cong V_1 \oplus (V_2 \oplus V_3)$, and $V_1 \oplus V_2 \cong V_2 \oplus V_1$, and the isomorphisms from \cref{dip-cartesian-coherent} in the Euclidean case, make each $\Luna^X$ into a symmetric monoidal category.
\end{lemma}

\begin{proof}
    All that we have to check is that each of these isomorphisms is a map in the appropriate flavor of $\Dip^X$. When all of the factors are in $\Dip^S$, this is the usual fact that the isomorphisms preserve the inner product. When one of the factors is in $\Dip^E$, we observe that the isomorphisms commute with the projection to $\R$ and preserve the inner product on the kernel of that projection. When both factors are in $\Dip^E$, this is the content of \cref{dip-cartesian-coherent}. (When both factors are the wedge point that is in both $\Dip^E$ and $\Dip^S$, we take their product to be the wedge point again.) When one of the factors is in $\Dip^H$, the isomorphisms preserve the quadratic form $q$ and the orientation of its negative eigenspace. When more than one factor is in $\Dip^H$, the sums are both trivial and so there is nothing to check.
\end{proof}

\subsection{More properties of direct sum and difference}
\begin{lemma}\label{sum_difference_inverses_surprise_surprise}
    For $U \subseteq V \in \Dip^X$ and $W \in \Dip^S$, we have canonical isomorphisms
    \[ (V \ominus U) \oplus U \cong V \textup{ in } \Dip^X, \]
    \[ (W \oplus U) \ominus U \cong W \textup{ in } \Dip^S. \]
\end{lemma}

\begin{proof}
    In each case the map $(V \ominus U) \oplus U \to V$ is given by taking the sum as subspaces of $V$. It is injective since $U \cap (V \ominus U) = 0$ as subspaces of $V$ (which is most difficult to verify in the hyperbolic case, where we use the fact that $U$ contains a vector $u$ with $q(u) < 0$ and therefore up to coordinate change it is $\R \oplus \R^k \oplus 0 \subseteq \R \oplus \R^k \oplus \R^\ell$). It is surjective by counting dimensions. Since the orthogonality condition in $\ominus$ guarantees that this map preserves the quadratic form, it is an isomorphism in $\Dip^X$. (It also preserves the map $x_0$ in the Euclidean case.)
    
    The second map $W \to (W \oplus U) \ominus U$ is given by inclusion into the first factor. Again it is injective because $W$ is always in the complement of $U$ in $W \oplus U$, and it is surjective by counting dimensions.
\end{proof}

It will also be helpful for us to characterize direct sum decompositions in $\Dip^X$ as sequences of nested subobjects:

\begin{lemma}\label{lem:sequence of oplus vs ominus}
    For $V \in \Dip^X$, consider direct sum decompositions
    \[
        \varphi \colon W_1 \oplus W_2 \oplus \dotsc \oplus W_n \cong V,
    \]
    up to isomorphism of each summand, where $W_1 \in \Dip^X$ and $W_i \in \Dip^S$ for $i \geq 2$. There is a bijective map $f$ from the set of isomorphism classes of such decompositions to nested subobjects in $\Dip^X$
    $$V_1 \subseteq V_2 \subseteq \dotsc \subseteq V_n = V$$
    given by $V_i =\varphi( W_1 \oplus W_2 \oplus \dotsc \oplus W_i)$, $W_1 = V_1$ and $W_i = V_i \ominus V_{i-1}$ for $i \geq 2$. Furthermore,
    \begin{enumerate}[(i)]
        \item along this bijection, grouping together subspaces $W_i \oplus W_{i+1}$ in a direct sum decomposition corresponds to omitting the subspace $V_i$ from the nesting, and adding a zero space $W_i = 0$ to the direct sum decomposition corresponds to repeating the subspace $V_i$ in the nesting.
        \item The inverse $f^{-1}$ produces a canonical representative of the direct sum decomposition in which all the terms are subspaces of $V$ which are orthogonal to each other, $\oplus$ is given by the sum of vector subspaces of $V$, and the isomorphism is equality.
    \end{enumerate}
\end{lemma}

In part (ii), we say that $U_1,U_2 \subseteq V$ are \emph{orthogonal} if $b(u_1,u_2) = 0$ for all $u_1 \in U_1$ and $u_2 \in U_2$ (that are in the kernel of $x_0$ in the Euclidean case).

\begin{proof}
    The map $f$ is surjective because every set of nested subspaces is isomorphic to a sequence of direct sums, by repeated application of the first isomorphism in \cref{sum_difference_inverses_surprise_surprise}. It is injective because any two direct sum decompositions can be identified by repeated application of the second isomorphism in \cref{sum_difference_inverses_surprise_surprise}.

    Properties (i) and (ii) follow from the definition of the bijection $f$. For example, $V_i \ominus V_{i-1}$ is orthogonal to $V_{i-1}$ by definition. Furthermore, if $U, W \subseteq V$ are orthogonal, then the direct sum $U \oplus W$ can be taken to be their sum as vector spaces, and $(U \oplus W) \ominus W = U$.
\end{proof}

Once we use this bijection, we often find ourselves in a situation where we are summing together subspaces of a fixed vector space $V$, and so we can ask for an equality and not just an isomorphism with a direct sum, as in the following lemma.

\begin{lemma}\label{lem:oplus_ominus}
    Let $V \in \Dip^X$. For any sequence of nested subobjects $U_1 \subseteq U_2 \subseteq U_3 \subseteq V$ of $V$, we have
    $$(U_2 \ominus U_1) \oplus (U_3 \ominus U_2) = U_3 \ominus U_1,$$
    $$(U_3 \ominus U_1) \ominus (U_2 \ominus U_1) = U_3 \ominus U_2$$
\end{lemma}
    
\begin{proof}
    For the first equality, the left-hand side is certainly contained in $U_3 \ominus U_1$, since both factors are orthogonal to $U_1$. The dimensions of both sides are equal, and therefore they are equal. The second equality is deduced from the first equality by applying $\ominus (U_2 \ominus U_1)$ to both sides. 
\end{proof}

\section{The classical Dehn invariant and product structures}\label{sec:classical_structures}

We next recall the classical description of the product and the Dehn invariant in each of these geometries, before we define a more sophisticated lift of these structures to spectra.

\subsection{The polytope module}

\begin{definition}\label{def-convex-polytope}
    Let $X$ be a geometry of dimension $n$. A geodesic simplex in $X$ is the geodesic convex hull of $n+1$ points in general position (contained in an open hemisphere in the spherical case). A convex polytope in $X$ is a subset $P \subseteq X$ obtained by intersecting $X$ with finitely many half-spaces in $\R^{n+1}$.
\end{definition}

\begin{definition}\label{pt}
    For any geometry $X$ of dimension $n$, the \emph{polytope module} $\Pt(X)$ is the abelian group generated by elements $[P]$ for each nondegenerate geodesic simplex $P = [v_0,\ldots,v_n]$ in $X$, modulo the relation that when one simplex $P$ is covered by finitely many simplices $P_i$ with disjoint interiors, we have $[P] = \sum_i [P_i]$, as in \cite[Definition 5.1]{KKMMW1}.
    
    As in \cite{KLMMS-2}, the group $\Pt(X)$ can alternatively be defined as having one generator $[P]$ for every convex polytope $P$ (contained in an open hemisphere in the spherical case), modulo the relation that when one convex polytope $P$ is covered by finitely many convex polytopes $P_i$ with disjoint interiors, we have $[P] = \sum_i [P_i]$. Every polytope $P$ from this second presentation gives a well-defined element in the first presentation by subdividing $P$ into simplices.

    We also call this group $\Pt(V)$ if $V$ is the ambient vector space. For the empty spherical geometry $S^{-1} = \varnothing$, we define $\Pt(\R^0)$ to have a single generator corresponding to the empty simplex.

    As in \cite[Definition 5.1]{KKMMW1}, we can also write $[v_0,\ldots,v_n] \in \Pt(V)$ if the vectors $v_i$ are any vectors in the ambient vector space $V$ of dimension $(n+1)$, not just lying along the geometry $X^n$ inside.

    Also as in \cite[Definition 5.1]{KKMMW1}, in the spherical case we also define the \emph{reduced polytope module} $\tPt(V) = \Pt(V)/\Sigma$ as the quotient of $\Pt(V)$ by the image of ``suspension'' maps. Here we define a suspension map $\Pt(U) \to \Pt(V)$ for each codimension-one subspace $U \subseteq V$ with orthogonal complement spanned by $v \in S(V)$, by taking each polytope in $S(U)$ to its join with the vectors $v$ and $-v$. On the level of geodesic simplices, it takes $[v_0,\ldots,v_{n-1}]$ to the sum
	\begin{equation}\label{pt_relation}
		[v_0,\ldots,v_{n-1},v] + [v_0,\ldots,v_{n-1},-v].
	\end{equation}
\end{definition}

The order of the points $[v_0,\ldots,v_n]$ does not affect the resulting element of $\tPt(V)$, so rearranging them does not introduce a sign. The isometry group $I(X)$ acts on $\Pt(X)$ by acting on the vertices of each simplex. Similarly, $O(n)$ acts on $\tPt(S^{n-1})$ for all $n$. This does not introduce a sign, even when the matrix reverses orientation. This makes the group $\Pt(-)$ into a functor $\Dip^X\to\Ab$, which we call the \emph{(unreduced) polytope functor}.

\begin{definition}\label{homology_sah_module}
Taking the homotopy colimit of these as functors $\Dip^X \to \Ch$, or equivalently the group homology of the automorphism groups $I(X)$ with coefficients in these modules, gives graded abelian groups that we call
\begin{align*}
H_*(\mathcal{E}) := \bigoplus_{n \geq 0} H_*(I(E^n);\Pt(E^n)), & \qquad H_*(\mathcal{S}) := \bigoplus_{n \geq 0} H_*(O(n);\Pt(S^{n-1})), \\
H_*(\mathcal{H}) := \bigoplus_{n \geq 0} H_*(I(H^n);\Pt(H^n)), & \qquad H_*(\widetilde{\mathcal{S}}) := \bigoplus_{n \geq 0} H_*(O(n);\tPt(S^{n-1})).
\end{align*}
\end{definition}

In particular, at homology level 0 we get the classical scissors congruence groups
\begin{align*}
H_0(\mathcal{E}) := \bigoplus_{n \geq 0} \Pt(E^n)_{I(E^n)}, & \qquad H_0(\mathcal{S}) := \bigoplus_{n \geq 0} \Pt(S^{n-1})_{O(n)}, \\
H_0(\mathcal{H}) := \bigoplus_{n \geq 0} \Pt(H^n)_{I(H^n)}, & \qquad H_0(\widetilde{\mathcal{S}}) := \bigoplus_{n \geq 0} \tPt(S^{n-1})_{O(n)}.
\end{align*}
These groups are bigraded, with the first grading given by homological degree and the second by the summation index. In spherical geometry, this second grading is the dimension of the ambient vector space.

\subsection{Definition of the Dehn invariant and products}
Next we recall the classical Dehn invariant. Its formula looks just like the coproduct in \cite{KKMMW1}, only the meaning of some of the symbols has changed.

\begin{definition}\label{pt_dehn}
    If $X = E^n$ or $X = H^n$ with ambient vector space $V$ of dimension $(n+1)$, we define the Dehn invariant
	\[
    	\delta\colon \Pt(V) 
			\to 
		\bigoplus_{0 \subsetneq U \subseteq V} 
			\Pt(U) \otimes \tPt(V \ominus U)
	\]
	by taking each simplex $[v_0,\ldots,v_n]$ to the sum
	\[ 
		\sum_{\varnothing \subsetneq S\subseteq \{v_0,\dots, v_n\}} 
			[S] \otimes [\pr_{V \ominus \langle S\rangle} S^c]. 
	\]
    Here $U$ ranges over the \emph{nonzero} subobjects of $V$ in the sense of \cref{dip-subobject}. We also have $S^c = \{v_0,\ldots,v_n\} \setminus S$, and $\pr_{V \ominus \langle S\rangle} S^c$ means projection onto the factor $V \ominus \langle S\rangle \in \Dip^S$ along the canonical isomorphism $V \cong (V \ominus \langle S\rangle) \oplus \langle S\rangle$ from \cref{sum_difference_inverses_surprise_surprise}. As in the previous paper \cite{KKMMW1}, we allow the vectors in these expressions to be scaled by any positive scalar, so that they may again lie on the unit sphere.
    
    If $X = S(V) \cong S^{n-1}$ then the Dehn invariant has the same definition, except that we also allow $U$ to be the zero subspace:
	\[
    	\delta\colon \Pt(V) 
			\to 
		\bigoplus_{0 \subseteq U \subseteq V} 
			\Pt(U) \otimes \tPt(V \ominus U), \,
		[v_1,\ldots,v_n] 
			\mapsto 
		\sum_{\varnothing \subseteq S\subseteq \{v_1,\dots, v_n\}} 
			[S] \otimes [\pr_{V \ominus \langle S\rangle} S^c]. 
	\]
    This turns out to send the image of suspension to zero, and so it descends to a Dehn invariant in the reduced case as well, see \cite{KKMMW1}:
	\[
    	\delta\colon \tPt(V) 
			\to 
		\bigoplus_{0 \subseteq U \subseteq V} 
			\tPt(U) \otimes \tPt(V \ominus U), \,
		[v_1,\ldots,v_n] 
			\mapsto 
		\sum_{\varnothing \subseteq S\subseteq \{v_1,\dots, v_n\}} 
			[S] \otimes [\pr_{V \ominus \langle S\rangle} S^c]. 
	\]
    Geometrically, in each of these formulas, each term of this sum is one face of the simplex, and its ``link'' as a simplex in the complementary subspace of $V$. In the Euclidean and hyperbolic cases the empty face of $[v_0,\ldots,v_n]$ is not considered because it does not have a well-defined link due to $0$ not being an admissible subobject, whereas in the spherical case the empty face of $[v_1,\ldots,v_n]$ is allowed, and its link is the entire polytope $[v_1,\ldots,v_n]$.
\end{definition}

\begin{remark}\label{why_no_zero}
    Note that the Dehn invariant does not allow $U = 0$ in the Euclidean and hyperbolic cases but does allow $U = 0$ in the spherical case. This is because a Euclidean geometry must have an ambient vector space that is positive-dimensional by \cref{coordinate-free-dip}, and so the empty subset of $E^n$ does not have a span that is an object of $\Dip^E$, and similarly for hyperbolic geometry.
\end{remark}

We next define the multiplication on $E^n$ and the multiplication on $S^{n-1}$.
\begin{definition}\label{e_product}
    For two Euclidean geometries with ambient vector spaces $U$ and $V$ we define the multiplication
    \[ \mu\colon \Pt(U) \otimes \Pt(V)
        \to
    \Pt(U \times_\R V) \]
    by taking a pair of simplices $P \subseteq E^m$, $Q \subseteq E^n$ to their Cartesian product $P \times Q \subseteq E^m \times E^n$. In terms of simplices this can be written as 
    \[ \mu\colon [u_0,\ldots,u_m], [v_0,\ldots,v_n] \mapsto \sum_{\textup{paths}} [(u_{a_i},v_{b_i})_i] \]
    where $(a_i,b_i)$ ranges over lattice paths from $(0,0)$ to $(m,n)$, which at each step increase either the first or second coordinate by 1. 
\end{definition}

\begin{remark}
    Note the lack of signs in the formula in \cref{e_product}. This is because elements of $\Pt(U)$ are insensitive to the orientation induced on the simplex by the order of the vertices. There are other equivalent ways to define the group (as in \cite[Definition 5.3]{KKMMW1}) under which such signs would appear.
\end{remark}

\begin{definition}\label{s_product}
    For two spherical geometries with ambient vector spaces $U$ and $V$ we define the join product
    \[ \mu\colon \Pt(U) \otimes \Pt(V)
        \to
    \Pt(U \oplus V) \]
    by taking each pair of simplices $P \subseteq S^{m-1}$, $Q \subseteq S^{n-1}$ to their join $P * Q \subseteq S^{m+n-1}$. In terms of simplices this can be written very simply as
    \[ \mu\colon [u_1,\ldots,u_m], [v_1,\ldots,v_n] \mapsto [u_1,\ldots,u_m,v_1,\ldots,v_n]. \]
    Again this descends to a product in the reduced case as well, see \cite{KKMMW1}:
    \[ \mu\colon \tPt(U) \otimes \tPt(V)
        \to
    \tPt(U \oplus V) \]
\end{definition}

\begin{theorem}\label{mainthm_simple}
    After passing to field coefficients and homotopy colimits, or equivalently after taking group homology with field coefficients, the above products and coproducts make
\begin{itemize}
    \item $H_*(\widetilde{\mathcal{S}};k)$ into a Hopf algebra,
    \item $H_*(\mathcal{E};k) =  \bigoplus_{n \geq 0} H_*(I(E^n);\Pt(E^n) \otimes k)$ into a comodule algebra over $H_*(\widetilde{\mathcal{S}};k)$,
    \item $H_*(\mathcal{H};k) =  \bigoplus_{n \geq 0} H_*(I(H^n);\Pt(H^n) \otimes k)$ into a comodule over $H_*(\widetilde{\mathcal{S}};k)$, and
    \item $H_*(\mathcal{S};k) =  \bigoplus_{n \geq 0} H_*(O(n);\Pt(S^{n-1}) \otimes k)$ into a comodule algebra over $H_*(\widetilde{\mathcal{S}};k)$
\end{itemize}
in the category of graded abelian groups. 
\end{theorem}

These homology statements from \cref{mainthm_simple} can be checked directly, see \cite{cathelineau_sc}. Our goal in this paper is go give a more structured and more topological proof of this, lifting the theorem to the $\infty$-category of spectra. This gives products and coproducts with more structure than the above ones on homology. 

As proven in \cite{scissors_thom}, the rational homotopy groups are the same as the rational homology groups for all of these spectra. Therefore,
\begin{equation}
	\bigoplus_{m \geq 0} \pi_m ^\Q(\widetilde{\Sah}) 
		\cong 
	\bigoplus_{m,n \geq 0} \ 
		\widetilde K_m(\Pol{S^{n-1}}{O(n)}) \otimes \Q
\end{equation}
form a bigraded commutative Hopf algebra. Similarly, for spherical and Euclidean geometry, the rational homotopy groups
\begin{equation}
	\bigoplus_{m \geq 0} \pi_m ^\Q(\Sah) 
		\cong 
	\bigoplus_{m,n \geq 0} \ 
		 K_m(\Pol{S^{n-1}}{O(n)}) \otimes \Q,
\end{equation}
and
\begin{equation}
	\bigoplus_{m \geq 0} \pi_m ^\Q(\mathcal{E}) 
		\cong 
	\bigoplus_{m,n \geq 0} \ 
		 K_m(\Pol{E^n}{I(E^n)}) \otimes \Q,
\end{equation}
form bigraded comodule algebras over $\bigoplus_{m \geq 0} \pi_m ^\Q(\widetilde{\Sah})$. For hyperbolic geometry, the rational homotopy groups
\begin{equation}
	\bigoplus_{m \geq 0} \pi_m ^\Q(\mathcal{H}) 
		\cong 
	\bigoplus_{m,n \geq 0} \ 
		 K_m(\Pol{H^n}{I(H^n)}) \otimes \Q,
\end{equation}
form a bigraded comodule over $\bigoplus_{m \geq 0} \pi_m ^\Q(\widetilde{\Sah})$.

Concretely, set
$$A_{m,n} = \widetilde{K}_m(\Pol{S^{n-1}} {O(n)}) \otimes \Q$$ and denote by $B_{m,n}$ one of 
$$K_m(\Pol{S^{n-1}}{O(n)}) \otimes \Q, \ K_m(\Pol{E^n}{I(E^n)}) \otimes \Q, \ \textup{ or } \ K_m(\Pol{H^n}{I(H^n)}) \otimes \Q.$$ Then the higher Dehn invariants have the form
$$\delta \colon B_{m,n} \rightarrow \bigoplus_{\substack{i+j=m\\k+l=n}} B_{i,k} \otimes_{\mathbb{Q}} A_{j,l}$$
and in the Euclidean and spherical cases, the Cartesian and join products give multiplication maps
$$\mu \colon B_{i,k}\otimes_{\mathbb{Q}} B_{j,l} \rightarrow B_{i+j,k+l}.$$

\begin{example}
    The original Dehn invariant is the single piece of the $H_*(\widetilde{\mathcal{S}})$-comodule structure on $H_*(\mathcal{E})$ where $n = 3$, $k = 1$, $l = 2$, and $i=j=m=0$:
    \begin{align*}
        H_0(I(E^3);\Pt(E^3)) &\to H_0(I(E^1);\Pt(E^1)) \otimes H_0(O(2);\tPt(S^{2-1})) \\
        K_0(\Pol{E^3}{I(E^3)}) &\to K_0(\Pol{E^1}{I(E^1)}) \otimes \widetilde{K_0}(\Pol{S^1}{O(2)})\cong \R \otimes_\Z \R/\pi\Z.
    \end{align*}
    Geometrically, it takes each 3-dimensional polyhedron $P$ to the sum over all the edges of $\ell \otimes \alpha$, where $\ell \in \R$ is the length of the edge and $\alpha \in \R/\pi\Z$ is the dihedral angle.
\end{example}

\section{The suspended and polytopal Tits complexes}

In this section we define the topological spaces $\ST(X)$ and $\PT(X)$ that appear when computing the homotopy type of scissors congruence $K$-theory, and whose homology agrees with the groups $\tPt(S^{n-1})$ and $\Pt(X)$ (see \cref{sec:apartments}). Since we prefer to think in terms of vector spaces $V$ instead of geometries $X$, we often call these spaces $\ST(V)$ and $\PT(V)$ instead.

\subsection{The Tits complex and its suspensions}
Let $X$ be a geometry with ambient vector space $V \cong \R^n$ or $V \cong \R \oplus \R^n$ equipped with the respective structure from \cref{coordinate-free-dip}.

\begin{definition}\label{df:tits_and_suspensions}
The \emph{Tits complex} $T(X)$ is the realization of the poset of all proper nonempty geometric subspaces of $X$, or equivalently the proper nonzero subobjects of $V$ in the sense of \cref{dip-subobject}. 
\end{definition}
As in \cite[Section 2]{KKMMW1}, we can suspend $T(X)$ once or twice and get quotients of realizations of posets
\begin{align*}
	\T(X) 
		& = 
	| \varnothing \subsetneq Y \subsetneq X |, \\
	\ST(X) 
		& = 
	\frac{| \varnothing \subsetneq Y \subseteq X |}%
		{| \varnothing \subsetneq Y \subsetneq X |}, \\
	\Sigma\!\ST(X) 
		& = 
	\frac{| \varnothing \subseteq Y \subseteq X |}%
		{| \varnothing \subseteq Y \subsetneq X | \cup | \varnothing \subsetneq Y \subseteq X |}. 
\end{align*}
We can also think of $\ST$ as a quotient of homotopy colimits
\[ \ST(X) = \frac{ \underset{\varnothing \subsetneq Y \subseteq X}\hocolim\ {*} }{ \underset{\varnothing \subsetneq Y \subsetneq X}\hocolim\ {*} }
=
\frac{\left( \displaystyle\coprod_{\varnothing \subsetneq Y_0 \subsetneq Y_1 \subsetneq \cdots \subsetneq Y_p \subseteq X} {*} \times \Delta^p \right) / \sim}{\left( \displaystyle\coprod_{\varnothing \subsetneq Y_0 \subsetneq Y_1 \subsetneq \cdots \subsetneq Y_p \subsetneq X} {*} \times \Delta^p \right) / \sim}. \]
Here the homotopy colimit is the Bousfield--Kan model of the homotopy colimit in unbased spaces, which we have written out explicitly on the right-hand side. The equivalence relation deletes the $i$th subspace $Y_i$ whenever the point $t \in \Delta^p$ is in the $i$th face $\Delta^{p-1} \subseteq \Delta^p$. Quotienting out by the homotopy colimit on the bottom imposes the additional relation that $\Delta^p$ is collapsed to the basepoint whenever $Y_p \neq X$.

It's convenient to also think of these in terms of the ambient vector space $V$, giving the isomorphic spaces
\begin{align*}
	\T(V) 
		= 
	| 0 \subsetneq U \subsetneq V |, & \quad \textup{$U$ a subobject of $V$}, \\
	\ST(V) 
		= 
	\frac{| 0 \subsetneq U \subseteq V |}%
		{| 0 \subsetneq U \subsetneq V |}, & \quad \textup{$U$ a subobject of $V$}, \\
	\Sigma\!\ST(V) 
		= 
	\frac{| 0 \subseteq U \subseteq V |}%
		{| 0 \subseteq U \subsetneq V | \cup | 0 \subsetneq U \subseteq V |}, & \quad \textup{$U$ either a subobject of $V$ or 0}.
\end{align*}
Note that in the spherical case $S^{n-1}$, the zero space is already a subobject of $\R^n$, but in the Euclidean and hyperbolic cases we have to explicitly add it back in. 

\subsection{The polytopal Tits complex and its suspension}
\begin{definition}
By analogy, and following \cite{scissors_thom}, when $X = E^n$ or $H^n$ with $n \geq 0$, or $S^{n-1}$ with $n \geq 1$, we define the \emph{polytopal Tits complex} as the quotient of homotopy colimits
\[ \PT(X) = \frac{ \underset{\varnothing \subsetneq Y \subseteq X}\hocolim\ Y }{ \underset{\varnothing \subsetneq Y \subsetneq X}\hocolim\ Y }
=
\frac{\left( \displaystyle\coprod_{\varnothing \subsetneq Y_0 \subsetneq Y_1 \subsetneq \cdots \subsetneq Y_p \subseteq X} Y_0 \times \Delta^p \right) / \sim}{\left( \displaystyle\coprod_{\varnothing \subsetneq Y_0 \subsetneq Y_1 \subsetneq \cdots \subsetneq Y_p \subsetneq X} Y_0 \times \Delta^p \right) / \sim} \]
Again the homotopy colimit is taken in the category of unbased spaces, which we have written out explicitly. The spaces $Y_0$ are given the subspace topology coming from $\R \oplus \R^n$ or $\R^n$. The equivalence relation again deletes the $i$th subspace $Y_i$ whenever the point $t \in \Delta^p$ is in the $i$th face $\Delta^{p-1} \subseteq \Delta^p$, and when $i = 0$ it also applies the inclusion map $Y_0 \to Y_1$. Quotienting out by the homotopy colimit on the bottom means that we also quotient $Y_0 \times \Delta^p$ to the basepoint whenever $Y_p \neq X$.

We also call this $\PT(V)$ when $V$ is the ambient vector space. But it is important that we actually take the homotopy colimit of the geometric subspaces $Y = U \cap X$ and not the vector spaces $U$, at least in the spherical case where $U$ is contractible and $Y$ is a sphere.
\end{definition}

\begin{lemma}\label{spt_simplify_euclidean}
In the Euclidean and hyperbolic cases, we have $\PT(X) \simeq \ST(X)$ and therefore
\[ \Sigma\!\PT(X) \simeq \Sigma\!\ST(X) = \frac{|\varnothing \subseteq Y \subseteq X|}{|\varnothing \subseteq Y \subsetneq X| \cup |\varnothing \subsetneq Y \subseteq X|}. \]
\end{lemma}

\begin{proof}
    This follows since every nonempty geometric subspace is contractible, and weak equivalences of diagrams induce weak equivalences of homotopy colimits and total homotopy cofibers.
\end{proof}

\begin{lemma}\label{spt_simplify_spherical}
In the spherical case, when $n \geq 1$, we have a homotopy equivalence
\[ \Sigma\!\PT(S^{n-1}) = \Sigma\!\PT(\R^n) \simeq \frac{ \underset{0 \subseteq U \subseteq \R^n}\bhocolim\, S^U }{ \underset{0 \subseteq U \subsetneq \R^n}\bhocolim\, S^U }
=
\frac{\left( \displaystyle\coprod_{0 \subseteq U_0 \subsetneq U_1 \subsetneq \cdots \subsetneq U_p \subseteq \R^n} S^{U_0} \times \Delta^p \right) / \sim}{\left( \displaystyle\coprod_{0 \subseteq U_0 \subsetneq U_1 \subsetneq \cdots \subsetneq U_p \subsetneq \R^n} S^{U_0} \times \Delta^p \right) / \sim}. \]
\end{lemma}
Here the homotopy colimit is along the diagram of spheres $S^U \cong D(U)/S(U)$, which are based spaces with basepoint $\{\infty\} = S(U)/S(U)$. We call this basepoint $\infty$ because you can also think of $S^U$ as the one-point compactification of $U$. These homotopy colimits are taken in the category of based spaces, so the equivalence relation sends any pair of the form $(\infty,t) \in S^{U_0} \times \Delta^p$ to the basepoint $*$, in addition to deleting the $i$th subspace $U_i$ whenever the point $t \in \Delta^p$ is in the $i$th face $\Delta^{p-1} \subseteq \Delta^p$, and when $i = 0$ also applying the inclusion map $S^{U_0} \to S^{U_1}$. Quotienting out by the homotopy colimit on the bottom, we also quotient $S^{U_0} \times \Delta^p$ to the basepoint whenever $U_p \neq \R^n$.

\begin{proof}
    Since geometric subspaces of $S^{n-1}$ correspond to linear subspaces $U \subseteq \R^n$, we can rewrite $\PT(\R^n)$ as
\[ \PT(\R^n) = \frac{ \underset{0 \subsetneq U \subseteq \R^n}\hocolim\  S(U) }{ \underset{0 \subsetneq U \subsetneq \R^n}\hocolim\  S(U) }. \]
Then we allow $U = 0$ and observe that this does not change the homotopy colimit since $S(0) = \varnothing$:
\[ \PT(\R^n) \cong \frac{ \underset{0 \subseteq U \subseteq \R^n}\hocolim\  S(U) }{ \underset{0 \subseteq U \subsetneq \R^n}\hocolim\  S(U) }. \]
So now we just have to prove that whenever $n > 0$,
\begin{equation}\label{sigmapt_proof_step}
\Sigma \left( \frac{ \underset{0 \subseteq U \subseteq \R^n}\uhocolim\, S(U) }{ \underset{0 \subseteq U \subsetneq \R^n}\uhocolim\, S(U) } \right)
\simeq \frac{ \underset{0 \subseteq U \subseteq \R^n}\bhocolim\, S^U }{ \underset{0 \subseteq U \subsetneq \R^n}\bhocolim\, S^U }.
\end{equation}

		The right-hand side of \eqref{sigmapt_proof_step} has a contractible subcomplex built out of the points $0 \in S^U$ (which are not the basepoint, the basepoint is $\infty \in S^U$):
		\[ \frac{ \underset{0 \subseteq U \subseteq \R^n}\hocolim\  \{0\}  }{ \underset{0 \subseteq U \subsetneq \R^n}\hocolim\  \{0\} }
		\subseteq \frac{ \underset{0 \subseteq U \subseteq \R^n}\bhocolim\, S^U }{ \underset{0 \subseteq U \subsetneq \R^n}\bhocolim\, S^U }. \]
		This subcomplex is contractible if $n > 0$ because the poset of all $0 \subseteq U \subsetneq \R^n$ has an initial object. If we quotient the right-hand side of \eqref{sigmapt_proof_step} by this contractible subcomplex, we get
        \[ \frac{\coprod_{U_0 \subsetneq \cdots \subsetneq U_k} S^{U_0} \times \Delta^k}{\textrm{face relations, }\{0,\infty\} \times \textrm{anything, }U_k \neq \R^n}. \]
        On the other hand, using the fact that $\Sigma \frac{X}{A} \cong \frac{X \times I}{(X \times \partial I) \cup (A \times I)}$, we write the left-hand side of \eqref{sigmapt_proof_step} as
        \[ \frac{\coprod_{U_0 \subsetneq \cdots \subsetneq U_k} I \times S(U_0) \times \Delta^k}{\textrm{face relations, }\partial I \times \textrm{anything, }U_k \neq \R^n}. \]
        Using the fact that $\frac{S(U) \times I}{S(U) \times \partial I} \cong \frac{S^U}{\{0,\infty\}}$, we can identify these two quotients together.
        \end{proof}

As in \cite{KKMMW1}, these definitions are all made for the geometries $E^n$ and $H^n$ when $n \geq 0$ and $S^{n-1}$ when $n \geq 1$, but we can take the result of \cref{spt_simplify_spherical} as the \emph{definition} of $\Sigma\!\PT(\R^0) = \Sigma\!\PT(S^{-1})$, which just gives the space $S^0$. More generally, we could define $\Sigma\!\PT$ of any empty geometry to be $S^0$. (However, in this paper we won't consider empty Euclidean or hyperbolic geometries, so this will only be relevant for the empty spherical geometry $\Sigma\!\PT(\R^0) = S^0$.)

\subsection{Configuration space models for $\ST(V)$, $\Sigma\!\ST(V)$, $\PT(V)$ and $\Sigma\!\PT(V)$} As in the article \cite{KKMMW1}, we need to describe points in these spaces in an explicit way.

\begin{definition}
    Let $V$ be the ambient vector space of a (spherical, hyperbolic or Euclidean) geometry. This is a vector space with additional structure, as in \cref{coordinate-free-dip}.
    \begin{itemize}
        \item Let $\Sub(V)$ be the poset of vector spaces $U \subseteq V$ that are either subobjects in the sense of \cref{dip-subobject} or zero.
        \item Let $\Step([0,1],\Sub(V))$ be the set of order-preserving functions $f\colon [0,1] \to \Sub(V)$. They are ``step functions'' because such a function $f$ must take finitely many values $U_0 \subsetneq \cdots \subsetneq U_p$, where $f^{-1}(U_i)$ is an interval in $[0,1]$.
        \item For $f\in \Step([0,1],\Sub(V))$, we define its \emph{cut points} to be the points $x\in [0,1]$ such that $f(t)$ for $t<x$ is different from $f(t)$ for $t>x$. Note that the number of cut points is always finite.
    \end{itemize}
\end{definition}

The following has the same proof as in the $X = S^n$ case in \cite[Proposition 2.12]{KKMMW1}:
\begin{proposition}\label{prop-SST-conf}
    Let $X = E^n, H^n$, or $S^{n-1}$ with $n \geq 0$. As a set, we can identify $\Sigma\!\ST(V)$ with a quotient
    \[ \Step([0,1],\Sub(V))/\sim. \]
   The equivalence relation $\sim$ identifies two functions $f$ and $g$ if they differ on a finite set of points, and identifies a step function $f$ with the basepoint if $f(t)\neq 0$ for all $t>0$, or if $f(t)\neq V$ for all $t<1$.
\end{proposition}
Informally, the first part of the relation $\sim$ ensures that only the location of the cut points matters, not which space $U_i$ is assigned to each cut point. The second part enforces that the first space in the list must be $U_0 = 0$, and the last space must be $U_p = V$, otherwise it is collapsed to the basepoint.

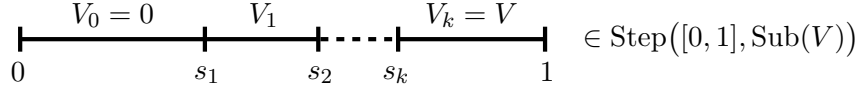
\begin{figure}[h]
\begin{center}
	\begin{tikzpicture}
		\node at (9.3,0) {$ \in \Step\big([0,1],\Sub(V)\big)$};
		\draw[ultra thick,|-|] 
			(0,0) node[below=1.5mm]{$0$} 
				--node[above]{$V_0=  0$} 
			(2.5,0);
        \draw[ultra thick,-|] 
        	(2.5,0) node[below=2mm]{$s_1$} 
				--node[above]{$V_1$} 
			(4,0) node[below=2mm]{$s_2$};
		\draw[ultra thick,dashed]
			(4,0) -- (5,0);
		\draw[ultra thick,|-|]
			(5,0) node[below=2mm]{$s_k$}
				--node[above]{$V_k = V$} 
			(7,0) node[below=1.5mm]{$1$};
	\end{tikzpicture}
\end{center}
\caption{A schematic illustration of a step function representing a non-basepoint element of $\Sigma\!\ST(V)$.} 
\end{figure}

\begin{proposition}\label{prop-ST-conf}
    Let $X = E^n, H^n$, with $n \geq 0$, or $S^{n-1}$ with $n \geq 1$. As a set, we can identify $\ST(V)$ with a quotient
    \[ \Step([0,1],\Sub(V)_{\neq 0})/\sim. \]
   where $f \sim g$ if $f$ and $g$ differ on a finite set of points, and $f$ is identified with the basepoint if $f(t)\neq V$ for all $t<1$.
\end{proposition}
Here $\Sub(V)_{\neq 0}$ is the poset of nonzero subobjects $0 \neq U \subseteq V$. The proof is also the same as in \cite[Proposition 2.12]{KKMMW1}.

\begin{figure}[h]
\begin{center}
	\begin{tikzpicture}
		\node at (9.5,0) {$ \in \Step\big([0,1],\Sub(V)_{\neq 0}\big)$};
		\draw[ultra thick,|-|] 
			(0,0) node[below=1.5mm]{$0$} 
				--node[above]{$V_0 \neq 0$} 
			(1.5,0);
        \draw[ultra thick,-|] 
        	(1.5,0) node[below=2mm]{$s_1$} 
				--node[above]{$V_1$} 
			(3.5,0) node[below=2mm]{$s_2$};
		\draw[ultra thick,dashed]
			(3.5,0) -- (4.5,0);
		\draw[ultra thick,|-|]
			(4.5,0) node[below=2mm]{$s_k$}
				--node[above]{$V_k = V$} 
			(7,0) node[below=1.5mm]{$1$};
	\end{tikzpicture}
\end{center}
\caption{A schematic illustration of a step function representing a non-basepoint element of $\ST(V)$.} 
\end{figure}
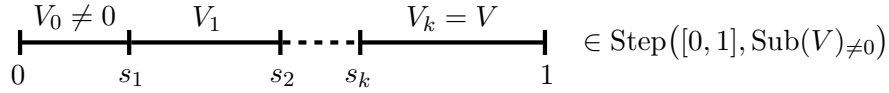

For $X = E,H$, $\PT(V)$ is equivalent to $\ST(V)$, so it is unnecessary to define a configuration space model for it. For $X = S$, to give a step function definition of $\PT(V)$, we need to introduce labels. Let $\Step^P([0,1],\Sub(V))$ denote the subspace of $\Step([0,1],\Sub(V)) \times S(V)$ consisting of $(f,v)$ such that $v \in S(f(0))$. (Necessarily, the step function also lands in $\Step([0,1],\Sub(V)_{\neq 0})$ because $S(0) = \varnothing$.)

\begin{proposition}\label{prop-PT-conf}
    Let $X = S^{n-1}$ with $n \geq 1$. As a set, we can identify $\PT(V)$ with 
    $$\Step^P([0,1],\Sub(V)) / \sim$$
    where $(f,v) \sim (g,v)$ if $f$ and $g$ differ on a finite number of points, and $(f,v)$ is identified with the basepoint if  $f(t)\neq V$ for all $t<1$.
\end{proposition}

\begin{figure}[h]
\begin{center}
	\begin{tikzpicture}
        \node at (-0.5,0) {$\big($};
        \node at (8.1,0) {$, v\in S(V_0) \big)$};
		\node at (11.2,0) {$ \in \Step^P\big([0,1],\Sub(V)\big)$};
		\draw[ultra thick,|-|] 
			(0,0) node[below=1.5mm]{$0$} 
				--node[above]{$ V_0 \neq 0$} 
			(2.5,0);
        \draw[ultra thick,-|] 
        	(2.5,0) node[below=2mm]{$s_1$} 
				--node[above]{$V_1$} 
			(4,0) node[below=2mm]{$s_2$};
		\draw[ultra thick,dashed]
			(4,0) -- (5,0);
		\draw[ultra thick,|-|]
			(5,0) node[below=2mm]{$s_k$}
				--node[above]{$V_k = V$} 
			(7,0) node[below=1.5mm]{$1$};
	\end{tikzpicture}
\end{center}
\caption{A schematic illustration of a step function representing a non-basepoint element of $\PT(V)$.} 
\end{figure}
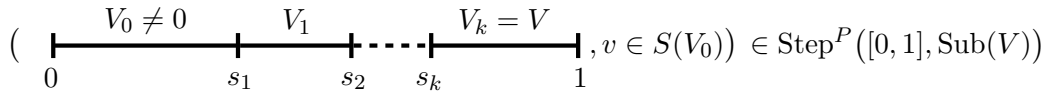

Similarly, let $\Step^{\Sigma P}([0,1],\Sub(V))$ denote the subspace of $\Step([0,1],\Sub(V)) \times S^V$ consisting of $(f,v)$ such that $v \in S^{f(0)}$. We think of $S^V$ as $D(V)/S(V)$.

\begin{proposition}\label{prop-SPT-conf}
    Let $X = S^{n-1}$ with $n \geq 0$. As a set, we can identify $\Sigma\!\PT(V)$ with 
    $$\Step^{\Sigma P} ([0,1],\Sub(V)) / \sim$$
    where $(f,v) \sim (g,v)$ if $f$ and $g$ differ on a finite number of points, and $(f,v)$ is identified with the basepoint if  $f(t)\neq V$ for all $t<1$ or if $v$ is the basepoint of $S^V$.
\end{proposition}

\begin{figure}[h]
\begin{center}
	\begin{tikzpicture}
        \node at (-0.5,0) {$\big($};
        \node at (8.5,0) {$, v\in S^{V_0}\setminus\{*\} \big)$};
		\node at (12,0) {$ \in \Step^{\Sigma P}\big([0,1],\Sub(V)\big)$};
		\draw[ultra thick,|-|] 
			(0,0) node[below=1.5mm]{$0$} 
				--node[above]{$ V_0$} 
			(2.5,0);
        \draw[ultra thick,-|] 
        	(2.5,0) node[below=2mm]{$s_1$} 
				--node[above]{$V_1$} 
			(4,0) node[below=2mm]{$s_2$};
		\draw[ultra thick,dashed]
			(4,0) -- (5,0);
		\draw[ultra thick,|-|]
			(5,0) node[below=2mm]{$s_k$}
				--node[above]{$V_k = V$} 
			(7,0) node[below=1.5mm]{$1$};
	\end{tikzpicture}
\end{center}
\caption{A schematic illustration of a step function representing a non-basepoint element of $\Sigma\!\PT(V)$.} 
\end{figure}
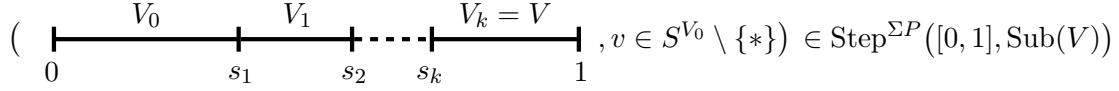

The proofs are the same as the proof of Proposition 2.12 in \cite{KKMMW1}, using the barycentric coordinates $(t_0,\ldots,t_p)$ to define a step function with cut points located at $s_0 = t_0$, $s_1 = t_0 + t_1$, $\ldots$, and $s_{p-1} = t_0 + \cdots + t_{p-1}$.

We can also visualize the proof of \cref{spt_simplify_spherical} in these models: if we take all configurations in $\Sigma\!\PT(V)$ and collapse to $*$ those configurations in which the point $v \in S^V$ is the center point 0, all that is left is given by a suspension coordinate (the radial coordinate in $D(V)$), and a point in $\PT(V)$.

\section{Scissors congruence spectra and scissors congruence functors}
\subsection{Scissors congruence spectra}
\begin{definition}\label{spectral_sah_module}
    For each type of geometry (Euclidean, hyperbolic, or spherical) we define the corresponding scissors congruence spectrum by the wedge sum
    \[ \mathcal{E} = \bigvee_{n \geq 0} (EI(E^n)_+ \wedge_{I(E^n)} (\Sph^{-(\R \oplus \R^n)} \wedge \Sigma\!\ST(E^n))), \]
    \[ \mathcal{H} = \bigvee_{n \geq 0} (EI(H^n)_+ \wedge_{I(H^n)} (\Sph^{-(\R \oplus \R^n)} \wedge \Sigma\!\ST(H^n))), \]
    \[ \mathcal{S} = \bigvee_{n \geq 0} (EI(S^{n-1})_+ \wedge_{I(S^{n-1})} (\Sph^{-\R^n} \wedge \Sigma\!\PT(\R^n))). \]
    \[ \widetilde{\mathcal{S}} = \bigvee_{n \geq 0} (EI(S^{n-1})_+ \wedge_{I(S^{n-1})} (\Sph^{-\R^n} \wedge \Sigma\!\ST(\R^n))). \]
    Here $\Sigma\!\ST$ and $\Sigma\!\PT$ are defined in \cref{df:tits_and_suspensions} and \cref{spt_simplify_spherical} respectively, and $EI(X^n)$ is the total space of the principal bundle $EI(X^n) \to BI(X^n)$ for the isometry group $I(X^n)$. In other words, $EI(X^n)$ is any contractible free $I(X^n)$-CW complex. Note that $\widetilde{\mathcal{S}}$ is the spectral Sah algebra from \cite{KKMMW1}, while $\mathcal{S}$ is a version of it for unreduced spherical scissors congruence. Each of these agrees with the direct sum of scissors congruence $K$-theory spectra of the appropriate geometry in all dimensions, by \cite{scissors_thom}:
    \[ \mathcal{E} \simeq \bigvee_{n \geq 0} K(\Pol{E^n}{I(E^n)}), \quad \mathcal{H} \simeq \bigvee_{n \geq 0} K(\Pol{H^n}{I(H^n)}), \quad \mathcal{S} \simeq \bigvee_{n \geq 0} K(\Pol{S^{n-1}}{O(n)}), \quad \widetilde{\mathcal{S}} \simeq \bigvee_{n \geq 0} \widetilde{K}(\Pol{S^{n-1}}{O(n)}), \]
\end{definition}

Taking homology with field coefficients recovers the graded abelian groups $H_*(\mathcal{E};k)$, $H_*(\mathcal{H};k)$, $H_*(\mathcal{S};k)$, and $H_*(\widetilde{\mathcal{S}};k)$ from \cref{homology_sah_module}.

Our goal is to prove that $\mathcal{H}$, $\mathcal{E}$, and $\mathcal{S}$ are $E_1$-comodules over $\widetilde{\mathcal{S}}$. As in the previous paper, we approach this by thinking of them as colimits of diagrams over the category $\Dip^X$ from \cref{coordinate-free-dip}.

\begin{definition}
    We define functors
    \begin{align*}
    \ST, \Sigma\!\ST, \Sigma\!\PT\colon \Dip^X & \to \topp
    \end{align*}
    by sending each object $V \in \Dip^X$ to the appropriate space from \cref{df:tits_and_suspensions} or \cref{spt_simplify_spherical}, and each isomorphism $V \cong V'$ preserving the relevant structure to the evident induced isomorphism on $\ST$, $\Sigma\!\PT$, or $\Sigma\!\ST$.

Recall the extended category $\Luna^X$ from \cref{def-Luna}. We define extended functors
    \begin{align*}
\Sigma\!\PT_S &: \Luna^S \to  \topp\\
\Sigma\!\ST_S &: \Luna^S \to \topp \\
  (\ST,*)_E &:\Luna^E \to \topp, \\
   (\Sigma\!\ST,*)_E &:\Luna^E \to \topp\\
(*, \Sigma\!\ST)_E &: \Luna^E \to \topp \\
(\ST,*)_H&:\Luna^H \to \topp \\
(\Sigma\!\ST,*)_H&:\Luna^H \to \topp \\
(*, \Sigma\!\ST)_H &: \Luna^H \to \topp
\end{align*}
as follows.
\begin{itemize}
    \item $\Sigma\!\PT_S$ is the same as $\Sigma\!\PT\colon \Luna^S = \Dip^S \to  \topp$.
    \item $\Sigma\!\ST_S$ is the same as $\Sigma\!\ST\colon \Luna^S = \Dip^S \to \topp $.
    \item $(\ST,*)_E$ is defined as $\ST$ on $\Dip^E$ and the trivial functor $*$ on $\Dip^S \setminus \{0\}$.
    \item $(\Sigma\!\ST,*)_E$ is defined as $\Sigma\!\ST$ on $\Dip^E$ and the trivial functor $*$ on $\Dip^S \setminus \{0\}$.
    \item $(*, \Sigma\!\ST)_E$ is the trivial functor $*$ on $\Dip^S$ and $\Sigma\!\ST$ on $\Dip^E \setminus \{(\R,\id_\R)\}$.
    \item $(\ST,*)_H$ is defined as $\ST$ on $\Dip^H$ and the trivial functor $*$ on $\Dip^S$.
    \item $(\Sigma\!\ST,*)_H$ is defined as $\Sigma\!\ST$ on $\Dip^H$ and the trivial functor $*$ on $\Dip^S$.
    \item $(*, \Sigma\!\ST)_H$ is the trivial functor $*$ on $\Dip^S$ and $\Sigma\!\ST$ on $\Dip^H$.
\end{itemize}
The last three functors are also defined to be the trivial functor $*$ on ${*} \in \Luna^H$. The idea behind this definition is that functors in the right-hand slot will play the role of a coalgebra, and functors in the left-hand slot will play the role of a comodule over that coalgebra.
\end{definition}

\subsection{Desuspensions and Sah functors.}

These functors are closely related to the spectra $\mathcal{E},\mathcal{H}, \mathcal{S}$ and $\widetilde{\mathcal{S}}$, as is clear from Definition \ref{spectral_sah_module}. To make the connection more precise, we use the following functor to desuspend them.

\begin{definition}\label{def-desusp-Dip}
    In the spherical case, let
    \[
    	D \colon \Fun(\Dip^S, \topp ) \to \Fun(\Dip^S, \Sp^O)
	\]
    be the functor that sends $G \colon \Dip^S \to \topp$ to the functor given by 
    \[
    	V \mapsto \Sph^{-V} \sma G(V)
	\]
    where $\Sph^{-V} = F_VS^0$ is the negative $V$-sphere from \cite[Definition 2.5]{KKMMW1}. In the Euclidean and hyperbolic cases, we define analogues of this functor as follows:
    \begin{align*}
    D \colon \Fun(\Dip^E, \topp ) \to \Fun(\Dip^E, \Sp^O) & \qquad
    V \mapsto F(S^V,\Sph) \sma G(V) \\
    \Sigma D \colon \Fun(\Dip^E, \topp ) \to \Fun(\Dip^E, \Sp^O) & \qquad
    V \mapsto \Sph^{-V_0} \sma G(V) \\
    D \colon \Fun(\Dip^H, \topp ) \to \Fun(\Dip^H, \Sp^O) & \qquad
    V \mapsto F(S^V,\Sph) \sma G(V) \\
    \Sigma D \colon \Fun(\Dip^H, \topp ) \to \Fun(\Dip^H, \Sp^O) & \qquad
    V \mapsto \Sigma F(S^V, \Sph) \wedge G(V).
    \end{align*}
    In each of these cases, $S^V$ is the one-point compactification of the vector space $V$,
   and $F(S^V,\Sph) = \Omega^V\Sph$ denotes the cotensor or Spanier--Whitehead dual of $S^V$ in orthogonal spectra. In the second of these functors, we replace the expected expression $\Sigma F(S^V,\Sph)$ with the equivalent expression $\Sph^{-V_0} = F_{V_0}S^0$, because it has better monoidal properties. Here $V_0 = \ker(x_0\colon V \to \R)$ and $\Sph^{-V_0} = F_{V_0}S^0$ is also the negative $V_0$-sphere from \cite[Definition 2.5]{KKMMW1}.

   In all of the resulting functors, the left action of the isometries $I(V)$ is by a diagonal action on $G(V)$ and on $\Sph^{-V}$, $\Sph^{-V_0}$, or $F(S^V,\Sph)$. In all three of these cases, a matrix $A$ acts by pre-composition with $A^{-1}$, as in \cite[Definition 2.5]{KKMMW1}.
\end{definition}

\begin{lemma}\label{switch_which_one_is_suspended}
    In the Euclidean and hyperbolic cases, we have natural equivalences $D(\Sigma G) \simeq \Sigma D(G)$.
\end{lemma}

\begin{proof}
    In the hyperbolic case this is clear from the definition. In the Euclidean case, this comes down to the claim that there is an $I(E^n)$-equivariant equivalence of spectra $\Sph^{-V_0} \simeq \Sigma F(S^V,\Sph)$ for each $V$. Smashing with $\Sph^{-1}$, this is equivalent to specifying an equivalence $\Sph^{-(\R \oplus V_0)} \simeq F(S^V,\Sph)$. The proof of \cite[Lemma 2.6]{KKMMW1} gives an equivariant equivalence $\Sph^{-(\R \oplus V_0)} \simeq F(S^{\R \oplus V_0},\Sph)$, so now we just have to give an equivalence $S^{\R \oplus V_0} \simeq S^V$ that commutes with the $I(E^n)$-action. We do this by picking an $O(n)$-equivariant identification $\R \oplus V_0 \cong V$, and then making it $I(E^n)$-equivariant by smashing with $E^n \subseteq V$, which is equivalent to a point $*$ with trivial action, and forming an equivariant homeomorphism as follows:
    \[ \begin{tikzcd}
        S^{\R \oplus V_0} & S^V \\
        E^n_+ \sma S^{\R \oplus V_0} \uar{\simeq} \rar[<->]{\cong} & E^n_+ \sma S^V \uar{\simeq}
    \end{tikzcd} \]
    The bottom map has the formula $(x,t,v) \mapsto (x,tx+v)$ for $x \in E^n \subseteq V$, $t \in \R$, and $v \in V_0$. This is invertible with inverse $(x,v) \mapsto (x,x_0(v),v-x_0(v)x)$, and can be directly checked to be $I(E^n)$-equivariant from the formula $A(tx + v) = tAx + Av$.
\end{proof}

The desuspension functors $D$ and $\Sigma D$ also extend to the corresponding Luna categories.

\begin{definition}\label{def-desusp-Luna}
   \begin{itemize}
       \item $D \colon \Fun(\Luna^S, \topp ) \to \Fun(\Luna^S, \Sp^O)$ is the same as the functor $D \colon \Fun(\Dip^S, \topp ) \to \Fun(\Dip^S, \Sp^O)$.
       \item $(\Sigma D,D)_E: \Fun(\Luna^E, \topp ) \to \Fun(\Luna^E, \Sp^O)$ is defined as $\Sigma D$ on the category $\Fun(\Dip^E, \topp)$ and as $D$ on $\Fun(\Dip^S, \topp)$.
       \item $(\Sigma D,D)_H:\Fun(\Luna^H, \topp ) \to \Fun(\Luna^H, \Sp^O)$ is defined as $\Sigma D$ on the category $\Fun(\Dip^H, \topp)$, as $D$ on $\Fun(\Dip^S, \topp)$, and as $*$ on $*$.
   \end{itemize} 
\end{definition}

The categories $\Fun(\Luna^X, \topp )$ and $\Fun(\Luna^X, \Sp^O)$ are symmetric monoidal categories with the Day convolution product, which we denote by $\boxwedge$ in all cases. To transfer our algebras, coalgebras, and comodules from $\Fun(\Luna^X, \topp )$ to $\Fun(\Luna^X, \Sp^O)$, we use the following proposition.

\begin{proposition}\label{prop-D-symm-mon-Dip}
    The functors 
    \[
    	D \colon \Fun(\Dip^S, \topp ) \to \Fun(\Dip^S, \Sp^O)
	\]
    \[
    	\Sigma D \colon \Fun(\Dip^E, \topp ) \to \Fun(\Dip^E, \Sp^O)
	\]
    are strong symmetric monoidal. 
\end{proposition}

\begin{proof}
    For the first functor this is \cite[Lemma 4.4]{KKMMW1}. For the second functor this follows from the same lemma and the natural isomorphism $(V \times_\R W)_0 \cong V_0 \oplus W_0$.
\end{proof}

\begin{corollary}\label{cor-D-symm-mon-Luna}
    The functors
    $$D \colon \Fun(\Luna^S, \topp ) \to \Fun(\Luna^S, \Sp^O)$$
    $$(\Sigma D,D)_E: \Fun(\Luna^E, \topp ) \to \Fun(\Luna^E, \Sp^O)$$
    are strong symmetric monoidal, and
    $$(\Sigma D,D)_H: \Fun(\Luna^H, \topp ) \to \Fun(\Luna^H, \Sp^O)$$
    is lax symmetric monoidal with structure maps $F(X) \boxwedge F(Y) \to F(X \boxwedge Y)$ that are stable equivalences of spectra whenever $X,Y \in \Dip^S \amalg \Dip^H$.
\end{corollary}

\begin{proof}
    When multiplying two spaces from the same copy of $\Dip$, this follows from \cref{prop-D-symm-mon-Dip}. For the isomorphism $\Sph^{-V_0} \sma \Sph^{-W} \cong \Sph^{-(V_0 \oplus W)}$ when $V \in \Dip^E$ and $W \in \Dip^S$ we use the same isomorphism, and the coherences follow from \cite[Lemma 4.4]{KKMMW1}. Note that the definition of $\Sigma D$ and $D$ agree on the unit object $(\R,\id_\R) \sim (0)$ because the first smashes with $\Sph^{-V_0} = \Sph^0$ and the second smashes with $\Sph^0$.
    
    On the object $* \in \Luna^H$, the lax symmetric monoidal structure exists trivially. For the remaining objects in the hyperbolic case, we construct a stable equivalence $\Sigma F(S^V,\Sph) \sma \Sph^{-W} \to \Sigma F(S^{V \oplus W},\Sph)$ using the suspension of the adjoint of the stable equivalence
    \[ F(S^V,\Sph) \sma \Sph^{-W} \sma S^{V \oplus W} \to \Sph \]
    that decomposes $S^{V \oplus W} \cong S^V \sma S^W$, applies the stable equivalence $\Sph^{-W} \sma S^W \to \Sph$ from \cite[Lemma 2.6]{KKMMW1}, and then applies the evaluation map $F(S^V,\Sph) \sma S^V \to \Sph$, which is also a stable equivalence. It is straightforward to check that applying this map with $W = 0$ induces the identity, and applying it with $W_1$ and then $W_2$ gives the same map as applying it with $W_1 \oplus W_2$. This finishes the verification that it defines a lax symmetric monoidal functor. It is a stable equivalence because the adjunction between $S^{V \oplus W} \sma (-)$ and $F(S^{V \oplus W},-)$ preserves stable equivalences.
\end{proof}

\begin{remark}
    We are unable to find a point-set model for $\Sigma D$ in the hyperbolic case that defines a strong symmetric monoidal functor $\Fun(\Luna^H, \topp) \to \Fun(\Luna^H, \Sp^O)$, and we suspect that one does not exist. However, since $(\Sigma D, D)_H$ is lax symmetric monoidal along weak equivalences, it will be strong symmetric monoidal when we pass to $\infty$-categories.
\end{remark}

\begin{definition}\label{def-spectral-sah}
    We use the following notation for certain desuspensions of our earlier diagrams on $\Luna^X$:
    \begin{align*}
    \funSah^E = (\Sigma D, D)_E((\ST, *)_E), & \qquad \funSah^{\widetilde S} = (\Sigma D, D)_E((*, \Sigma\!\ST)_E), \\
    \funSah^H = (\Sigma D, D)_H((\ST, *)_H), & \qquad \funSah^{\widetilde S} = (\Sigma D, D)_H((*, \Sigma\!\ST)_H), \\
    \funSah^S = D(\Sigma PT_S), & \qquad \funSah^{\widetilde S} = D(\Sigma\!\ST_S).
    \end{align*}
\end{definition}

This notation is somewhat ambiguous since $\funSah^{\widetilde S}$ has three distinct definitions. But they are essentially the same, as they are desuspensions by the same functor $D$ of the same diagram $\Sigma\!\ST$ on the same subcategory $\Dip^S$ of different categories $\Luna^X$. This will allow us to define different comodules over the same coalgebra.

We have defined these precisely so that their homotopy colimits are the scissors congruence spectra from \cref{spectral_sah_module}:
\begin{align*}
\hocolim \funSah^E = \mathcal{E}, & \qquad \hocolim \funSah^S = \mathcal{S}, \\
\hocolim \funSah^H = \mathcal{H}, & \qquad \hocolim \funSah^{\widetilde S} = \widetilde{\mathcal{S}}.
\end{align*}
Rewriting the spectra in this way will allow us to define their algebra, coalgebra, and comodule structures in the next section.

\begin{remark}
    By \cref{switch_which_one_is_suspended}, we also have two more alternative definitions for $\funSah^E$ and $\funSah^H$:
    \begin{align*}
    \funSah^E \simeq (D, D)_E((\Sigma\!\ST, *)_E), & \qquad \funSah^H \simeq (D, D)_H((\Sigma\!\ST, *)_H).
    \end{align*}
    In essence, this means that we can move the suspension from the $D$ term to the $\ST$ term if desired. This observation will be important when we use apartment maps to check that our comodule structures give the Dehn invariant on $\pi_0$.
\end{remark}

\section{Modules and comodules over an operad}

In this section we prepare for the construction of the higher Dehn invariants, by making precise what a comodule over a coalgebra over an operad should be. To do this carefully, we find it is easiest if we work in the setting of coordinate-free operads as in \cite[2.1.1.1]{HA}, so we recall the setup here first, for operads in topological spaces. 
\begin{definition}\label{df:operad_coord_free}
	A \emph{coordinate-free operad} is
	\begin{itemize}
		\item an assignment of a topological space $\mathcal O(I)$ to each finite set $I$,
		\item an identity element $\id_i \in \mathcal O(\{i\})$ for each singleton $\{i\}$, and
		\item composition maps
		\[ \circ\colon \mathcal O(J) \times \prod_{j \in J} \mathcal O(\alpha^{-1}(j)) \longrightarrow \mathcal O(I) \]
		for each function $\alpha\colon I \to J$
	\end{itemize}
	such that the following equations hold:
	\begin{itemize}
		\item (Identity) \ $\id_j \circ f = f$ when $f \in \mathcal O(I)$ and the composition is along the unique map $I \to \{j\}$, and $f = f \circ (\id_i)_{i \in I}$ with composition along the identity map $I \to I$,
		\item (Associativity) \ for two composable maps $\alpha\colon I \to J$ and $\beta\colon J \to K$ and elements $f \in \mathcal O(K)$, $g_k \in \mathcal O(\beta^{-1}(k))$, and $h_j \in \mathcal O(\alpha^{-1}(j))$ we have
		\[ (f \circ (g_k)_{k \in K}) \circ (h_j)_{j \in J} = f \circ (g_k \circ (h_j)_{j \in \beta^{-1}(k)})_{k \in K}. \]
	\end{itemize}
	A map of operads $\mathcal O \to \mathcal O'$ again consists of maps $\mathcal O(I) \to \mathcal O'(I)$ commuting with the composition map and preserving the identity elements.
\end{definition}

We can go from a coordinate-free operad $\mathcal O$ to a standard operad $\mathcal O'$ as follows: we let $\mathcal O'(n)=\mathcal O(\{1,\dots, n \})$. The right $\Sigma_n$-action on $\mathcal O'(n)$ is given by setting $f \cdot \sigma= (\id)_{\{1,\dots, n \}} \circ f$ along the map $\sigma:\{1,\dots,n\} \to \{1,\dots, n \}$, for $f\in \mathcal O'(n)$. The composition $f\circ(g_1,\dots, g_n)$ for $f\in \mathcal O'(n)$ and $g_i$ in $\mathcal O'(k_i)$ is defined as the composition in $\mathcal O$ along the map $$\{1,\dots, k_1+\dots+k_n\}  \to \{1,\dots, n\}$$ that sends the first $k_1$ letters to $1$, the second $k_2$ to $2$, and so on.

An example of a coordinate-free operad that we will use is the following. Recall that a \emph{rectilinear embedding} $[0,1] \to [0,1]$ is an embedding of the form $x\mapsto a+bx$ with $0 < b\leq 1$.
\begin{definition}\label{little-intervals-operad}
    The coordinate-free little intervals operad $D_1$ is given by 
    \begin{itemize}
        \item $D_1(J)$ is the space whose points are tuples $\{e_j\}_{j\in J}$ of rectilinear embeddings $e_j:[0,1] \to [0,1]$, such that the sets $e_j((0,1))$ are pairwise disjoint. This is topologized as a subspace of the space $\textup{Map}(\amalg_J [0,1], [0,1])$ or equivalently as a subspace of $[0,1]^{2|J|}$.
        \item The identity element in $ D_1(\{j\})$ is $\{\id:[0,1] \to [0,1]\}_{\{j\}}$.
        \item For a function $\alpha:J\to K$, the composition 
        $$D_1(K) \times \prod_{k\in K} D_1(\alpha^{-1}(k)) \to D_1(J) $$
        is given by sending $(\{e_k\}_K, \{e_j \}_{\alpha^{-1}(k)} : k\in K) $ to the set $\{e_{\alpha(j)} \circ e_j  \}_{j \in J}$.
    \end{itemize}
\end{definition}
The associative operad also has a coordinate free version.
\begin{definition}
    The coordinate-free associative operad $\Assoc$ is given by 
    \begin{itemize}
        \item $\Assoc(I)$ is the discrete space of total orderings of $I$,
        \item the identity element in $\Assoc(\{i\})$  is the unique total ordering on $\{i \}$
        \item  For a function $\alpha:I\to J$, the composition 
        $$\Assoc(J) \times \prod_{j\in J} \Assoc(\alpha^{-1}(j)) \to \Assoc(I) $$
        is defined by putting the lexicographical ordering on $I$: the ordering on $J$ is used to put the subsets $\alpha^{-1}(j)$ in order, and the orderings on $\alpha^{-1}(j)$ order the elements within $\alpha^{-1}(j)$.
    \end{itemize}
\end{definition}
There is a natural operad map from $D_1$ to $\Assoc$: a map
$$D_1(J)\to  \Assoc(J) $$
can be given by sending $\{e_j \}_{j\in J}$ to the ordering on $J$ induced by the ordering on the elements $e_j(1)$ coming from the total ordering on the interval $I$. 

When an operad takes values in the category $\Top$ of topological spaces and continuous functions, we can define algebras and modules over that operad in any symmetric monoidal topological category:
\begin{definition}\label{df:smtc}
    A \emph{symmetric monoidal topological category} is a closed symmetric monoidal category $(C,\otimes,1)$ with a strong symmetric monoidal left adjoint $L\colon (\Top,\times,\ast) \to (C,\otimes,1)$, where $\Top$ is equipped with the Cartesian symmetric monoidal structure. We write $R$ for the right adjoint functor to $L$.
\end{definition}

Given such a symmetric monoidal topological category $(C,\otimes,1)$ with adjunction $L \dashv R$, we see that $C$ is enriched in spaces with morphism spaces $RC(X,Y)$, where $C(X,Y) \in \ob(C)$ is the internal hom in the closed symmetric monoidal category $C$. This category $C$ is also tensored and cotensored over spaces via 
\[
    \begin{tikzcd}[row sep=0]
        \Top \times C \ar[r] & C \\
        (K,X) \ar[r, mapsto] & L(K) \otimes X
    \end{tikzcd}
    \qquad
    \text{ and }
    \qquad
    \begin{tikzcd}[row sep=0]
        \Top^{\op} \times C \ar[r] & C \\
        (K,X) \ar[r, mapsto] & C(L(K),X).
    \end{tikzcd}
\]

\begin{example}
    The following symmetric monoidal functors define symmetric monoidal topological categories:
    \begin{enumerate}[(a)]
        \item The identity functor $\id \colon (\Top, \times, \ast) \to (\Top, \times, \ast)$.
        \item The suspension spectrum functor 
        $\Sigma^\infty_+ \colon (\Top,\times,\ast) \to (\Sp^O,\wedge,\Sph)$. 
        \item For any symmetric monoidal category $(D,\cdot,1)$, we may turn $\Fun(D,\Sp^O)$ into a symmetric monoidal category $(\Fun(D,\Sp^O),\boxdot,1_D)$ via the Day convolution. The functor 
        \[
            L \colon (\Top, \times, \ast) \to (\Fun(D,\Sp^O),\boxdot,1)
        \]
        given by 
        \[
            L(K)(X) = (\Sigma^\infty_+K) \sma D(1_D,X)_+
        \]
        defines a symmetric monoidal topological category structure on $\Fun(D,\Sp^O)$. In fact, this category is a symmetric monoidal topological model category by \cite[Prop 3.5]{KKMMW1}.
    \end{enumerate}

\end{example}

We will be working with modules over an algebra over an operad, so we will also need the version of an operad where there are many objects.

\begin{definition}\label{df:multicategory}
	A \emph{coordinate-free symmetric multicategory}, or just a \emph{multicategory} for short, is
	\begin{itemize}
		\item a collection of objects $\ob \mathcal M$,
		\item a topological space of \emph{multimorphisms} $\mathcal M(\{a_i\}_{i \in I};b)$ for each finite set $I$, function $a\colon I \to \ob(\mathcal M)$ and for each object $b$
		\item an identity element $\id_{a,i} \in \mathcal M(\{a\}_{i \in \{i\}};a)$ for each object $a$ and singleton $\{i\}$, and
		\item continuous composition maps
		\[ \circ\colon \mathcal M(\{b_j\}_{j \in J};c) \times \prod_{j \in J} \mathcal M(\{a_i\}_{i \in \alpha^{-1}(j)};b_j) \longrightarrow \mathcal M(\{a_i\}_{i \in I};c) \]
		for each map of finite sets $\alpha\colon I \to J$ and indexed families of objects $\{a_i\}_{i \in I}$, $\{b_j\}_{j \in J}$, and $c$,
	\end{itemize}
	such that the following equations hold:
	\begin{itemize}
		\item (Identity) \ $\id_{b,j} \circ f = f$ when $f \in \mathcal M(\{a_i\}_{i \in I};b)$, $\id_{b,j} \in \mathcal M(\{b\}_{j \in \{j\}};b)$ and the composition is along the unique map $I \to \{j\}$, and $f = f \circ (\id_{a_i,i})_{i \in I}$ when $f \in \mathcal M(\{a_i\}_{i \in I};b)$, $\id_{a_i,i} \in \mathcal M(\{a_i\}_{i \in \{i\}};a_i)$ and the composition is along the identity map $I \to I$,
		\item (Associativity) \ for two composable maps $\alpha\colon I \to J$ and $\beta\colon J \to K$ and elements $f \in \mathcal M(\{c_k\}_{k \in K};d)$, $g_k \in \mathcal M(\{b_j\}_{j \in \beta^{-1}(k)};c_k)$, and $h_j \in \mathcal M(\{a_i\}_{i \in \alpha^{-1}(j)};b_j)$ we have
		\[ (f \circ (g_k)_{k \in K}) \circ (h_j)_{j \in J} = f \circ (g_k \circ (h_j)_{j \in \beta^{-1}(k)})_{k \in K}. \]
	\end{itemize}
	A \emph{multifunctor} or map of multicategories $\mathcal M \to \mathcal M'$ consists of a function $F\colon \ob \mathcal M \to \ob \mathcal M'$ and continuous maps $\mathcal M(\{a_i\}_{i \in I};b) \to \mathcal M'(\{F(a_i)\}_{i \in I};F(b))$, commuting with the composition and preserving the identity elements.
\end{definition}

\begin{definition}
    For any symmetric monoidal topological category $(C,\otimes,1)$ with adjunction $L \colon \Top \leftrightarrows C \colon R$, the \emph{underlying multicategory} $\C$ has the same objects as $C$, and the multimorphisms from $\{a_i\}_{i \in I}$ to $b$ are given by the space of maps $RC(\bigotimes_{i \in I} a_i,b)$. The composition is by tensoring the morphisms and then composing them.
    The opposite of $C$ also has an underlying multicategory, denoted $\C^{\op}$, with the same objects as $C$, and multimorphisms from $\{a_i\}_{i \in I}$ to $b$ given by $C^{\op}(\bigotimes_{i \in I} a_i,b) = RC(b,\bigotimes_{i \in I} a_i)$.
\end{definition}

\begin{definition}
    For $\mathcal{M}$ a multicategory and $C$ a symmetric monoidal topological category, an \emph{$\mathcal{M}$-algebra in $C$} is a functor of multicategories
    $$\mathcal{M} \to \C.$$ An \emph{$\mathcal{M}$-coalgebra in $C$} is a functor of multicategories 
    $$\mathcal{M}\to \C^{\op}.$$
    
\end{definition}

\begin{example}
    Let $(C, \otimes, 1)$ be a symmetric monoidal topological category with adjunction $L \colon \Top \leftrightarrows C \colon R$. 
    If $\mathcal M$ is an operad $\mathcal O$, then an $\mathcal O$-algebra in $C$ is a single object $X \in C$ and action maps
    \[ \mu_I\colon L(\mathcal O(I)) \otimes X^{\otimes I} \longrightarrow X \]
    that satisfy the usual axioms for an algebra over an operad:
    \begin{itemize}
        \item the identity element $\id \in \mathcal O(\{i\})$ is sent to the identity map of $X$, and
        \item for each function $\alpha\colon I \to J$, the following diagram commutes:
        \begin{center}
        \begin{tikzcd}[column sep = 7em]
            L\mathcal O(J) \otimes \bigotimes_{j \in J} L\mathcal O(\alpha^{-1}(j)) \otimes X^{\otimes I} \dar{\circ \otimes \id} 
            	\rar{\id \otimes (\mu_{\alpha^{-1}(j)})_{j \in J}}
            	& 
            L\mathcal O(J)  \otimes X^{\otimes J} 
            	\dar{\mu_J}
            	\\
            L\mathcal O(I)  \otimes X^{\otimes I} 
            	\rar{\mu_I}
				& 
			X.
        \end{tikzcd}
        \end{center}
    \end{itemize}
    Similarly, an $\mathcal O$-coalgebra in $C$ is a single object $X \in C$ and comultiplication maps
    \[ \delta_I\colon L(\mathcal O(I)) \otimes X \longrightarrow X^{\otimes I} \]
    such that
    \begin{itemize}
        \item the identity element $\id \in \mathcal O(\{i\})$ is sent to the identity map of $X$, and
        \item for each map of finite sets $\alpha\colon I \to J$ the following diagram commutes:
        \begin{center}
        \begin{tikzcd}[column sep = 7em]
            L\mathcal O(J) \otimes (\bigotimes_{j \in J} L\mathcal O(\alpha^{-1}(j))) \otimes X \dar{\circ \otimes \id} 
            	\rar{\delta_J}
            	& 
            \bigotimes_{j \in J} (L\mathcal O(\alpha^{-1}(j)) \otimes X)
            	\dar{(\delta_{\alpha^{-1}(j)})_{j \in J}}
            	\\
            L\mathcal O(I)  \otimes X
            	\rar{\delta_I}
				& 
			X^{\otimes I}
        \end{tikzcd}
        \end{center}
    \end{itemize}
\end{example}

Now that we have defined algebras for multicategories and not just for operads, we can recall the operadic notion of a right module over an algebra.
\begin{definition}[cf.~{\cite[4.2.1.1]{HA}}]
\label{left_module_multicategory}
	Suppose $\mathcal O$ is any operad with a map to the associative operad $\Assoc$. For each finite set $I$ and element $j \in I$, define the subspace
    \[ \mathcal O(I)^j \subseteq \mathcal O(I) \]
    to consist of those connected components whose image in $\Assoc(I)$ gives a total ordering on $I$ in which the element $j \in I$ comes first. Define the \emph{right module multicategory} $\mathcal{RM}_{\mathcal O}$ to be the multicategory with two objects $a$ and $m$, corresponding to an ``algebra'' and a ``right module'' over it, and multimorphisms
	\begin{itemize}
		\item $\mathcal{RM}_\mathcal{O}(\{a\}_{i \in I};a) = \mathcal O(I)$,
        \item $\mathcal{RM}_\mathcal{O}(\{a\}_{i \in I \setminus \{j\}} \amalg \{m\}_{\{j\}};m) = \mathcal O(I)^j$, and
        \item $\mathcal{RM}_\mathcal{O}(\{x_i\}_{i \in I};x) = \varnothing$ otherwise.
	\end{itemize}
	The identity morphisms and compositions are all given by the corresponding identity morphisms and compositions in $\mathcal O(I)$.
\end{definition}
Let $C$ be a topological symmetric monoidal category. Explicitly, an $\mathcal{RM}_{\mathcal O}$-algebra in $C$, which we also call a \emph{right $\mathcal O$-algebra module}, is a functor of multicategories $X:\mathcal{RM}_{\mathcal O} \to \C $ which singles out objects $A:=X(a)$ and $M:=X(m)$ in $C$. The maps
$\mathcal{RM}_{\mathcal O}( \{a \}_{i\in I}; a) \to C(\otimes_{i\in I} A,A)$ make $A$ into an $\mathcal O$-algebra in the traditional sense. Moreover, the maps 
$\mathcal{RM}_{\mathcal O}( \{a \}_{i\in I\setminus \{j\}} \amalg \{m\}_{\{j\}}; m) \to C(M \otimes \bigotimes_{i\in I \setminus \{ j \} } A,M)$ provide, for any $f\in \mathcal{O}(I)$, a map
$M\otimes A \otimes \dots \otimes A \to M$
where the order of the $A$'s is determined by the element in $\Assoc(I)$ that is the image of $f$ under the map $\mathcal{O} \to \Assoc$.

The axioms that need to be satisfied are the usual axioms for an $\mathcal O$-algebra structure on $A$, and in addition:
\begin{itemize}
    \item 
    For any singleton $\{i\}$, the identity element $\id_{m,i}\in \mathcal{RM}_{\mathcal{O}}(\{m\}_{\{i\}}, m)$ is sent to the identity in $C(M,M)$.
    \item For any map of finite sets $\alpha:I\to J$ and element $i' \in I$ whose image is $\alpha(i') = j' \in J$, the following diagram commutes:
    \begin{center}
        \begin{tikzcd}[column sep = -7em, row sep = 1.5em]
            L\mathcal{O}(J)^{j'} \otimes \big[ L\mathcal O(\alpha^{-1}(j'))^{i'} \otimes \bigotimes_{j \in J \setminus \{j'\}} L\mathcal O(\alpha^{-1}(j))\big] \otimes M \otimes A^{\otimes I \setminus \{i'\}} \ar{dd}{\circ \otimes \id} 
            	\ar{rd}{\id \otimes (\mu_{\alpha^{-1}(j)})_{j \in J}}
            	& \\
            &
            L\mathcal{O}(J)^{j'} \otimes (M \otimes A^{\otimes J \setminus \{j'\}})
            	\ar{dd}{\mu_J}
            	\\
            L\mathcal O(I)^{i'}  \otimes M \otimes A^{\otimes I \setminus \{i'\}}
            	\ar{rd}{\mu_I}
				& \\
			&
            M.
        \end{tikzcd}
    \end{center}
\end{itemize}

On the other hand, an $\mathcal{RM}_\mathcal{O}$-coalgebra in $C$, which we also call a \emph{ right $\mathcal{O}$-coalgebra-comodule},  is a functor of multicategories $X:\mathcal{RM}_{\mathcal O} \to \C^{\op} $ which singles out objects $A:=X(a)$ and $M:=X(m)$ in $C$. The maps
$\mathcal{RM}_{\mathcal O}( \{a \}_{i\in I}; a) \to C(A,\otimes_{i\in I} A)$ make $A$ into an $\mathcal O$-coalgebra in the traditional sense. Moreover, the maps 
$\mathcal{RM}_{\mathcal O}( \{a \}_{i\in I\setminus \{j\}} \amalg \{m\}_{\{j\}}; m) \to C(M,M \otimes \bigotimes_{i\in I \setminus \{ j \}} A)$ provide, for any $f\in \mathcal{O}(I)$, a map
$M \to M\otimes A \otimes \dots \otimes A  $
where the order of the $A$'s is determined by the element in $\Assoc(I)$ that is the image of $f$ under the map $\mathcal{O} \to \Assoc$.

The axioms that need to be satisfied are the usual axioms for an $\mathcal O$-coalgebra structure on $A$, and in addition:
\begin{itemize}
    \item 
    For any singleton $\{i\}$, the identity element $\id_{m,i}\in \mathcal{RM}_{\mathcal{O}}(\{m\}_{\{i\}}, m)$ is sent to the identity in $C(M,M)$.
    \item For any map of finite sets $\alpha:I\to J$ and element $i' \in I$ whose image is $\alpha(i') = j' \in J$, the following diagram commutes:
        \begin{center}
        \begin{tikzcd}
         L\mathcal{O}(J)^{j'} \otimes \big[ L\mathcal O(\alpha^{-1}(j'))^{i'} \otimes \bigotimes_{j \in J \setminus \{j'\}} L\mathcal O(\alpha^{-1}(j))\big] \otimes M \ar{d}{\delta_I} \arrow[r, "\circ \otimes \id"] & L\mathcal O(I)^{i'}  \otimes M \ar{dd}{\delta_J}\\
         \big[ L\mathcal O(\alpha^{-1}(j'))^{i'} \otimes \bigotimes_{j \in J \setminus \{j'\}} L\mathcal O(\alpha^{-1}(j))\big] \otimes (M \otimes A^{\otimes J \setminus \{j'\}})  \arrow[d, "\cong"] & \\ 
         (L\mathcal O(\alpha^{-1}(j'))^{i'} \otimes M)
            \bigotimes_{j \in J \setminus \{j'\}} (L\mathcal O(\alpha^{-1}(j)) \otimes A) \arrow[r, "(\delta_{\alpha^{-1}(j)})_{j \in J}"] & M \otimes A^{\otimes I \setminus \{i'\}}
        \end{tikzcd}
    \end{center}
\end{itemize}

\begin{example}\label{ex:little_intervals_righ_module_cat}
For the operad $D_1$, the right module multicategory $\mathcal{RM}_{D_1}$ looks as follows. The space of multimorphisms $\mathcal{RM}_{D_1}(\{a \}_{j\in J};a)$ is $D_1(J)$, and the space
\[ \mathcal{RM}_{D_1}(\{a\}_{j \in J \setminus \{k\}} \amalg \{m\}_k;m) \]
is the subspace of $D_1(J)$ consisting of those $\{e_j\}_{j\in J}$ such that $e_k(1)<e_j(1)$ for all $j\in J\setminus \{k\}$. In other words, the interval corresponding to the ``module'' must be the first interval. We denote this space also by $D_1(J)^k$.
\end{example}

Spelling this out, a right $D_1$-coalgebra comodule consists of two objects $A, M \in C$, composition maps
\[ L(D_1(I)) \otimes A \longrightarrow A^{\otimes I} \]
for each finite set $I$ and composition maps
\[ L(D_1(I)^j) \otimes M \longrightarrow M \otimes A^{\otimes I \setminus \{j\}} \]
for each finite set $I$ and element $j \in I$, such that
\begin{itemize}
    \item the identity element $\id \in D_1(\{i\})$ is sent to the identity maps of $A$ and $M$,
    \item for each map of finite sets $\alpha\colon I \to J$ the following diagram commutes:
        \begin{center}
        \begin{tikzcd}[column sep = 7em]
            LD_1(J) \otimes (\bigotimes_{j \in J} LD_1(\alpha^{-1}(j))) \otimes A \dar{\circ \otimes \id} 
            	\rar{\delta_J}
            	& 
            \bigotimes_{j \in J} (LD_1(\alpha^{-1}(j)) \otimes A)
            	\dar{(\delta_{\alpha^{-1}(j)})_{j \in J}}
            	\\
            LD_1(I)  \otimes A
            	\rar{\delta_I}
				& 
			A^{\otimes I}
        \end{tikzcd}
        \end{center}
    \item For any map of finite sets $\alpha:I\to J$ and element $i' \in I$ whose image is $\alpha(i') = j' \in J$, the following diagram commutes:
         \begin{center}
        \begin{tikzcd}
         LD_1(J)^{j'} \otimes \big[ LD_1(\alpha^{-1}(j'))^{i'} \otimes \bigotimes_{j \in J \setminus \{j'\}} LD_1(\alpha^{-1}(j))\big] \otimes M \ar{d}{\delta_I} \arrow[r, "\circ \otimes \id"] & LD_1(I)^{i'}  \otimes M \ar{dd}{\delta_J}\\
         \big[ LD_1(\alpha^{-1}(j'))^{i'} \otimes \bigotimes_{j \in J \setminus \{j'\}} LD_1(\alpha^{-1}(j))\big] \otimes (M \otimes A^{\otimes J \setminus \{j'\}})  \arrow[d, "\cong"] & \\ 
         (LD_1(\alpha^{-1}(j'))^{i'} \otimes M)
            \bigotimes_{j \in J \setminus \{j'\}} (LD_1(\alpha^{-1}(j)) \otimes A) \arrow[r, "(\delta_{\alpha^{-1}(j)})_{j \in J}"] & M \otimes A^{\otimes I \setminus \{i'\}}
        \end{tikzcd}
    \end{center}
\end{itemize}

\begin{definition}\label{def-comodule-algebra}
    Let $A$ be a $(\text{Comm}, D_1)$-bialgebra. That is, it is a $D_1$-algebra in the opposite category of the category of Comm-algebras; see \cite{KKMMW1} for details. We say that $M$ is a right $D_1$-comodule algebra over $A$ if $M$ is a commutative monoid, and its comodule structure maps are maps of commutative monoids. That is, the corresponding functor of multicategories $\mathcal{RM}_{D_1} \to \mathcal{C}$ refines to a functor $\mathcal{RM}_{D_1} \to \CMon(\mathcal{C})$.
\end{definition}

\section{The comodule structure (the higher Dehn invariant)}

One of the main goals of this paper is to show that the spectra $\mathcal{E}$, $\mathcal{H}$ and $\mathcal{S}$ are $D_1$-coalgebra comodules over the $D_1$-coalgebra $\widetilde{\mathcal{S}}$, where $D_1$ is the little intervals operad from \cref{little-intervals-operad}. Recall that these spectra are homotopy colimits of the spectra-valued functors $\funSah^E$, $\funSah^H$, $\funSah^S$ and $\funSah^{\tilde S}$, respectively.

The main technical ingredient that we provide in this section is to show that $(\ST,*)_E$ is a comodule over $(*,\Sigma\!\ST)_E$, $(\ST,*)_H$ is a comodule over $(*,\Sigma\!\ST)_H$, and $\Sigma\!\PT_S$ is a comodule over $\Sigma\!\ST_S$. Since the homotopy colimit functors $(\Sigma D,D)_E$ and $D$ are strong symmetric monoidal, this implies that $\funSah^E$ and $\funSah^S$ are comodules over $\funSah^{\tilde S}$. In Section \ref{sect:infinity}, we will pass to $\infty$-categories to show that this implies that $\mathcal{E}$ and $\mathcal{S}$ are comodules over $\widetilde{\mathcal{S}}$.

We want to conclude the same for $\mathcal{H}$, but as an ordinary functor, $(\Sigma D, D)_H$ is not strong symmetric monoidal. Therefore in this case we will have to pass to the $\infty$-category a little earlier so that $(\Sigma D, D)_H$ is strong symmetric monoidal. This will then show that $\funSah^{H}$ is a comodule over $\funSah^{\tilde S}$, and therefore $\mathcal{H}$ is a comodule over $\widetilde{\mathcal{S}}$.

\subsection{The comodules}

Recall that $\Sigma\!\ST:\Dip^S \to \topp$ is a $D_1$-coalgebra by \cite[Lemma 4.14]{KKMMW1}.

We showed this for the operad $D_1$ rather than the coordinate-free version of the operad, but these concepts are equivalent. Therefore, there is a functor of multicategories 
$$A:D_1 \to \Fun(\Dip^S ,\topp)^{\op}$$
that singles out the object $\Sigma\!\ST$ and exhibits it as a $D_1$-coalgebra. When $X = S$, this also makes $\Sigma\!\ST_S$ a $D_1$-coalgebra in $\Fun(\Luna^S ,\topp)^{\op}$, because this is the same functor on the same category. When $X = E$ or $H$, the same structure makes $(*,\Sigma\!\ST)_X$ into a $D_1$-coalgebra in $\Fun(\Luna^X ,\topp)$.

In the following construction, we define the structure maps that are needed to show that $(\ST, *)_E, ( \ST, *)_H$ and $\Sigma\!\PT_S$ are right $D_1$-coalgebra comodules over $(*,\Sigma\!\ST)_E$, $(*,\Sigma\!\ST)_H$ and $\Sigma\!\ST_S$ respectively. 

In the construction, we use that for $X=E,H$ and $S$, the symmetric monoidal category $(\Fun(\Luna^X,\topp), \boxwedge)$ is a symmetric monoidal topological category, where the left adjoint is the functor 
$$L:(\Top, \times) \to (\Fun(\Luna^X,\topp),\boxwedge)$$ which sends an unbased space $X$ to the diagram that is $X_+ = X \sqcup *$ on $0\in \Dip^S$, and $*$ on all other objects in $\Luna^X$. We denote the right adjoint of $L$ by $R$.

\begin{construction} \label{constr:d_1-coalg_comod}
We recall that a right $D_1$-coalgebra comodule is a functor of multicategories 
$$ F:\mathcal{RM}_{D_1} \to \Fun(\Luna^X,\topp)^{\op}.$$
To construct such a functor, we need objects $F(a)$ and $F(m)$ in $\Fun(\Luna^X,\topp)$. When $X=E,H$, we define $F(a)$ to be $(*,\Sigma\!\ST)_X$, and $F(m)$ to be $(\ST,*)_X$. For $X=S$, we define $F(a) = \Sigma\!\ST_S$ and $F(m)=\Sigma\!\PT_S$. 

We need to provide maps 
$$ \mathcal{RM}_{D_1}(\{a\}_{j\in J}; a) \to R\Fun(\Luna^X,\topp)(F(a),\boxwedge_{j\in J} F(a)),$$
or equivalently, 
maps 
$$ L\mathcal{RM}_{D_1}(\{a\}_{j\in J}; a) = D_1(J)_+ \to \Fun(\Luna^X,\topp)(F(a),\boxwedge_{j\in J} F(a)).$$
For this we can use the maps that make $F(a)$ a $D_1$-coalgebra, which we already have by the discussion above. The only other maps that we need to define are
$$ L\mathcal{RM}_{D_1}(\{a\}_{j\in J\setminus \{k\}}\amalg  \{m\}_{\{k\}}, m )  = D_1(J)^k_+ \to \Fun(\Luna^X,\topp)(F(m), F(m)\boxwedge  F(a)^{\boxwedge J\setminus \{k\}}) $$
for finite sets $J$ and $k\in J$, where $D_1(J)^k$ is as in Example \ref{ex:little_intervals_righ_module_cat}. This is equivalent to defining natural transformations
$$D_1(J)^k_+ \boxwedge F(m) \to F(m)\boxwedge F(a)^{\boxwedge J \setminus \{ k\}}.$$

Suppose $X=E,H$. Then on $V$ in $\Dip^S$ such a transformation necessarily evaluates to the trivial map $* \to *$. On $V$ in $\Dip^X$, the transformation evaluates to a map
$$D_1(J)^k_+ \wedge \ST (V) \to \bigvee_{\substack{W_k \in \Dip^X, W_j \in \Dip^S \text{ for }j \in J \setminus \{k\}\\ W_k \oplus \bigoplus_{J\setminus \{k\}} W_j \cong V}}\ST(W_k) \wedge  \Sigma ST(W_j)^{\wedge J\setminus \{k\}  },$$
where summands are identified together whenever there is an isomorphism on each $W_j$ that commutes with the isomorphism to $V$. If $X=S$, then on $V$ in $\Dip^S$, the natural transformation evaluates to a map
$$D_1(J)^k_+ \wedge \Sigma\!\PT (V) \to \bigvee_{\substack{W_k \in \Dip^X, W_j \in \Dip^S \text{ for }j \in J \setminus \{k\}\\ W_k \oplus \bigoplus_{J\setminus \{k\}} W_j \cong V}} \Sigma\!\PT(W_k) \wedge  \Sigma\!\ST(W_j)^{\wedge J\setminus \{k\}  }$$
where the summands are identified in the same way. It is helpful here to recall that for an element $e = (e_j)_{j\in J}$ in the space $D_1(J)^k$, there is a canonical total ordering on $J$ induced by the map $D_1 \to \Assoc$, and in particular a total ordering on $J$ in which $k$ is the first element.

The above wedge sums are messy because they have to be considered up to isomorphism in each summand $W_j$. To clean this up, we re-index using Lemma \ref{lem:sequence of oplus vs ominus}. If we let $n-1 = |J \setminus \{k\} |$, then the comultiplication for each $J$, $k\in J$, and $(e_j)_{j\in J}$ in $D_1(J)^k$ is of the form
\begin{equation}\label{eq:comodule_one_point_in_D1J_EH}
m:\ST (V) \to \bigvee_{\substack{0 \neq V_0\subseteq \dots \subseteq V_{n-1} = V \\ \textup{ in }\Dip^X }} \ST(V_0) \wedge  \Sigma\!\ST(V_1\ominus V_0)\dots \wedge \Sigma\!\ST(V_{n-1} \ominus V_{n-2})
\end{equation}
if $X=E,H$, and
\begin{equation}\label{eq:comodule_one_point_in_D1J_S}
m:\Sigma\!\PT (V) \to \bigvee_{\substack{0 \subseteq V_0\subseteq \dots \subseteq V_{n-1} = V \\ \textup{ in }\Dip^S }} \Sigma\!\PT(V_0) \wedge  \Sigma\!\ST(V_1\ominus V_0)\dots \wedge \Sigma\!\ST(V_{n-1} \ominus V_{n-2})
\end{equation}
if $X=S$. Now there is just one summand for every such string of subspaces $V_0\subseteq \dots \subseteq V_{n-1} = V$ with no further identifications.

We first construct the map $m$ in \eqref{eq:comodule_one_point_in_D1J_EH} when $X=E, H$. 
Let $e=(e_j)_{j\in J}$ in $D_1(J)^k$ and let $f\in \ST (V)$.  By \cref{prop-ST-conf}, $f$ is a step function $f:[0,1]\to \Sub(V)_{\neq 0}$, with corresponding string of subspaces $0 \neq U_0 \subsetneq \cdots \subsetneq U_m = V$. If any of the cut points of $f$ lie outside the set $[0,e_k(0)] \cup \bigcup_{j \in J} e_j((0,1))$, then $m(f)$ is defined to be the basepoint $*$. In other words, we allow cut points to the left of the leftmost interval $e_k((0,1))$, but if any cut point appears in between any two intervals, or to the right of the rightmost interval, then $m(f) = *$.

Outside of these cases, the cut points are all inside the open intervals $e_j((0,1))$ or to the left of the leftmost interval. Then we define $V_0$ to be $f(e_k(1))$, the endpoint of the leftmost interval. We let $\alpha: \{1,\dots, n-1 \} \to J\setminus \{k\}$ be the total ordering induced by the map $D_1 \to \Assoc$, and for $i=1, \dots, n-1$ we define $V_i$ to be $f(e_{\alpha(i)}(1))$, the endpoints of the remaining intervals. This gives a chain of subspaces $0 \neq V_0\subseteq \dots \subseteq V_{n-1} = V$ in $\Dip^X$.

Next we define $f_0:[0,1] \to \Sub(V_0)_{\neq 0}$ to be the step function $f\circ e_k$, and think of this as an element $f_0 \in \ST(V_0)$. For $i=1, \dots, n-1$ we have a step function $f_i:[0,1] \to \Sub(V_i)$ that is given by $f\circ e_{\alpha(i)}$. Note that for all $t \in [0,1]$, we have that $f(e_{\alpha(i)}(t))$ contains $V_{i-1}$. Therefore we can subtract $V_{i-1}$ from all values of $f \circ e_{\alpha(i)}$ and obtain a step function $f_i = (f \circ e_{\alpha(i)}) \ominus V_{i-1}: [0,1] \to \Sub(V_i \ominus V_{i-1})$. This defines an element $f_i \in \Sigma\!\ST (V_i \ominus V_{i-1})$. Finally, we define $m(f) = (f_0,f_1,\dots, f_{n-1})$.

Next we construct the map $m$ in \eqref{eq:comodule_one_point_in_D1J_S} when $X=S$. Let $\phi$ be in $\Sigma\!\PT(V)$. By Proposition \ref{prop-SPT-conf}, we can identify $\phi$ with a pair $(f,v)$ where $f:[0,1] \to \Sub(V)$ is a step function with corresponding string of subspaces $0 \subseteq U_0 \subsetneq \cdots \subsetneq U_m = V$, and $v \in S^{U_0} = S^{f(0)}$. Similarly to the case where $X=E,H$, we first define $m(f,v) = *$ if any of the cut points of $f$ lie outside the set $[0,e_k(0)] \cup \bigcup_{j \in J} e_j((0,1))$. Otherwise we define $V_0 = f(e_k(1))$ and $f_0 =  f \circ e_k$. The inclusion $U_0 = f(0) \subseteq f(e_k(0)) = f_0(0)$ induces an inclusion $S^{U_0} \subseteq S^{f_0(0)}$, so that we can regard $v \in S^{f_0(0)}$, making $(f_0,v) \in \Sigma\!\PT(V_0)$.

Then we define $V_i = f(e_{\alpha(i)}(1))$ as before, giving a chain of subspaces $0 \subseteq V_0\subseteq \dots \subseteq V_{n-1} = V$ in $\Dip^S$. Then we take $f_i = (f \circ e_{\alpha(i)}) \ominus V_{i-1} \in \Sigma\!\ST (V_i \ominus V_{i-1})$ and define $m(f,v) = ((f_0,v), f_1,\dots, f_{n-1})$.
\end{construction}
\begin{figure}[ht]
\begin{center}
\resizebox{\textwidth}{!}{ 
    \begin{tikzpicture}
        \draw[dotted,gray] (0.025,0) -- (0.025,-3) -- (-1,-4);
        \draw[dotted,gray] (0.975,-1) -- (0.975,-3) -- (0,-4);
        \draw[dotted,gray] (1.975,0) -- (1.975,-3) -- (1,-4);
        \draw[dotted,gray] (1.975,-3) -- (2.025,-4);
        \draw[dotted,gray] (2.975,-1) -- (2.975,-3) -- (2.975,-4);
        \draw[dotted,gray] (3.975,-1) -- (3.975,-3) -- (3.975,-4);
        \draw[dotted,gray] (4.975,0) -- (4.975,-3) -- (4.975,-4);
        \draw[dotted,gray] (4.975,-3) -- (6,-4);
        \draw[dotted,gray] (5.975,-1) -- (5.975,-3) -- (7,-4);        
        \draw[dotted,gray] (7.975,-1) -- (7.975,-3) -- (9,-4);        
        \draw[dotted,gray] (9.975,0) -- (9.975,-3) -- (11,-4);
    
        \draw[ultra thick,OI6,|-|] (0,0) -- (2,0);
        \draw[ultra thick,OI6,-|] (2,0) -- (5,0);
        \draw[ultra thick,OI6,-|] (5,0) -- (10,0);
        \node at (11.6,0) {$\in \textcolor{OI6}{D_1(\{1,2,3\})^1_+}$};

        \draw[ultra thick,OI2,|-|] (0,-1) --node[above]{$U_0$} (1,-1);
        \draw[ultra thick,OI2,-|] (1,-1) --node[above,fill=white]{$U_1$} (3,-1);
        \draw[ultra thick,OI2,-|] (3,-1) --node[above]{$U_2$} (4,-1); 
        \draw[ultra thick,OI2,-|] (4,-1) --node[above,fill=white]{$U_3$} (6,-1); 
        \draw[ultra thick,OI2,-|] (6,-1) --node[above]{$U_4$} (8,-1); 
        \draw[ultra thick,OI2,-|] (8,-1) --node[above]{$U_5=V$} (10,-1); 
        \node at (11,-1) {$\in \textcolor{OI2}{\ST(V)}$};

        \draw[ultra thick,OI2,|-|] (0,-3) --node[above]{$U_0$} (1,-3);
        \draw[ultra thick,OI2,-|] (1,-3) --node[above,fill=white]{$U_1$} (3,-3);
        \draw[ultra thick,OI6] (1.975,-3.15) -- (1.975,-2.85); 
        \draw[ultra thick,OI2,-|] (3,-3) --node[above]{$U_2$} (4,-3);
        \draw[ultra thick,OI2,-|] (4,-3) --node[above]{$U_3$} (6,-3);
        \draw[ultra thick,OI6] (4.975,-3.15) -- (4.975,-2.85); 
        \draw[ultra thick,OI2,-|] (6,-3) --node[above]{$U_4$} (8,-3);
        \draw[ultra thick,OI2,-|] (8,-3) --node[above]{$U_5$} (10,-3); 

        \draw[ultra thick,|-|] (-1,-4) --node[below]{\scriptsize $U_0$} (0,-4);
        \draw[ultra thick,-|] (0,-4) --node[below]{\scriptsize $U_1$} (1,-4);
        \node at (0,-5) {$\in \ST(U_1)$};

        \draw[ultra thick,|-|] (2,-4) --node[below]{\scriptsize $0$} (3,-4);
        \draw[ultra thick,-|] (3,-4) --node[below]{\scriptsize $U_2/U_1$} (4,-4);
        \draw[ultra thick,-|] (4,-4) --node[below]{\scriptsize $U_3/U_1$} (5,-4);
        \node at (3.5,-5) {$\in \Sigma \ST(U_3\ominus U_1)$};

        \draw[ultra thick,|-|] (6,-4) --node[below]{\scriptsize $0$} (7,-4);
        \draw[ultra thick,-|] (7,-4) --node[below]{\scriptsize $U_4/U_3$} (9,-4);
        \draw[ultra thick,-|] (9,-4) --node[below]{\scriptsize $U_5/U_3$} (11,-4);
        \node at (8.5,-5) {$\in \Sigma\ST(U_5\ominus U_3)$};
    \end{tikzpicture}
}
\end{center}
\caption{An illustration of the comultiplication $$D(3)_+ \wedge \ST(V) \to (\ST\protect \boxwedge \Sigma\ST^{\hspace{1pt}\protect \boxwedge 2})(V).$$
The depicted element of $D_1(\{1,2,3\})^1_+$ consists of embeddings $e_1,e_2,e_3:[0,1]\to [0,1]$ whose images together cover $[0,1]$. 
This particular coproduct lands in the product 
\(
	\ST(U_1) \wedge \Sigma \ST(U_3/U_1) \wedge \Sigma \ST(U_5/U_3),
\)
.}
\label{CoproductFigure}
\end{figure}

\begin{proposition}\label{prop:coalg_comod_continuous}
    The maps $m$ defined in \eqref{eq:comodule_one_point_in_D1J_EH} and \eqref{eq:comodule_one_point_in_D1J_S} 
    give rise to continuous maps 
    $$m_{J,k}^V: D_1(J)^k_+ \wedge \ST (V) \to \bigvee_{\substack{0 \neq V_0 \subseteq \dots \subseteq V_{n-1} = V \\ \textup{ in } \Dip^X}}\ST(V_0) \wedge  \Sigma\!\ST(V_1\ominus V_0) \wedge \dots \wedge \Sigma\!\ST(V_{n-1}\ominus V_{n-2}) $$
    for $V$ in $\Dip^X$ when $X=E,H$, and  
    $$m_{J,k}^V :D_1(J)^k_+ \wedge \Sigma\!\PT (V) \to \bigvee_{\substack{0 \subseteq V_0 \subseteq \dots \subseteq V_{n-1} = V \\ \textup{ in }\Dip^S }} \Sigma\!\PT(V_0) \wedge \Sigma\!\ST(V_1 \ominus V_0) \wedge \dots \wedge \Sigma\!\ST(V_{n-1}\ominus V_{n-2})$$ for $V$ in $\Dip^S$ when $X=S$. 
\end{proposition}
\begin{proof}
  First suppose $X = E$ or $H$. Recall that $\ST(V)$ is a quotient of realizations of posets, so we can reduce to showing that the restriction of $m^V_{J,k}$ to simplices $D_1(J)^k \times \Delta^p$ is continuous. The proof of this is essentially the same as in \cite[Lemma 4.12]{KKMMW1}: we cut $[0,1]$ into finitely many intervals $A_l$ using the endpoints $e_j(0)$ and $e_j(1)$, then we cover $D_1(J)^k \times \Delta^p$ by finitely many closed sets, each one defined by picking a closed interval $A_l$ where the $i$th cut point $s_i = t_0 + \cdots + t_i$ of the step function must live. We observe that on each of these closed sets $m^V_{J,k}$ is either the constant map to $*$ or is given by a linear formula, and is thus continuous.
  
  The key feature in \cref{constr:d_1-coalg_comod} that allows this to work is that when a cut point lies on the boundary of an interval $e_j((0,1))$, so that we are in the intersection of two of the above closed sets, the two possible interpretations for $m(f)$ agree. In most cases this is because they are both the basepoint $*$, but for the left-hand side of the leftmost little interval $e_k(0)$, we use the relation coming from the 0th face map in $|\varnothing \subsetneq U \subseteq V_0|$ to see that both interpretations give the same nontrivial point in $\ST(V_0)$.

  For $X=S$, we recall that by \cref{spt_simplify_spherical} the space $\Sigma\!\PT(V)$ is topologized as the quotient
  $$ \left(\coprod_{0 \subseteq U_0 \subsetneq \dots \subsetneq U_p = V} \frac{D(U_0)}{S(U_0)} \times \Delta^p\right)/\sim.$$
  Hence, it suffices to show that the restriction of $m^V_{J,k}$ to each $D_1(J)^k \times D(U_0) \times \Delta^p$, corresponding to a flag $0 \subseteq U_0 \subsetneq \dots \subsetneq U_p = V$, is continuous. To do this, we again cut $[0,1]$ into finitely many intervals $A_l$ using the endpoints $e_j(0)$ and $e_j(1)$, then we cover $D_1(J)^k \times D(U_0) \times \Delta^p$ by finitely many closed sets, each one defined by picking a closed interval $A_l$ where the $i$th cut point of the step function $s_i = t_0 + \cdots + t_i$ must live, as in the proof for $X = E$ and $H$.
  
  In this case, on each of these closed sets $m^V_{J,k}$ is given by the product of a linear map and an inclusion $D(U_0) \subseteq D(U_i)$. As a concrete example, take $J = \{j,k\}$. Then a point in  $D_1(J)^k \times D(U_0) \times \Delta^p$ is a tuple
  $$((a_k,b_k, a_j, b_j), v, (t_0, \dots, t_p))$$
  where $(a_k, b_k)$ and $(a_j, b_j)$ are disjoint open intervals in $[0,1]$ with $b_k\leq a_j$, $v\in D(U_0)$, and $(t_0,\dots, t_p) \in \Delta^p$. Recall from Proposition \ref{prop-SPT-conf} that the coordinates $(t_0,\dots, t_p)$ can be interpreted as the step lengths of a step function $\phi$ with images $U_0,\dots, U_p$ . One of the above closed sets of $D_1(J)^k \times D(U_0) \times \Delta^p$ is
  \begin{multline*}
      \{((a_k,b_k, a_j, b_j), v, (t_0, \dots, t_p)) \in D_1(J)^k \times D(U_0) \times \Delta^p : \\
      t_0 + \dots + t_h \leq a_k \leq t_0 + \dots + t_{h+1}, \\
      t_0 + \dots + t_i \leq b_k \leq a_j \leq t_0 + \dots + t_{i+1}, \\
      t_0 + \dots + t_{p-1} \leq b_j \leq 1 = t_0 + \dots + t_p  \}.
  \end{multline*}
 
 Again, the key feature of our construction that makes this work is that on the boundary of each of our closed sets where one of the cut points is at an endpoint of an interval $e_j((0,1))$, both interpretations of $m(f,v)$ agree.
\end{proof}

\begin{proposition}\label{prop:dehn_multifunctor}
    The continuous maps $m_{J,k}^V$ that arise from Proposition \ref{prop:coalg_comod_continuous} define a functor of multicategories $F:\mathcal{RM}_{D_1} \to \Fun(\Luna^X,\topp)^{\op}$. 
\end{proposition}
\begin{proof}
    All morphisms in $\Luna^X$ are isomorphisms and they commute with the maps $m_{J,k}^V$ as we defined them, giving a natural transformation
    $$m_{J,k}:D_1(J)^k_+ \boxwedge F(m) \to F(m) \boxwedge (F(a))^{\boxwedge J \setminus \{ k\}}.$$ 
    We need to show that the transposes of these natural transformations
    $$\mu_{J,k}: D_1(J)^k_+ \to \Fun(\Luna^X,\topp)(F(m), F(m) \boxwedge  F(a)^{\boxwedge J\setminus \{k\}}) $$
    along with the maps
    $$a_J:D_1(J) \to \Fun(\Luna^X,\topp)(F(a),\boxwedge_{j\in J}  F(a))$$
    indeed define a functor of multicategories
    $$ F:\mathcal{RM}_{D_1} \to \Fun(\Luna^X,\topp)^{\op}.$$
    Since we already know that the maps $a_J$ define a functor of multicategories $D_1 \to \Fun(\Luna^X,\topp)$, we only need to check the remaining conditions involving the maps $\mu_{J,k}$.

    First we check that for a singleton $\{j\}$, the map 
    $$\mu_{\{j\},j}:D_1(\{j\})^j_+ = D_1(\{j\})_+ \to \Fun(\Luna^X ,\topp)(F(m),F(m)) $$
    sends $\id_{\{j\}}$ to the identity $\id:(\ST, *)_X\to (\ST, *)_X$ (or, if $X=S$, $\Sigma \PT_S$), which it does by definition. 

    We also need to check that for any map of finite sets $\alpha:J\to K$ and element $j' \in J$ whose image is $\alpha(j') = k' \in K$, the following diagram commutes:
    \begin{center}
        \begin{tikzcd}
         D_1(K)^{k'}_+ \boxwedge \big[D_1(\alpha^{-1}(k'))^{j'}_+ \boxwedge D_1(\alpha^{-1}(k))_+^{ \boxwedge {k \in K \setminus \{k'\}}}\big] \boxwedge F(m) 
            \ar{d}{m_{K,k'}} 
            \arrow[r, "\circ \boxwedge \id"] 
            & 
        D_1(J)^{j'}_+  \boxwedge F(m) 
            \ar{dd}{m_{J,j'}}
            \\
         \big[ D_1(\alpha^{-1}(k'))^{j'}_+ \boxwedge   D_1(\alpha^{-1}(k))_+^{\boxwedge{k \in K \setminus \{k'\}}}\big] \boxwedge (F(m) \boxwedge F(a)^{\boxwedge K \setminus \{k'\}}) 
            \arrow[d, "\cong"] 
            & 
            \\ 
         (D_1(\alpha^{-1}(k'))^{j'}_+ \boxwedge F(m)) \boxwedge
        (D_1(\alpha^{-1}(k))_+ \boxwedge F(a))^{ \boxwedge{k \in K \setminus \{k'\}} }
            \arrow[r] 
            & 
        F(m) \boxwedge F(a)^{\boxwedge J \setminus \{j'\}}
        \end{tikzcd}
    \end{center}
where the bottom map is $(m_{\alpha^{-1}(k'),j'}, a_{\alpha^{-1}(k)})_{k\in  K\setminus \{k'\} )}$. Though the notation is more complicated, the proof of this is the same as in \cite[Proposition 4.14]{KKMMW1}. As in that case, the commutativity of the diagram at a certain $V$ in $\Dip^X$ boils down to showing that for nested geometric subspaces $A \subseteq B \subseteq C \subseteq V$ we have $(C \ominus A)\ominus (B \ominus A) = C \ominus B$ as subspaces of $V$. In this context, we proved this in Lemma \ref{lem:oplus_ominus}.
\end{proof}

\subsection{The suspended comodules}
We have now shown that $(\ST, *)_E, (\ST, *)_H$ and $\Sigma\!\PT_S$ are right $D_1$-coalgebra comodules over $(*,\Sigma\!\ST)_E$, $(*,\Sigma\!\ST)_H$ and $\Sigma\!\ST_S$ respectively. 

In this subsection, we take a look at $(\Sigma\!\ST, *)_E$ and $(\Sigma\!\ST, *)_H$, in other words, the suspended versions of  $(\ST, *)_E$ and $(\ST, *)_H$. By suspending the structure maps defined in Construction \ref{constr:d_1-coalg_comod}, these inherit the structure of comodules over $(*,\Sigma\!\ST)_E$ and $(*,\Sigma\!\ST)_H$ respectively. We will need the explicit version of the structure maps for the proofs in \cref{sec:apartments}.

\begin{construction}\label{constr:d_1-coalg_comod_suspended}
     Let $X=E,H$. Then for $V$ in $\Dip^X$, $J$ a finite set, $k\in J$, and $e = (e_j)_J$ in $D_1(J)^k$, we define maps
     $$Cm:C\!\ST (V) \to \bigvee_{\substack{0 \neq V_0\subseteq \dotsc \subseteq V_{n-1} = V \\ \textup{ in }\Dip^X }} C\!\ST(V_0) \wedge  \Sigma\!\ST(V_1\ominus V_0) \wedge \dotsc \wedge \Sigma\!\ST(V_{n-1} \ominus V_{n-2})$$
     $$\Sigma m:\Sigma\!\ST (V) \to \bigvee_{\substack{0 \neq V_0\subseteq \dotsc \subseteq V_{n-1} = V \\ \textup{ in }\Dip^X }} \Sigma\!\ST(V_0) \wedge  \Sigma\!\ST(V_1\ominus V_0) \wedge \dotsc \wedge \Sigma\!\ST(V_{n-1} \ominus V_{n-2})$$
     as follows. The cone $C\!\ST(V)$ is modeled by allowing our step functions to have $f(0) = 0$, so they take values in $\Sub(V)$ rather than $\Sub(V)_{\neq 0}$. The suspension $\Sigma\!\ST(V)$ is obtained from this by quotienting $f$ to the basepoint if $f(t)\neq 0$ for all $t>0$, as in \cref{prop-SST-conf}.

     We define $Cm$ as in \cref{constr:d_1-coalg_comod}, but with the following changes. As before we send $f$ to the basepoint if any of the cut points of $f$ lie outside the set $[0,e_k(0)] \cup \bigcup_{j \in J} e_j((0,1))$. We define $V_0 = f(e_k(1))$, and also send $f$ to the basepoint if $V_0 = 0$. Otherwise we set $f_0 =  f \circ e_k$ so that $f_0$ is in $C\!\ST(V_0)$, and for the remaining intervals $j = \alpha(i)$ for $i \in \{1,\dots, n-1 \}$, we let $V_i = f(e_{\alpha(i)}(1))$ and $f_i = (f\circ e_{\alpha(i)}) \ominus V_{i-1}$, giving an element of $\Sigma\!\ST (V_i \ominus V_{i-1})$. Finally, we define $Cm(f) = (f_0,f_1,\dots, f_{n-1})$.

     We define $\Sigma m$ as the quotient of $Cm$ in which we additionally send $f$ to the basepoint if $f(t) \neq 0$ for all $t > e_k(0)$, in other words if the zero subspace does not show up in the resulting point in $\Sigma\!\ST(V_0)$. In particular, if $f$ has any cut points to the left of the leftmost open interval then it will be sent to the basepoint, as was the case for our comultiplication map in \cite[Construction 4.10]{KKMMW1}.
 \end{construction}

\begin{lemma}
    The map $\Sigma m$ is homotopic to the suspension of the map $m$ from \eqref{eq:comodule_one_point_in_D1J_EH}.
\end{lemma}

\begin{proof}
    Both $Cm$ and $\Sigma m$ are continuous by the same proof as in \cref{prop:coalg_comod_continuous}. The space $C\!\ST(V)$ is homeomorphic to the cone on $\ST(V)$, so the map $Cm$ is homotopic to the cone on the map $m$, relative to the base of the cone $\ST(V)$. This induces a homotopy between $\Sigma m$ and the suspension of $m$ after collapsing the base of the cone to $*$.
\end{proof}

\begin{remark}
    In \cref{constr:d_1-coalg_comod_suspended} we made the additional convention of sending $f$ to the basepoint when $V_0 = 0$. This is necessary in order to land in the wedge of chains of nonzero subspaces, which is consistent with the classical Dehn invariant as discussed in \cref{why_no_zero}. This convention is necessary when defining the comodule structure on $\Sigma\!\ST(V)$, but not necessary when defining that structure on $\ST(V)$ as in \cref{constr:d_1-coalg_comod}, because in configurations that represent points of $\ST(V)$, only nonzero subspaces of $V$ appear.
\end{remark}

\section{The multiplication on Euclidean and spherical scissors congruence }

The categories $\Dip^E$ and $\Dip^S$ are both symmetric monoidal, and this is the inherent reason that $\mathcal{E}$ and $\mathcal{S}$ have a multiplicative structure, comparable to the multiplicative structure on $\widetilde{\mathcal{S}}$ that was defined in \cite{KKMMW1}. Like the multiplicative structure on $\widetilde{\mathcal{S}}$, we show that the ones on $\mathcal{E}$ and $\mathcal{S}$ are also compatible with the comultiplication (i.e. the comodule structure). This makes $\mathcal{E}$ and $\mathcal{S}$ right $D_1$-comodule algebras over the $(\Comm,D_1)$-bialgebra $\widetilde{\mathcal{S}}$, as in Definition \ref{def-comodule-algebra}. Conceptually, this means that the coaction maps 
$$\mathcal{E} \to \mathcal{E}\wedge \widetilde{\mathcal{S}}\ \ \ \textup{and}\ \ \  \mathcal{S} \to \mathcal{S}\wedge \widetilde{\mathcal{S}}$$
are compatible with the multiplicative structures on the source and target.

\subsection{The multiplication on Euclidean scissors congruence}

In this section, we will use Lemma 6.14 of \cite{KLMMS-1} to define a commutative multiplication on the Euclidean spectral Sah functor $\mathfrak{S}^E$, and show that this multiplication behaves well with respect to the $\mathfrak{S}^{\tilde{S}}$-comodule structure. Namely, this multiplication makes $\mathfrak{S}^E$ into a $D_1$-comodule algebra over $\mathfrak{S}^{\tilde{S}}$.

Recall from Section 6 of \cite{KLMMS-1}:
\begin{definition}\label{def-PTX}
    Let $X,Y$ be Euclidean geometries. Define
    $$\ST^\times (X \times Y) = \frac{|\emptyset \subsetneq X' \times Y' \subseteq X \times Y|}{|\emptyset \subsetneq X' \times Y' \subsetneq X \times Y|}$$
    Here we are taking only subgeometries of $X \times Y$ which are of the form $X' \times Y'$, where $X'$ is a subgeometry of $X$ and $Y'$ is a subgeometry of $Y$.
\end{definition}

\begin{remark}\label{rem-PTX-vector}
    If $(U, x_0: U \to \mathbb{R})$ is the ambient vector space of $X$ and $(V, y_0: V \to \mathbb{R})$ is the ambient vector space of $Y$, we can also write this space as
    $$\ST^\times (U \times_\R V) = \frac{|\emptyset \subsetneq U' \times_\R V' \subseteq U \times_\R V|}{|\emptyset \subsetneq U' \times_\R V' \subsetneq U \times_\R V|}$$
    where we are taking only vector subspaces of the form $U' \times_\R V' \subseteq U \times_\R V$, where $U' \subseteq U$ is a vector subspace on which $x_0$ is rank 1, and $V' \subseteq V$ is a vector subspace on which $y_0$ is rank 1.
\end{remark}

We recall from Lemma 6.14 of \cite{KLMMS-1}:

\begin{lemma}\label{lem-6.14-klmms}
    There is a natural homeomorphism
    $$\ST^\times(U \times_\R V) \cong \ST(U) \wedge \ST(V)$$
\end{lemma}

With the natural forgetful map $\ST^\times(U \times_\R V) \to \ST(U \times_\R V)$, we obtain

\begin{corollary}\label{cor-extended-PTE-commalg}
    The Cartesian product (as in \cref{dip-cartesian}) makes $\ST : \Dip^E \to \Top_*$ into a commutative monoid under Day convolution.
\end{corollary}

\begin{corollary}\label{cor-lunaE-multandcomod}
    The functor $(\ST, *)_E: \Luna^E \to \topp$ is a commutative monoid under Day convolution. Furthermore, the maps $m^V _{J, k}$ from \cref{prop:coalg_comod_continuous} make $(\ST, *)_E$ into a right $D_1$-comodule over $(*, \Sigma \ST)_E: \Luna^E \to \topp$.
\end{corollary}

\begin{proof}
    Since $(\ST, *)_E(V) = *$ for $V \in \Dip^S$, the commutative monoid structure follows from \cref{cor-extended-PTE-commalg}. The right $D_1$-comodule structure follows from \cref{prop:dehn_multifunctor}.
\end{proof}

\begin{corollary}\label{cor-ESah-multandcomod}
    The functor $\mathfrak{S}^E$ is a commutative monoid under Day convolution, and in addition a right $D_1$-comodule over $\mathfrak{S}^{\tilde{S}}$.
\end{corollary}

\begin{proof}
    This follows from \cref{cor-lunaE-multandcomod} above and \cref{cor-D-symm-mon-Luna}.
\end{proof}

In order to check that the commutative multiplication and the comodule structure are compatible, we will use a step function description of the multiplication on $(\ST, *)_E: \Luna^E \to \topp$. For elements $U,V \in \Dip^E$, $f \in \Step([0,1],\Sub(U)_{\neq 0})/\sim$, and an element $g \in \Step([0,1],\Sub(V)_{\neq 0})/\sim$, their product $fg \in \Step([0,1],\Sub(U \times_\R V)_{\neq 0})/\sim$ is given pointwise by $(fg)(t) = f(t) \times_\R g(t)$.

\begin{proposition}\label{prop-GE-comodulealg}
    The functor $(\ST,*)_E: \Luna^E \to \topp$ is a right $D_1$-comodule algebra over $(*, \Sigma\!\ST)_E$. 
\end{proposition}

\begin{proof}
    As in the proof of Lemma 4.16 of \cite{KKMMW1}, verifying this boils down to checking that for every $U' \subseteq U \in \Dip^E$ and $V' \subseteq V \in \Dip^E$,
    $$(U \ominus U') \oplus (V \ominus V') = (U \times_\R V) \ominus (U' \times_\R V')$$
    Considering all of these as subspaces of $U\times_\R V$ (in which $U \ominus U'$ and $V \ominus V'$ are orthogonal to each other by definition), note that $(u,v) \in (U \times_\R V)_0 = U_0 \oplus V_0$ is orthogonal to every $(u',v') \in (U'  \times_\R V')_0 = U'_0 \oplus V'_0$ if and only if $u$ is orthogonal to every $u' \in U'_0$ and $v$ is orthogonal to every $v' \in V'_0$.
\end{proof}

\subsection{The multiplication on spherical scissors congruence}

We now show that the spectral Sah functor $\mathfrak{S}^S$ has a commutative multiplication as well, which is compatible with the $\mathfrak{S}^{\tilde S}$-comodule structure. 

As before, we consider the functor
$$\Sigma\!\PT_S:\Luna^S \to \topp.$$ We now make $\Sigma\!\PT_S$ into a commutative monoid in $\Fun(\Luna^S,\topp)$.
\begin{definition}\label{product-collapse}
    For $U, V$ in $\Luna^S$, define $p:S^U \wedge S^V \to S^{U \oplus V}$ to be the map of quotients that sends $(u,v) \in D(U) \times D(V)$ to the same point $(u,v) \in D(U\oplus V)$ if it lies in the subspace $D(U \oplus V)$, and otherwise sends it to the basepoint $*$. %
\end{definition}
\begin{construction}\label{constr:mult_spherical}
    Let $a:\Sigma\!\PT_S \boxwedge \Sigma\!\PT_S \to \Sigma\!\PT_S$ be the natural transformation with components
    $$a_V:\bigvee_{\substack{U \subseteq V \in \Dip^S}} \Sigma\!\PT(U) \wedge \Sigma\!\PT(V \ominus U) \to \Sigma\!\PT(V)$$
    for $V$ in $\Luna^S$, that sends $((f,v), (g,w))$ to $(f \oplus g, p(v,w) )$ where $f\oplus g$ is the ``pointwise sum'' of $f$ and $g$, and therefore a step function $I \to \Sub(U \oplus V)_{\neq 0}$. 

    For $I:\Luna^S \to \topp$ the unit for Day convolution,
    $$i:I \to \Sigma\!\PT$$ is the natural transformation with 
    $i_0:S^0 \to \Sigma\!\PT(0) = S^0 $ the identity map.  
\end{construction}
\begin{proposition}
    The maps defined in Construction \ref{constr:mult_spherical} make $\Sigma\!\PT_S$ into a commutative monoid. 
\end{proposition}
\begin{proof}
    This follows from the fact that $\oplus$ on $\Dip^S$ is commutative, the fact that $S^0$ is the unit for $\wedge$ in $\topp$, and the fact that the maps $p$ from \cref{product-collapse} are unital, associative, and commutative.
\end{proof}

\begin{proposition}
    The functor of multicategories $F:\mathcal{RM}_{D_1} \to \Fun(\Luna^S,\topp)^{\op}$ defined in Construction \ref{constr:d_1-coalg_comod} takes values in $\CMon(\Fun(\Luna^S,\topp))$.
\end{proposition}
\begin{proof}
For the maps
    $$D_1(J)_+ \to \Fun(\Luna^S,\topp)(\Sigma\!\ST_S, \Sigma\!\ST_S^{\boxwedge J})$$
    we already know that they actually land in 
     $\CMon(\Fun(\Luna^S,\topp))(\Sigma\!\ST_S,\Sigma\!\ST_S \boxwedge \Sigma\!\ST_S^{\boxwedge J \setminus \{k \}})$, because of \cite[Lemma 4.16]{KKMMW1}.
    We need to check that the same holds for the maps
    $$D_1(J)^k_+ \to \Fun(\Luna^S,\topp)( \Sigma\!\PT_S,\Sigma\!\PT_S \boxwedge \Sigma\!\ST_S^{\boxwedge J \setminus \{k \}}) $$
    given a finite set $J$ and $k\in J$.
    The proof of this is very similar.
    For $e = (e_j)_{j\in J}$ in $D_1(J)^k$ we need to see that its image $F(e): \Sigma\!\PT _S \to \Sigma\!\PT_S \boxwedge \Sigma\!\ST_S^{\boxwedge J \setminus \{k \}}$ is a morphism of monoids in $ \Fun(\Luna^S,\topp)$. In other words, we need to show that the diagram 
    \begin{center}
        \begin{tikzcd}
            \Sigma\!\PT_S \boxwedge \Sigma\!\PT_S  \arrow[rr, "F(e) \boxwedge F(e)"] \arrow[dd, "a"] & & (\Sigma\!\PT_S\boxwedge \Sigma\!\ST_S^{\boxwedge J \setminus \{k\}} ) \boxwedge (\Sigma\!\PT_S  \boxwedge \Sigma\!\ST_S^{\boxwedge J \setminus \{k\}} )  \arrow[d, "\cong"] \\
             & & \Sigma\!\PT_S  \boxwedge \Sigma\!\PT_S  \boxwedge (\Sigma\!\ST_S\boxwedge\Sigma\!\ST_S)^{\boxwedge J \setminus \{k\}} \arrow[d, "{a \boxwedge (\tilde a)^{\boxwedge J \setminus \{k\} }}"]\\
             \Sigma\!\PT_S  \arrow[rr, "F(e)"] &  & \Sigma\!\PT_S  \boxwedge \Sigma\!\ST_S^{\boxwedge J \setminus \{k\} }
        \end{tikzcd}
    \end{center}
commutes (as well as a diagram involving the units). Here $\tilde a$ refers to the multiplication $\Sigma\!\ST \boxwedge \Sigma\!\ST \to \Sigma\!\ST$ that is defined in \cite[Construction 4.5]{KKMMW1}. That this diagram commutes follows from the fact that for $A\subseteq B$ and $C \subseteq D$ all in $\Dip^S$, we have an isomorphism 
$(B \ominus A) \oplus (D \ominus C)  \cong (B \oplus D) \ominus (A \oplus C).$ See also Figure \ref{figure:proof_picture_1} and \ref{figure:proof_picture_2}.
\end{proof}

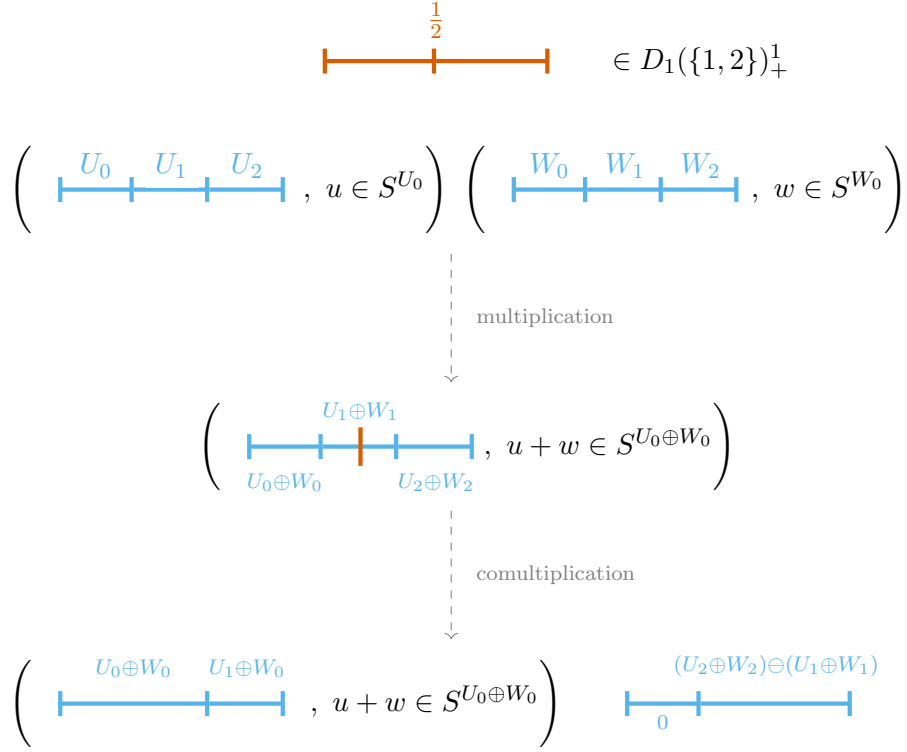
\begin{figure}[H]
\begin{center}
    \begin{tikzpicture}[yscale=1.7]

        \pgfmathsetmacro{\yzero}{-1}
        \pgfmathsetmacro{\yone}{-3}
        \pgfmathsetmacro{\ytwo}{-5}
    
        \draw[ultra thick,|-|,OI6] (3.5,0) -- (5,0) node[above=1.5mm]{$\frac{1}{2}$} ;
        \draw[ultra thick,-|,OI6] (5,0) -- (6.5,0) ;
        
        \node at (8.5,0) {$\in {D_1(\{1,2\})^1_+}$};

        \node at (-0.5,\yzero) {$\Bigg($};
        \draw[ultra thick,OI2,|-|] (0,\yzero) --node[above]{$U_0$} (1,\yzero);
        \draw[ultra thick,OI2,-|] (1,\yzero) --node[above,fill=white]{$U_1$} (2,\yzero);
        \draw[ultra thick,OI2,-|] (2,\yzero) --node[above]{$U_2$} (3,\yzero); 
        \node at (4.2,\yzero) {$, \ u\in S^{U_0} \Bigg)$};
        
        \node at (5.5,\yzero) {$\Bigg($};
        \draw[ultra thick,OI2,|-|] (6,\yzero) --node[above]{$W_0$} (7,\yzero); 
        \draw[ultra thick,OI2,-|] (7,\yzero) --node[above]{$W_1$} (8,\yzero); 
        \draw[ultra thick,OI2,-|] (8,\yzero) --node[above]{$W_2$} (9,\yzero); 
        \node at (10.2,\yzero) {$, \ w \in S^{W_0} \Bigg)$};

        \draw[dashed, gray,->] (5.2,\yzero-0.5)--node[right=2mm] {$\mathsmaller{\text{multiplication}}$} (5.2,\yone+0.5);
        
        \node at (2,\yone) {$\Bigg($};
        \draw[ultra thick,OI2,|-|] (2.5,\yone) --node[below=2mm]{$\mathsmaller{U_0\oplus W_0}$} (3.5,\yone);
        \draw[ultra thick,OI2,-|] (3.5,\yone) --node[above=2mm]{$\mathsmaller{U_1\oplus W_1}$} (4.5,\yone);
        \draw[ultra thick,OI2,-|] (4.5,\yone) -- node[below=2mm] {$\textcolor{OI2}{\mathsmaller{U_2 \oplus W_2}}$}(5.5,\yone); 
        \node at (7.3,\yone) {$, \ u+w\in S^{U_0\oplus W_0} \Bigg)$};
        \draw[ultra thick,OI6] (4,\yone+0.15)--(4,\yone-0.15);

        \draw[dashed, gray,->] (5.2,\yone-0.5)--node[right=2mm] {$\mathsmaller{\text{comultiplication}}$} (5.2,\ytwo+0.5);

          \node at (-0.5,\ytwo) {$\Bigg($};
        \draw[ultra thick,OI2,|-|] (0,\ytwo) --node[above=2mm]{$\mathsmaller{U_0\oplus W_0}$} (2,\ytwo);
        \draw[ultra thick,OI2,-|] (2,\ytwo) --node[above=2mm] {$\mathsmaller{U_1\oplus W_0}$} (3,\ytwo);
    
        \node at (5,\ytwo) {$, \ u+w\in S^{U_0\oplus W_0} \Bigg)$};
        
        \draw[ultra thick,OI2,|-|] (7.5,\ytwo) --node[below]{$\mathsmaller{0}$} (8.5,\ytwo); 
        \draw[ultra thick,OI2,-|] (8.5,\ytwo) --node[above=2mm]{$\mathsmaller{(U_2\oplus W_2)\ominus (U_1\oplus W_1)}$} (10.5,\ytwo); 
        
    \end{tikzpicture}
\end{center}
\caption{An illustration of $F(e)\circ a$ for $e  = (e_1,e_2) \in D_1(\{1,2\})_+^1$ on an element of $\Sigma \PT_S(U)\wedge \Sigma \PT(W)$ which includes into $(\Sigma \PT_S \protect\boxwedge\Sigma \PT_S)(V)$ for $V\cong U\oplus W$.}\label{figure:proof_picture_1}
\end{figure}

\newpage
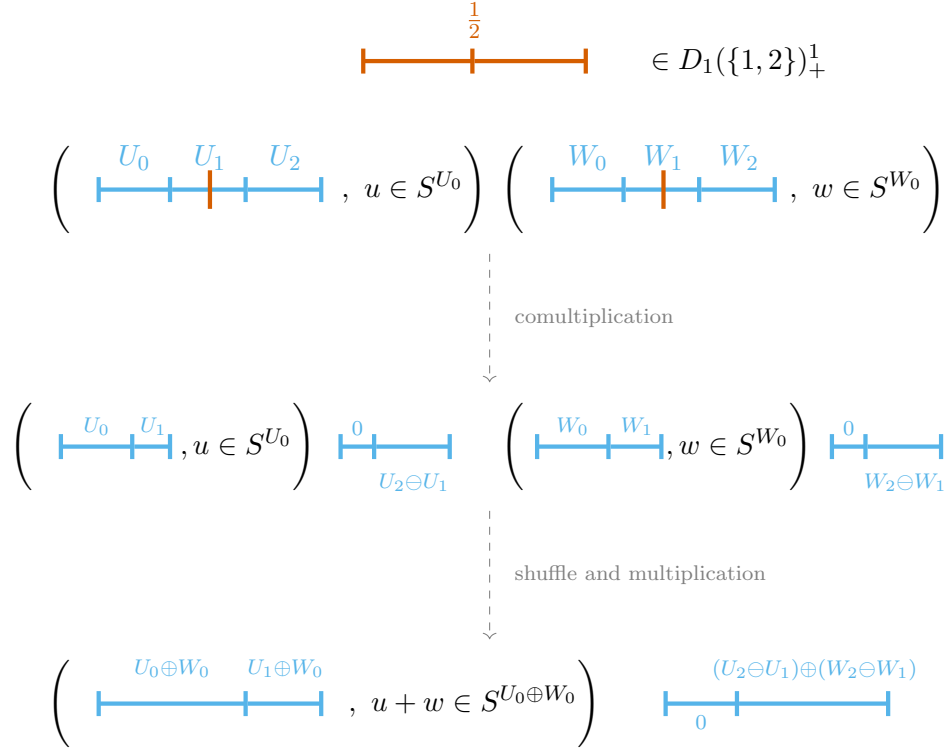
\begin{figure}[H]
\begin{center}
    \begin{tikzpicture}[yscale=1.7]

        \pgfmathsetmacro{\yzero}{-1}
        \pgfmathsetmacro{\yone}{-3}
        \pgfmathsetmacro{\ytwo}{-5}
    
        \draw[ultra thick,|-|,OI6] (3.5,0) -- (5,0) node[above=1.5mm]{$\frac{1}{2}$} ;
        \draw[ultra thick,-|,OI6] (5,0) -- (6.5,0) ;
        
        \node at (8.5,0) {$\in {D_1(\{1,2\})^1_+}$};

        
        \node at (-0.5,\yzero) {$\Bigg($};
        \draw[ultra thick,OI2,|-|] (0,\yzero) --node[above, yshift=3pt]{$U_0$} (1,\yzero);
        \draw[ultra thick,OI2,-|] (1,\yzero) --node[above,fill=white,yshift=3pt]{$U_1$} (2,\yzero);
        \draw[ultra thick,OI2,-|] (2,\yzero) --node[above,yshift=3pt]{$U_2$} (3,\yzero); 
        \node at (4.2,\yzero) {$, \ u\in S^{U_0} \Bigg)$};
        \draw[ultra thick,OI6] (1.5,\yzero-0.15)--(1.5,\yzero+0.15);
        
        \node at (5.5,\yzero) {$\Bigg($};
        \draw[ultra thick,OI2,|-|] (6,\yzero) --node[above,yshift=3pt]{$W_0$} (7,\yzero); 
        \draw[ultra thick,OI2,-|] (7,\yzero) --node[above,yshift=3pt]{$W_1$} (8,\yzero); 
        \draw[ultra thick,OI2,-|] (8,\yzero) --node[above,yshift=3pt]{$W_2$} (9,\yzero); 
        \node at (10.2,\yzero) {$, \ w \in S^{W_0} \Bigg)$};
        \draw[ultra thick,OI6] (7.5,\yzero-0.15)--(7.5,\yzero+0.15);

        \draw[dashed, gray,->] (5.2,\yzero-0.5)--node[right=2mm] {$\mathsmaller{\text{comultiplication}}$} (5.2,\yone+0.5);
        
        \node at (-1,\yone) {$\Bigg($};
        \draw[ultra thick,OI2,|-|] (-0.5,\yone) --node[above]{$\mathsmaller{U_0}$} (0.5,\yone);
        \draw[ultra thick,OI2,-|] (0.5,\yone) --node[above]{$\mathsmaller{U_1}$} (1,\yone);
        \node at (2,\yone) {$, u\in S^{U_0} \Bigg)$};
        
        \draw[ultra thick, OI2, |-|] (3.2,\yone)-- node[above]{$\mathsmaller{0}$} (3.7,\yone);
        \draw[ultra thick, OI2,-|] (3.7,\yone)-- node[below=2mm]{$\mathsmaller{U_2 \ominus U_1}$}(4.7,\yone);

        \node at (5.5,\yone) {$\Bigg($};
        \draw[ultra thick,OI2,|-|] (5.8,\yone) --node[above]{$\mathsmaller{W_0}$} (6.8,\yone);
        \draw[ultra thick,OI2,-|] (6.8,\yone) --node[above]{$\mathsmaller{W_1}$} (7.5,\yone);
        \node at (8.5,\yone) {$, w\in S^{W_0} \Bigg)$};
        
        \draw[ultra thick, OI2, |-|] (9.7,\yone)-- node[above]{$\mathsmaller{0}$} (10.2,\yone);
        \draw[ultra thick, OI2,-|] (10.2,\yone)-- node[below=2mm]{$\mathsmaller{W_2 \ominus W_1}$}(11.2,\yone);

        \draw[dashed, gray,->] (5.2,\yone-0.5)--node[right=2mm] {$\mathsmaller{\text{shuffle and multiplication}}$} (5.2,\ytwo+0.5);

        \node at (-0.5,\ytwo) {$\Bigg($};
        \draw[ultra thick,OI2,|-|] (0,\ytwo) --node[above=2mm]{$\mathsmaller{U_0\oplus W_0}$} (2,\ytwo);
        \draw[ultra thick,OI2,-|] (2,\ytwo) --node[above=2mm] {$\mathsmaller{U_1\oplus W_0}$} (3,\ytwo);
    
        \node at (5,\ytwo) {$, \ u+w\in S^{U_0\oplus W_0} \Bigg)$};
         
        \draw[ultra thick,OI2,|-|] (7.5,\ytwo) --node[below]{$\mathsmaller{0}$} (8.5,\ytwo); 
        \draw[ultra thick,OI2,-|] (8.5,\ytwo) --node[above=1.5mm]{$\mathsmaller{(U_2\ominus U_1)\oplus (W_2\ominus W_1)}$} (10.5,\ytwo); 
        
    \end{tikzpicture}
\end{center}
\caption{An illustration of $(a \protect \boxwedge \tilde a) \circ F(e)\protect \boxwedge F(e)$ on the same element of $\Sigma \PT_S(U)\wedge \Sigma \PT(W)$ as in \Cref{figure:proof_picture_1}. The outcome is identical to the outcome in the figure above since $(U_2\ominus U_1)\oplus (W_2\ominus W_1) \cong (U_2 \oplus W_2)\ominus (U_1\oplus W_1)$.}\label{figure:proof_picture_2}
\end{figure}

\section{Lifting to \texorpdfstring{$\infty$}{infty}-categories}\label{sect:infinity}

Finally we show that the spectrum version of the algebra and coalgebra structures lift to the $\infty$-category of spectra.

\begin{remark}
For the proof of the main result of our previous article \cite{KKMMW1}, we needed to be able to translate point-set models for algebras over operads into the respective $\infty$-categorical notions, which we took care of in \cite{kkmmw0}. Similar translations are needed in this paper, but it turns out that the results of \cite{kkmmw0} suffice: although not stated in this generality, the proof of Theorem 3.5 of loc.\ cit. works verbatim for a simplicial multicategory $\mathcal{O}_{\bullet}$ and hence in particular for the simplicial right module operad $\RM$. The latter can, for instance, be defined as singular simplices of the topological right module multicategory and is thus a fibrant simplicial operad.

This also yields a more general version of Theorem 5.14 of \cite{kkmmw0}, by replacing the first half of its proof by the version for multicategories. In particular, there is a canonical interchange map
 \[ 
    	N^s(\Alg_{\RM}(\Alg_{\Comm}(M^\flat_\bullet)^{\op})^{\op})
			\to 
		\Alg_{N^s(\RM)}(\CAlg(\mathcal{M})^{\op})^{\op} \simeq \Alg_{\mathcal{R}\mathcal{M}}(\CAlg(\mathcal{M})^{\op})^{\op}
	\]
    which preserves the product and coaction in the homotopy category $hM^\flat_\bullet$. Here $\mathcal{R}\mathcal{M}$ is the right module $\infty$-operad as in Section 4.2.1 of \cite{HA} and the rest is the notation from just above Theorem 5.14 of \cite{kkmmw0}.
\end{remark}

This finishes the proof of \Cref{mainthm} for all geometries except for the hyperbolic one. We will take care of the latter next.

As explained in Section 5, we do not expect to have a symmetric monoidal point-set model for $(\Sigma D, D)_H$ and will instead define it within the framework of $\infty$-categories. More explicitly, we want to construct a lax symmetric monoidal functor
$$(\Sigma D, D)_H \colon \Fun(\Luna^H, \Top_*) \to \Fun(\Luna^H,\Spp)$$
whose structure maps $F(X) \boxwedge F(Y) \to F(X \boxwedge Y)$ are isomorphisms whenever $X,Y \in \Dip^S \amalg \Dip^H$. Here $\Spp$ is the $\infty$-category of spectra. 
\begin{construction}
We start with the functor
$$(\Sigma D, D)_H \colon \Fun(\Luna^H, \Top_*) \to \Fun(\Luna^H,(\Sp^O)^{\cof})$$ defined by $G \mapsto (V \mapsto Q\Sigma F(S^V,\mathbb{S}) \wedge G(V))$ on $\Dip^H$ and by $G \mapsto (V \mapsto QF(S^V,\mathbb{S}) \wedge G(V))$ on $\Dip^S$, where $Q$ is the cofibrant replacement functor. All the functors $\Luna^H \to \Top_*$ we care about factor through the inclusion $\Top_*^{\cof} \hookrightarrow \Top_*$, so that we might as well consider the symmetric monoidal category $\Fun(\Luna^H, \Top_*^{\cof})$ of pointwise cofibrant functors instead. This gives us the desired functor 
$$(\Sigma D, D)_H \colon \Fun(\Luna^H, \Top_*^{\cof}) \to \Fun(\Luna^H,\Spp)$$
by composing with the Dwyer-Kan localization $\Fun(\Luna^H, (\Sp^{O})^{\cof}) \to \Fun(\Luna^H, \Spp)$.
\end{construction}
\begin{proposition}
The functor $(\Sigma D, D)_H \colon \Fun(\Luna^H, \Top_*^{\cof}) \to \Fun(\Luna^H,\Spp)$ is lax symmetric monoidal, and the structure maps $F(X) \boxwedge F(Y) \to F(X \boxwedge Y)$ are isomorphisms whenever $X,Y \in \Dip^S \amalg \Dip^H$.
\end{proposition}
\begin{proof}
By \cref{cor-D-symm-mon-Luna}, the functor
$$(\Sigma D, D)_H: \Fun(\Luna^H, \topp^{\cof}) \to \Fun(\Luna^H, (\Sp^O)^{\cof} )$$
is lax symmetric monoidal with structure maps that are weak equivalences whenever $X,Y \in \Dip^S \amalg \Dip^H$; cofibrant replacement does not change this, since cofibrant replacement is lax symmetric monoidal and replaces each object by a weakly equivalent cofibrant object. Passing to $\infty$-categories, we obtain a lax symmetric monoidal functor of $\infty$-categories
$$(\Sigma D, D)_H: \Fun(\Luna^H, \Spc_*) \to \Fun(\Luna^H, \Spp)$$
where $\Spc_*$ is the $\infty$-category of pointed spaces. Precomposing with the Dwyer--Kan localization $\Fun(\Luna^H, \topp^{\cof}) \to \Fun(\Luna^H, \Spc_*)$, which is strong symmetric monoidal by Proposition 4.6 of \cite{kkmmw0}, we obtain the desired result.
\end{proof}

\begin{corollary}\label{cor-SahH-comod}
    The hyperbolic Sah spectrum $\mathcal{H}$ is a right $\widetilde{\mathcal{S}}$-comodule in the $\infty$-category of spectra.
\end{corollary}

\begin{proof}
    By \cref{prop:dehn_multifunctor}, $(\ST, *)_H: \Luna^H \to \topp$ is a $D_1$-comodule over $(*, \Sigma \ST)_H$. Applying $(\Sigma D, D)_H$ to both of these functors gives the spectral Sah functors $\funSah^H$ and $\funSah^{\tilde S}$ respectively. Since these lie in a subcategory on which $(\Sigma D, D)_H$ is a strong symmetric monoidal functor of $\infty$-categories, $\funSah^H$ is thus a comodule over $\funSah^{\tilde S}$. Taking colimits to land in the $\infty$-category of spectra, we obtain the conclusion.
\end{proof}

This finishes the proof of \Cref{mainthm} in the hyperbolic case as well.

\section{The comodule algebra structure on apartments}\label{sec:apartments}

In this section we show that our constructions on the scissors congruence spectra from the last few sections recover on homology the classical Dehn invariant and product maps from \cref{sec:classical_structures}, proving \cref{mainthm_pi0}.

To do this we recall the notion of an apartment in $\Sigma\!\ST(\R^n)$ from \cite{KKMMW1}, and show how to generalize it to an apartment in the polytopal Tits complex following \cite{scissors_thom} and \cite{KLMMS-2}.

\subsection{Apartments in $\ST(E^n)$ and $\ST(H^n)$}
We recall how the space $\PT(X)$ is always a wedge of spheres whose dimension matches the dimension of $X$, with the polytope module $\Pt(X)$ giving an explicit description of which spheres show up. We do this for Euclidean and hyperbolic geometry first, because our conventions for spherical geometry will differ.

\begin{definition}\label{apartment_1}
    Let $X^n$ be a Euclidean or hyperbolic geometry of dimension $n \geq 0$ with ambient vector space $V$. For each $(n+1)$-tuple of linearly independent vectors $(v_0,\ldots,v_n)$ in $X^n \subseteq V$, the corresponding \emph{suspended apartment} is defined as the map of quotients
    \[ 
    	\frac{\Delta^n}{\partial \Delta^n} 
			\to
		\frac{ |0 \subsetneq U \subseteq V| }{ |0 \subsetneq U \subsetneq V| }
            =
        \ST(V)
	\]
    induced by the map of simplicial complexes
    \(
    	\sd(\Delta^{n}) \to |0 \subsetneq U \subseteq V|
	\) 
    that on vertices is given by
    \[ 
    	(\varnothing \subsetneq S \subseteq \{v_0,\ldots,v_n\}) 
			\mapsto 
		\langle S \rangle.
	\]
   
    Here $\langle S \rangle$ denotes the linear span of $S$.
    More generally, for any convex polytope $P \subseteq X^n$, we call a map $P \to |0 \subsetneq U \subseteq V|$ a \emph{suspended apartment-like map} if it sends each face $F \subseteq P$ to the realization of the smaller poset $|0 \subsetneq U \subseteq \langle F \rangle|$. Any suspended apartment-like map induces a map from the quotient
    \[ 
    	\frac{P}{\partial P} 
			\to
		\frac{ |0 \subsetneq U \subseteq V| }{ |0 \subsetneq U \subsetneq V| }
            =
        \ST(V).
	\]
    These definitions are consistent in the sense that the standard apartment map for $P = \Delta^n$ is apartment-like.
\end{definition}

The following form of the Solomon-Tits theorem can be found in \cite[Corollary 2.11]{dupont_book} or \cite[Theorem B]{KLMMS-2}.

\begin{proposition}\label{solomon_tits_w_apts_1}
    If $X^n = E^n$ or $H^n$, the space $\ST(X^n)$ is homotopy equivalent to a wedge sum of $n$-spheres, and choosing one suspended apartment-like map for each generator of $\Pt(X^n)$ gives an isomorphism
    \[ 
    	[\apt]\colon \Pt(X^n) \cong H_n(\ST(X^n),*).
	\]
\end{proposition}
The isomorphism in this proposition is only well-defined up to sign. It depends on fixing a choice of orientation for $X^n$, which induces an orientation on all of the spheres $P/\partial P$ for $P \subseteq X^n$.

To analyze the effect of the Dehn invariant, we will have to suspend these apartments one more time. We only do this for the standard apartments on simplices $\Delta^n$, not the more general apartments for any convex polytope $P$.
\begin{definition}\label{apartment_2}
    Let $X^n$ be a Euclidean or hyperbolic geometry of dimension $n \geq 0$ with ambient vector space $V$. For each $(n+1)$-tuple of linearly independent vectors $(v_0,\ldots,v_n)$ in $X^n \subseteq V$, the corresponding \emph{doubly-suspended apartment} is the map of quotients
    \[ 
    	\frac{I^{n+1}}{\partial I^{n+1}} 
			\to 
		\frac{| 0 \subseteq U \subseteq V |}%
		{| 0 \subseteq U \subsetneq V | \cup | 0 \subsetneq U \subseteq V |}
            =
        \Sigma\!\ST(V)
	\]
    induced by the map of posets $\underline{2}^{n+1} \to \{0 \subseteq U \subseteq V\}$ that sends
    \[
    	(\varnothing \subseteq S \subseteq \{v_0,\ldots,v_n\}) 
			\mapsto 
		\langle S \rangle,
	\]
    where again $\langle S \rangle$ denotes the linear span of $S$.
    
    When writing this out in coordinates, we adopt the convention that $(0,...,0)$ corresponds to the terminal element of the poset $\underline{2}^{n+1}$, $(1,...,1)$ corresponds to the initial element, and more generally the vertex corresponding to $S \subseteq \{v_0,\ldots,v_n\}$ is the vertex in which the coordinate $c_i = 0$ when $v_i \in S$ and $c_i = 1$ when $v_i \not\in S$. This convention allows us to illustrate each doubly-suspended apartment by taking each point $(c_0,\ldots,c_n) \in I^{n+1}$, drawing an interval with a cut point at $c_i$ labeled by the vector $v_i$, and then labeling each interval by the span of the vectors $v_i$ to its left, as in \cite[Proposition 7.2]{KKMMW1} and illustrated in \Cref{double_apt_figure}.
\end{definition}
\begin{figure}[h]
\begin{center}
\resizebox{\textwidth}{!}{ 
\begin{tikzpicture}[scale=2]
	\draw[thick] 
		(0,0) node[anchor=north east]{$(0,0)$}
			-- 
		(1,0) node[anchor=north west]{$(1,0)$} 
			-- 
		(1,1) node[anchor=south west]{$(1,1)$} 
			-- 
		(0,1) node[anchor=south east]{$(0,1)$} 
			-- 
		cycle;
        \draw[thick] (0,0) -- (1,1);
        \draw[thick,fill,OI2] 
        	(1/3,2/3) circle (0.02) node[above]{$a$};
        \draw[thick,fill,OI1] 
        	(1/3,1/3) circle (0.02) node[above]{$b$};
        \draw[thick,fill,OI3] 
        	(4/5,1/5) circle (0.02) node[above]{$c$};
	
        \begin{scope}[xshift=2cm,scale=1.5]
        	\draw[ultra thick,OI3]
				(0,0) 
					edge[|-|] 
				(1/5,0) 
					edge[-|] 
				(4/5,0) 
					edge[-|] 
				(1,0);
			\node[scale=0.75] at (1/5,0.1) {$v_1$};
			\node[scale=0.75] at (4/5,0.1) {$v_0$};
			\node[OI3] at (-0.15,0) {$c = $};
			\draw[thick,|->] (1.1,0) -- (1.4,0);
			\draw[ultra thick,OI3,]
				(1.5,0) 
					edge[|-|] 
				(1.7,0)
					edge[-|] 
				(2.3,0)
					edge[-|]
				(2.5,0);
			\node[scale=0.75] at (1.6,0.1) {$0$};
			\node[scale=0.75] at (2,0.1) {$\langle v_1 \rangle$};
			\node[scale=0.75] at (2.4,0.1) 
				{$\langle v_0,v_1 \rangle$};
        \end{scope}

        \begin{scope}[xshift=2cm,scale=1.5,yshift=0.33cm]
        	\draw[ultra thick,OI1]
				(0,0) 
					edge[|-|] 
				(1/3,0) 
					edge[-|]  
				(1,0);
			\node[scale=0.75] at (1/3,0.1) {$v_0,v_1$};
			\node[OI1] at (-0.15,0) {$b = $};
			\draw[thick,|->] (1.1,0) -- (1.4,0);
			\draw[ultra thick,OI1]
				(1.5,0) 
					edge[|-|] 
				({1.5+1/3},0)
					edge[-|] 
				(2.5,0);
			\node[scale=0.75] at (5/3,0.1) {$0$};
			\node[scale=0.75] at (13/6,0.1) {$\langle v_0,v_1 \rangle$};
        \end{scope} 
        
        \begin{scope}[xshift=2cm,scale=1.5,yshift=0.66cm]
        	\draw[ultra thick,OI2]
				(0,0) 
					edge[|-|] 
				(1/3,0) 
					edge[-|] 
				(2/3,0) 
					edge[-|] 
				(1,0);
			\node[scale=0.75] at (1/3,0.1) {$v_0$};
			\node[scale=0.75] at (2/3,0.1) {$v_1$};
			\node[OI2] at (-0.15,0) {$a = $};
			\draw[thick,|->] (1.1,0) -- (1.4,0);
			\draw[ultra thick,OI2,]
				(1.5,0) 
					edge[|-|] 
				({1.5+1/3},0)
					edge[-|] 
				({1.5+2/3},0)
					edge[-|]
				(2.5,0);
			\node[scale=0.75] at (5/3,0.1) {$0$};
			\node[scale=0.75] at (2,0.1) {$\langle v_0 \rangle$};
			\node[scale=0.75] at (7/3,0.1) 
				{$\langle v_0,v_1 \rangle$};
        \end{scope} 
    \end{tikzpicture}
}
\end{center}
\caption{Illustration of a doubly-suspended apartment for $[v_0,v_1]$. The corner $(0,0)$ is sent to the constant configuration on $\langle v_0,v_1 \rangle = V$ and the corner $(1,1)$ is sent to the constant configuration on $\langle \varnothing\rangle = 0$.}\label{double_apt_figure}
\end{figure}
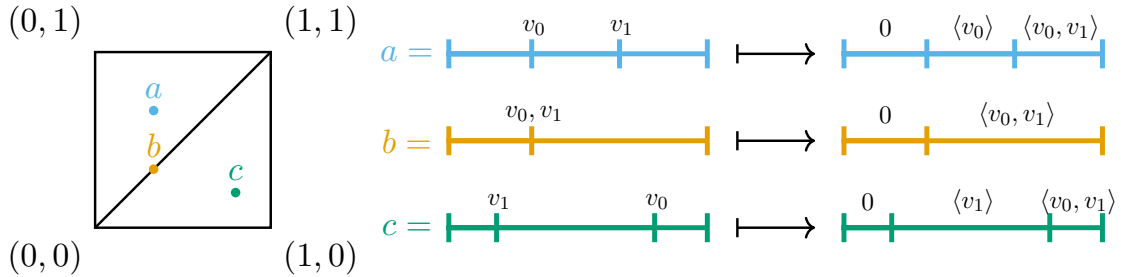

It follows from \cite[Lemma 5.8]{KKMMW1} that the doubly-suspended apartment is indeed the suspension of the suspended apartment maps, and therefore induces an isomorphism
\[ 
    [\apt]\colon \Pt(X^n) \cong H_{n+1}(\Sigma\!\ST(X^n),*).
\]
We would like this to be an $I(X^n)$-equivariant isomorphism, or equivalently an isomorphism of diagrams over $\Dip^X$. This cannot be true as stated because reflections that preserve a polytope $P$ act as the identity on $[P] \in \Pt(X^n)$, but apply a minus sign to the resulting class in $H_{n+1}(\Sigma\!\ST(X^n),*)$. In \cite{KKMMW1} we resolve this issue by forming an equivariant isomorphism out of a different group called $\Ls(V)$, and then passing to $\Pt(V)$ by choosing orientations. Here we resolve the issue by adding a desuspension to $\ST(V)$ itself.

\begin{construction}\label{apartment_oriented}
    We make the doubly suspended apartment map of \cref{apartment_2} respect orientation by viewing $I^{n+1}$ as embedded into $V$, via the map  $I^{n+1} \to V$ that sends each point $(c_0,\ldots,c_n) \in I^{n+1}$ to the linear combination $(1-c_0)v_0 + \cdots + (1-c_n)v_n \in V$. This embedding gives rise to a map $S^V = V \cup \{\infty \}\to \frac{I^{n+1}}{\partial I^{n+1}}$ by sending each point in $I^{n+1}\subseteq V$ to itself and collapsing the rest. Thus we get a map
    \begin{equation}\label{eq:apt_oriented_double}
    S^V = V \cup \{\infty\} \to \frac{I^{n+1}}{\partial I^{n+1}} \to \Sigma\!\ST(V)
    \end{equation}
    This gives an element of $H_0(\Omega^V\Sigma\!\ST(V),*)$ instead of $H_{n+1}(\Sigma\!\ST(V),*)$, where $\Omega^V X = \textup{Map}_*(S^V,X)$.\footnote{We could have taken $\sum_i c_iv_i$ instead of $\sum_i (1-c_i)v_i$, giving a convention that differs from this one by a factor of $(-1)^{\dim V}$, and everything would also be consistent. This convention is geometrically a little more natural because it places the vertex corresponding to $S \subseteq \{v_0,\ldots,v_n\}$ in the span $\langle S \rangle$.}
    
    In the Euclidean case we can also do this to any singly suspended apartment-like map: let $T:V \to V$ be a translation that takes $E^n = x_0^{-1}(1)$ to $V_0 = x_0^{-1}(0)$, then $T(P)$ is embedded in $V_0$, and collapsing everything outside of $T(P)$ gives a map
\begin{equation}\label{eq:apt_oriented_single}
    S^{V_0} \to \frac{P}{\partial P} \to \ST(V).
    \end{equation}
   In the case that $P\cong \Delta^n$ is a simplex, by translating the embedding $T(P) \to V_0$ back to $E^n$ and coning off to zero, we can observe that the suspension of \eqref{eq:apt_oriented_single} agrees up to homotopy with \eqref{eq:apt_oriented_double}. 

    Here we are using the identification $\Sigma S^{V_0} \cong S^{\R \oplus V_0} \cong S^V$ coming from any isomorphism of vector spaces $V \cong \R \oplus V_0$ that is the identity on $V_0$ and the identity after projecting to $\R$; these conditions are enough to get the sign right.
\end{construction}

\begin{proposition}\label{solomon_tits_w_apts_2}
    If $X^n = E^n$ or $H^n$, the oriented doubly suspended apartment maps for each generator of $\Pt(V)$ induce an $I(V)$-equivariant isomorphism
    \[ 
    	[\apt]\colon \Pt(V) \cong \pi_0(\Omega^V\Sigma\!\ST(V)) \cong H_0(\Sph^{-V} \sma \Sigma\!\ST(V)).
	\]
    In other words, they define an isomorphism of functors $\Dip^X \to \Ab$. In the Euclidean case the oriented singly suspended apartment-like maps also give an $I(V)$-equivariant isomorphism
    \[ 
    	[\apt]\colon \Pt(V) \cong \pi_0(\Omega^{V_0}\ST(V)) \cong H_0(\Sph^{-V_0} \sma \ST(V)).
	\]
    These isomorphisms agree along the map $\Omega^{V_0}\ST(V) \to \Omega^V\Sigma\!\ST(V)$ that applies the unit map $Y \to \Omega\Sigma Y$ and identifies $\Sigma S^{V_0} \cong S^V$ as in \cref{apartment_oriented}.
\end{proposition}

\begin{proof}
    In each case, the map is still an isomorphism because it differs from the map in \cref{solomon_tits_w_apts_1} by an isomorphism. But now we can check directly from \cref{apartment_oriented} that it is $I(V)$-equivariant. (In the second case we also observe that the choice of translation to $V_0$ has no effect on this map on homology.) For the final compatibility statement we just check on a single simplex.
    
    Up to homotopy, the identification $S^V \to \Sigma S^{V_0}$ can be changed to one that collapses away all points with $x_0 \leq 0$ or $x_0 \geq 1$, and then sends the remaining points of the form $tv$ with $t \in [0,1]$ and $v \in V_0$ to the pair $(t,v) \in [0,1] \times V_0$. (It is enough to observe that they have the same degree.) After this modification, the suspension of \eqref{eq:apt_oriented_single} becomes a collapse onto the cone $C(P) \cong I^{n+1}$ followed by the map $\frac{I^{n+1}}{\partial I^{n+1}} \to \Sigma\!\ST(V)$ from \cref{apartment_2}. The straight-line homotopy between the embedding $C(P) \to V$ and the embedding defined in \cref{apartment_oriented} gives a homotopy between the collapse maps. Therefore suspension of \eqref{eq:apt_oriented_single} is homotopic to \eqref{eq:apt_oriented_double}.
\end{proof}

\subsection{Apartments in $\PT(S^{n-1})$ and $\Sigma\!\PT(S^{n-1})$}
In this case, apartments are more difficult to define explicitly, so we first recall the notion of apartment-like map from \cite{KLMMS-2}.

\begin{definition}\label{apartment_3}
    Let $V$ be an $n$-dimensional inner product space where $n \geq 1$, let $S(V)$ be the unit sphere in $V$, and let $P = [v_1,\dotsc,v_n]$ be a geodesic simplex in $S(V)$. A \textit{suspended apartment-like map} is a map
    $$a \colon P \to \underset{0 \subseteq U \subseteq V}\hocolim \ S(U)$$
    such that for every subset $M \subseteq \lbrace v_1,\dotsc, v_n \rbrace$, the image $a([M])$ is contained in the subspace 
    
    $$\underset{0 \subseteq U \subseteq \langle M \rangle}\hocolim \ (S(U) \cap [M]) 
        =
    \underset{0 \subseteq U \subseteq \langle M \rangle}\hocolim \ (S(U) \cap P).$$
\end{definition}
By \cite[Lem 5.9]{KLMMS-2}, the space of suspended apartment-like maps is weakly contractible, so the choice of a suspended apartment-like map is essentially unique. We next extend this definition to the suspended polytopal Tits complex:
\begin{definition}\label{apartment_4}
    Let $V$ be an $n$-dimensional inner product space where $n \geq 0$. Let $P = [v_1,\dotsc,v_n]$ be a geodesic simplex in $S(V)$ and let $C(P) = \{0\} * P \subseteq D(V)$ be the subspace formed by all scalar multiples of points in $P$ by scalars $t \in [0,1]$. This is a cone on $P$ with cone point at the origin. A \textit{doubly suspended apartment-like map} is a map
    $$a \colon C(P) \to \underset{0 \subseteq U \subseteq V}\hocolim \ D(U)$$
    such that the base of the cone $P \subseteq C(P)$ is sent into the homotopy colimit of the spheres $S(U)$ and for each subset  $M \subseteq \lbrace v_1,\dotsc, v_n \rbrace$, the image $a(C([M]))$ is contained in the subspace 
    $$\underset{0 \subseteq U \subseteq \langle M \rangle}\hocolim \ (D(U) \cap C([M]))
        =
    \underset{0 \subseteq U \subseteq \langle M \rangle}\hocolim \ (D(U) \cap C(P)).$$
    This includes the empty face, where the condition tells us that the cone point $0 \in C(P)$ must go to the unique point in $D(0)$.
\end{definition}

\begin{lemma}
    The space of doubly suspended apartment-like maps is weakly contractible.
\end{lemma}

\begin{proof}
    This is the same as in \cite[Lem 5.9]{KLMMS-2}.
\end{proof}

Every doubly-suspended apartment-like map induces a map of quotients
\[ 
    \Sigma(P/\partial P) \cong \frac{C(P)}{\partial C(P)} 
		\to 
	\frac{ \underset{0 \subseteq U \subseteq V}\hocolim \ D(U) }{ \left( \underset{0 \subseteq U \subsetneq V}\hocolim \ D(U) \right) \cup \left( \underset{0 \subseteq U \subseteq V}\hocolim \ S(U) \right) }
        \cong
     \frac{ \underset{0 \subseteq U \subseteq V}\bhocolim \ D(U)/S(U) }{ \underset{0 \subseteq U \subsetneq V}\bhocolim \ D(U)/S(U) }.
\]
This last term is our definition of $\Sigma\!\PT(V)$ from \cref{spt_simplify_spherical}. Notice that the formula for the doubly-suspended apartment makes sense even when $n = 0$, where the only polytope is $P = \varnothing$, $C(P) = \{0\}$, $C(P)/\partial C(P) = S^0$, and the map to $\Sigma\!\PT(\R^0) = S^0$ is the identity map $S^0 \to S^0$.

By \cref{spt_simplify_spherical}, this model of $\Sigma\!\PT(V)$ has a contractible subcomplex given by the hocolim of all the origins in the discs $D(U)$, and the quotient by this complex is homeomorphic to the suspension of $\PT(V)$:
\begin{equation}
    \Sigma\!\!\left(\frac{ \underset{0 \subseteq U \subseteq V}\hocolim \ S(U) }{ \underset{0 \subseteq U \subsetneq V}\hocolim \ S(U) }\right)
        \cong
     \frac{ \underset{0 \subseteq U \subseteq V}\bhocolim \ D(U)/S(U) }{ \Bigl( \underset{0 \subseteq U \subsetneq V}\bhocolim \ D(U)/S(U)\Bigr) \cup \Bigl(\underset{0 \subseteq U \subsetneq V}\hocolim \ \{0\}\Bigr)}
     \overset{\sim}\leftarrow
     \frac{ \underset{0 \subseteq U \subseteq V}\bhocolim \ D(U)/S(U) }{ \underset{0 \subseteq U \subsetneq V}\bhocolim \ D(U)/S(U) }
\end{equation}

\begin{lemma}\label{single_to_double_apt_maps}
    Let $n \geq 1$. For every suspended apartment-like map 
    \[
        a\colon P \to \underset{0 \subseteq U \subseteq V}\hocolim \ S(U),
    \]
    the map 
    \[
        C(a)\colon C(P) \to \underset{0 \subseteq U \subseteq V}\hocolim \ D(U)
    \]
    obtained by coning off to $0$ is a doubly suspended apartment-like map, and the resulting maps from $\Sigma(P/\partial P)$ to the equivalent spaces from above agree up to homotopy.
\end{lemma}

\begin{proof}
    We simply check that both recipes applied to the same suspended apartment-like map $a\colon P \to \underset{0 \subseteq U \subseteq V}\hocolim \ S(U)$ produce the same map from $C(P)$ to the numerator of the space in the middle, $\underset{0 \subseteq U \subseteq V}\bhocolim \ D(U)/S(U)$.
\end{proof}

\begin{proposition}\label{solomon_tits_w_apts_3}
    For any inner product space $V$ of dimension $n \geq 0$, the space $\Sigma\!\PT(V)$ is homotopy equivalent to a wedge sum of $n$-spheres, and the doubly suspended apartment-like maps for each generator of $\Pt(V)$ induce an isomorphism
    \[ 
    	[\apt]\colon \Pt(V) \cong H_n(\Sigma\!\PT(V),*).
	\]
\end{proposition}

\begin{proof}
    When $n = 0$ this follows directly from the fact that $\Pt(0) \cong \Z$ and the unique doubly-suspended apartment-like map induces the identity $S^0 \to S^0$. When $n \geq 1$, by \cite[Thm 5.14]{KLMMS-2}, $\PT(V)$ is a wedge of $(n-1)$-spheres and the singly suspended apartment-like maps induce an isomorphism $\Pt(V) \cong H_{n-1}(\PT(V),*)$. By \cref{single_to_double_apt_maps} when we suspend again the generators are therefore given by the doubly suspended apartment maps.
\end{proof}

\begin{construction}\label{apartment_5}
    We make an explicit choice of doubly suspended apartment-like map for $\Sigma\!\PT(V)$ for the element $[v_1,\ldots,v_n] \in \Pt(V)$. Again we think of $P$ as a simplex and $C(P)$ as a cube, with the convention that $(0,...,0)$ is the barycenter of $P$ and $(1,...,1)$ is the cone point. So each vertex of the cube corresponds to a set $S \subseteq \{v_1,\ldots,v_n\}$, but the coordinates of the vertex are 0 for each element of $S$ and 1 for each element of $S^c$. The faces where one coordinate is 0 are in $P$, and the faces in which one coordinate is 1 together form the cone on the boundary, $C(\partial P)$.
    
    A map $f\colon I^n \to \underset{0 \subseteq U \subseteq V}\hocolim \ D(U)$ is apartment-like if each face where $c_i = 0$ goes to the homotopy colimit of the boundary spheres
    \[ \underset{0 \subseteq U \subseteq V}\hocolim \ S(U) \cap P, \]
    and each face where $c_i = 1$ goes to the subspace
    \[ \underset{0 \subseteq U \subseteq \langle v_1, \ldots, \hat{v_i}, \ldots, v_n\rangle}\hocolim \ D(U \cap \langle v_1, \ldots, \hat{v_i}, \ldots, v_n\rangle) \cap C(P). \]
    In other words, the subspace of $D(U)$ that is the cone on the $i$th face of $P$.

    We define our map
    \[ f\colon I^n \to \underset{0 \subseteq U \subseteq V}\hocolim \ D(U) \]
    by defining its two coordinates separately, one of which is a map $f_1\colon I^n \to |0 \subseteq U \subseteq V|$ taking each point $t \in I^n$ to a configuration of cut points and subspaces of $V$, the other of which is a map $f_2\colon I^n \to D(V)$ such that $f_2(x)$ always lies in $D(U)$ where $U$ is the subspace at the left-hand side of the configuration $f_1(x)$.

    To define $f_1$, we start with the map $I^n \to |0 \subseteq U \subseteq V|$ from \cref{apartment_2} that creates a configuration of cut points in which each subspace is the span of the vectors $v_i$ to its left. Compose this with the map $I^n \to I^n$, $c_i \mapsto \max(0,(3c_i-1)/2)$, which in each coordinate direction collapses the sub-interval $[0,1/3]$ to 0 and expands the sub-interval $[1/3,1]$ to fill the whole interval $[0,1]$. Let $f_1\colon I^n \to I^n \to |0 \subseteq U \subseteq V|$ denote the composition. Visually, it has the same formula as \cref{apartment_2} except that everything to the left of 1/3 is collapsed to the left-hand side of the interval first.
    
    To define $f_2$, we pick any $R > 0$ such that $D(V)$ is contained in the parallelepiped of all vectors of the form $\sum_i c_i v_i$, $|c_i| \leq R$. Pick any homeomorphism $\phi\colon [0,1] \to [0,R]$ with $\phi(0) = R$ and $\phi(1) = 0$, then define $\psi\colon [0,1] \to [0,R]$ by the rule $\psi(t) = \phi(3t)$ when $t \leq 1/3$ and $\psi(t) = 0$ when $t \geq 1/3$. So $\psi$ goes down from $R$ to 0 over the interval $[0,1/3]$ and then stays at 0.
    
    Define the map $I^n \to V$ by $(c_1,\ldots,c_n) \mapsto \sum_i \psi(c_i)v_i$. Compose this with $V \to D(V)$ by sending each point in $V$ back to its nearest point in $D(V)$, so that everything outside $D(V)$ is sent to $S(V)$, and call the composition $f_2\colon I^n \to D(V)$. Visually, it is sending $I^n$ to the cube of all linear combinations of $\{v_1,\ldots,v_n\}$ with coefficients between 0 and $R$, and stretching each coordinate direction. When a coordinate is zero we go to the boundary of $D(V)$, and when a coordinate is 1/3 or greater, the corresponding vector $v_i$ does not appear in the linear combination, so we land in the span of the remaining vectors. This guarantees that $f_2(x)$ always lands in the correct subspace of $V$ specified by $f_1(x)$, as required.

    Altogether, the map $f$ starts with a point in $I^n$, represented as in \cref{apartment_2} by a configuration of cut points $c_i$ labeled by the vectors $v_i$. If we let $S$ denote the set of coordinates that are $\leq 1/3$, then $f$ is defined by giving a configuration $f_1(x) \in |\langle S \rangle \subseteq U \subseteq V|$ given by rescaling the remaining points in $S^c$ so that they fill up the interval $[0,1]$, and a point $f_2(x) \in D(\langle S \rangle)$. Each of these rules is continuous on each sub-cube $[0,1/3]^S \times [1/3,1]^{S^c}$, and therefore $f$ is continuous on the whole cube. The map $f$ is also apartment-like because when any coordinate is 0 the map $f_2$ lands in the boundary $S(V)$, and when the $i$th coordinate is 1 the maps $f_1$ and $f_2$ land in the span of $\{v_1, \ldots, \hat{v_i}, \ldots, v_n\}$.
\end{construction}

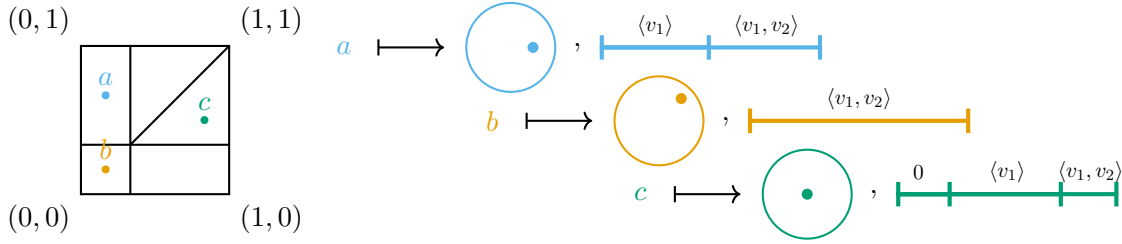
\begin{figure}[h]
\begin{center}
\resizebox{\textwidth}{!}{
\begin{tikzpicture}[scale=2]
	\draw[thick] 
		(0,0) node[anchor=north east]{$(0,0)$}
			-- 
		(1,0) node[anchor=north west]{$(1,0)$} 
			-- 
		(1,1) node[anchor=south west]{$(1,1)$} 
			-- 
		(0,1) node[anchor=south east]{$(0,1)$} 
			-- 
		cycle;
        \draw[thick] (1/3,1/3) -- (1,1);
        \draw[thick] (0,1/3) -- (1,1/3);
        \draw[thick] (1/3,0) -- (1/3,1);
        \draw[thick,fill,OI2] 
        	(1/6,2/3) circle (0.02) node[above]{$a$};
        \draw[thick,fill,OI1] 
        	(1/6,1/6) circle (0.02) node[above]{$b$};
        \draw[thick,fill,OI3] 
        	(5/6,1/2) circle (0.02) node[above]{$c$};
	
        \begin{scope}[xshift=4cm,scale=1.5]
			\node[OI3] at (-0.15,0) {$c$};
			\draw[thick,|->] (0,0) -- (0.3,0);
            \draw[thick,OI3] (.6,0) circle (0.2);
            \draw[thick,fill,OI3] (.6,0) circle (0.02);
            \node at (.9,0) {,};
			\draw[ultra thick,OI3,]
				(1.0,0) 
					edge[|-|] 
				(1.25,0)
					edge[-|] 
				(1.75,0)
					edge[-|]
				(2.0,0);
			\node[scale=0.75] at (1.1,0.1) {$0$};
			\node[scale=0.75] at (1.5,0.1) {$\langle v_1 \rangle$};
			\node[scale=0.75] at (1.875,0.1) 
				{$\langle v_1,v_2 \rangle$};
        \end{scope}

        \begin{scope}[xshift=3cm,scale=1.5,yshift=0.33cm]
			\node[OI1] at (-0.15,0) {$b$};
			\draw[thick,|->] (0,0) -- (0.3,0);
            \draw[thick,OI1] (.6,0) circle (0.2);
            \draw[thick,fill,OI1] (.7,.1) circle (0.02);
            \node at (.9,0) {,};
			\draw[ultra thick,OI1]
				(1.0,0) 
					edge[|-|]
				(2.0,0);
			\node[scale=0.75] at (1.5,0.1) {$\langle v_1,v_2 \rangle$};
        \end{scope} 
        
        \begin{scope}[xshift=2cm,scale=1.5,yshift=0.66cm]
			\node[OI2] at (-0.15,0) {$a$};
			\draw[thick,|->] (0,0) -- (0.3,0);
            \draw[thick,OI2] (.6,0) circle (0.2);
            \draw[thick,fill,OI2] (.7,0) circle (0.02);
            \node at (.9,0) {,};
			\draw[ultra thick,OI2,]
				(1.0,0) 
					edge[|-|] 
				(1.5,0)
					edge[-|]
				(2.0,0);
			\node[scale=0.75] at (1.25,0.1) {$\langle v_1 \rangle$};
			\node[scale=0.75] at (1.75,0.1) 
				{$\langle v_1,v_2 \rangle$};
        \end{scope} 
    \end{tikzpicture}
}
\end{center}
\caption{The doubly-suspended apartment-like map for $[v_1,v_2]$ from \cref{apartment_5}.}\label{double_apt_pt_figure}
\end{figure}

\begin{construction}\label{apartment_oriented_2}
    We make the doubly suspended apartment-like maps of \cref{apartment_4} respect orientation by collapsing $S^V$ onto the cone $C(P) \subseteq V$ modulo its boundary:
    \begin{equation}\label{eq:apt_oriented_pt}
    \begin{tikzcd}
    S^V \rar & \displaystyle \frac{C(P)}{\partial C(P)} \rar{f} & \Sigma\!\PT(V).
    \end{tikzcd}
    \end{equation}
    This gives an element of $H_0(\Omega^V\Sigma\!\PT(V),*)$, where $\Omega^V X = \textup{Map}_*(S^V,X)$.
\end{construction}

\begin{proposition}\label{solomon_tits_w_apts_4}
    For any inner product space $V$ of dimension $n \geq 0$, the oriented doubly suspended apartment-like maps for each generator of $\Pt(V)$ induce an $O(V)$-equivariant isomorphism
    \[ 
    	[\apt]\colon \Pt(V) \cong \pi_0(\Omega^V\Sigma^\infty\Sigma\!\PT(V)) \cong H_0(\Omega^V\Sigma^\infty\Sigma\!\PT(V),*) \cong H_0(\Sph^{-V} \sma \Sigma\!\PT(V)).
	\]
\end{proposition}
\begin{proof}
    The map is an isomorphism because it differs from the map in \cref{solomon_tits_w_apts_3} by an isomorphism, and we can check directly from \cref{apartment_oriented_2} that it is $O(V)$-equivariant.
\end{proof}

\subsection{Checking the comodule structure on apartments}

Now we check that our topological Dehn invariant defined on the spaces $\PT(X)$ agrees on homology with the classical Dehn invariant defined on the groups $\Pt(X)$.

\begin{proposition}
\label{coproduct_pi0_EH}
    If $X^n = E^n$ or $H^n$ with ambient vector space $V$ of dimension $(n+1)$, the comodule structure fits into the commuting diagram
    
\begin{adjustbox}{width=\textwidth}
\begin{tikzcd}[column sep=4em]
        \Pt(V) 
        	\ar[d,"\textup{Dehn (\cref{pt_dehn})}"'] 
			\ar{r}{[\apt]} 
			& 
		H_0(\Sigma^{-V}\Sigma\!\ST(V)) 
			\ar[d,"\textup{\cref{constr:d_1-coalg_comod_suspended}}"] 
			\\
        \displaystyle 
        	\bigoplus_{0 \neq U \subseteq V} 
				\Pt(U) \otimes 
				\tPt(V \ominus U) 
			\ar{r}{[\apt] \otimes [\apt]} 
			& 
		\displaystyle 
			\bigoplus_{0 \neq U \subseteq V} 
				H_0(\Sigma^{-U}\Sigma\!\ST(U)) \otimes 
				H_0(\Sigma^{-(V \ominus U)}
					\Sigma\!\ST(V \ominus U)).
    \end{tikzcd}
\end{adjustbox}
    
\end{proposition}

\begin{proof}
    As in \cite[Proposition 7.2]{KKMMW1}, we select the point in $D_1(2)$ corresponding to the point that cuts the interval exactly in half: the pair of intervals $[0,1/2]$ and $[1/2,1]$ in $[0,1]$. We describe the action of the comultiplication for this point on the oriented doubly suspended apartment
    \[ S^V \to \frac{I^{n+1}}{\partial I^{n+1}} \to \Sigma\!\ST(V). \]
    from \eqref{eq:apt_oriented_double} for a single fixed element $[v_0,\ldots,v_n] \in \Pt(V)$. Composing this with the map $\Sigma m$ from \cref{constr:d_1-coalg_comod_suspended} gives a map
    \[ S^V \to \bigvee_{\substack{0 \neq U\subseteq V \\ \textup{ in }\Dip^X }} \Sigma\!\ST(U) \wedge  \Sigma\!\ST(V\ominus U). \]
    As in the proof of \cite[Proposition 7.2]{KKMMW1}, this map is trivial except in the interior of $I^{n+1} \subseteq S^V$, where it can be described using another configuration space model: the coordinates of $I^{n+1}$ describe $(n+1)$ cut points in the interval $[0,1]$, each of which is labeled by the corresponding vector $v_i$. The intervals between these cut points are labeled by the span of all of the $v_i$ to their left.

    Since the coproduct collapses everything to $*$ if there is a cut point at $1/2 \in [0,1]$, the map is trivial except on the interior of $2^{n+1}$ sub-cubes $[0,1/2]^S \times [1/2,1]^{S^c} \subseteq I^{n+1}$, one for each subset $S \subseteq \{v_0,\ldots,v_n\}$. Our map takes each of these sub-cubes to $\Sigma\!\ST(U) \wedge \Sigma\!\ST(V\ominus U)$ where $U = \langle S \rangle$, by the product of the doubly-suspended apartment map for the simplex $[S]$ and the doubly suspended apartment map for the simplex $[\pr_{V \ominus U} S^c]$. (On the sub-cube $[1/2,1]^{n+1}$ where $S = \varnothing$, the map vanishes by the convention in \cref{constr:d_1-coalg_comod_suspended}.)
    In summary, we have the commuting diagram
    \[ 
    \begin{tikzcd}[arrows=rightarrow, column sep=8em]
        \frac{I^{n+1}}{\partial I^{n+1}} 
        	\ar{d}{\textup{collapse}} 
			\ar{r}{\apt[v_0,v_1,\ldots,v_n]} 
			& 
		\Sigma\!\ST(V) 
			\ar{d}{\textup{\cref{constr:d_1-coalg_comod_suspended}}} 
			\\
        \bigvee_{\varnothing \subsetneq 
        		 S \subseteq 
		 		 \{v_0,\ldots,v_n\}
				} 
        \frac{[0,1/2]^S}{\partial([0,1/2]^S)} \sma 
		\frac{[1/2,1]^{S^c}}{\partial([1/2,1]^{S^c})} 
			\ar{r}{\apt[S] \sma 
				   \apt[\pr_{V \ominus \langle S \rangle} S^c]
				  } 
			& 
		\bigvee_{0 \subsetneq U \subseteq V} 
			\Sigma\!\ST(U) \sma \Sigma\!\ST(V \ominus U),
    \end{tikzcd}
    \]
    where the vertical map on the left-hand side collapses the cube $I^{n+1}$ down to each of these $(2^{n+1}-1)$ sub-cubes.

    To conclude that the diagram on homology commutes, the last step is to show that the collapse map
    \[ S^V \to \frac{I^{n+1}}{\partial I^{n+1}} \to \frac{I^S}{\partial(I^S)} \sma 
		\frac{I^{S^c}}{\partial(I^{S^c})} \]
    has the same degree as the smash product of the collapse maps
    \[ S^V \cong S^U \sma S^{V \ominus U} \to \frac{I^S}{\partial(I^S)} \sma 
		\frac{I^{S^c}}{\partial(I^{S^c})} \]
    for each fixed subset $S$. (The two maps are not identical!) To see this, we define a homotopy between them by rescaling the first one so that instead of a collapse it gives a homeomorphism $\frac{I^{n+1}}{\partial I^{n+1}} \cong \frac{I^S}{\partial(I^S)} \sma 
	\frac{I^{S^c}}{\partial(I^{S^c})}$. Then both maps are collapses from $S^V$ onto cubes spanned by different bases in $V$, one of which is $S \cup S^c$, the other of which is $S \cup pr_{V \ominus \langle S \rangle} S^c$. We deform from one basis to another by moving each $v_i \in S^c$ along a straight line to its projection $\pr_{V \ominus \langle S \rangle} v_i$. Doing this for each $v_i$ simultaneously, we get a one-parameter family of bases for $V$. (They remain linearly independent since the ones in $U$ stay fixed and the ones not in $U$ have fixed projections to $V \ominus U$.) Therefore we get an isotopy of embeddings $I^S \times I^{S^c} \to V$, which shows that the two collapse maps have the same degree.
\end{proof}

\begin{proposition}\label{coproduct_pi0_S}
    If $X = S^{n-1}$ with ambient vector space $V$ of dimension $n$, the comodule structure fits into the commuting diagram
    
\begin{adjustbox}{width=\textwidth}
\begin{tikzcd}[column sep=4em]
        \Pt(V) 
        	\ar[d,"\textup{Dehn (\cref{pt_dehn})}"'] 
			\ar{r}{[\apt]} 
			& 
		H_0(\Sigma^{-V}\Sigma\!\PT(V)) 
			\ar[d,"\textup{\eqref{eq:comodule_one_point_in_D1J_S}}"] 
			\\
        \displaystyle 
        	\bigoplus_{0 \subseteq U \subseteq V} 
				\Pt(U) \otimes 
				\tPt(V \ominus U) 
			\ar{r}{[\apt] \otimes [\apt]} 
			& 
		\displaystyle 
			\bigoplus_{0 \subseteq U \subseteq V} 
				H_0(\Sigma^{-U}\Sigma\!\PT(U)) \otimes 
				H_0(\Sigma^{-(V \ominus U)}
					\Sigma\!\ST(V \ominus U)).
    \end{tikzcd}
\end{adjustbox}
\end{proposition}

\begin{proof}
    We use the explicit doubly-suspended apartment map from \cref{apartment_5}. We select the point in $D_1(2)$ corresponding to the point that cuts the interval exactly in half, and consider the comultiplication for this point. The resulting map
    \[ I^n \to \bigvee_{\substack{0 \subseteq U\subseteq V \\ \textup{ in }\Dip^S }} \Sigma\!\PT(U) \wedge \Sigma\!\ST(V\ominus U) \]
    vanishes except on the interior of $2^n$ sub-cubes, in which each coordinate $c_i$ is $\leq 2/3$ or $\geq 2/3$, because this is exactly where a cut point becomes $1/2$ after the rescaling operation. Each of these cubes goes to a smash product $\Sigma\!\PT(\langle S \rangle) \wedge \Sigma\!\ST(V\ominus \langle S \rangle)$ where $S$ is the set of coordinates that are $\leq 2/3$. The map on this sub-cube, after rescaling the sub-cube to have side length 1, is the product of the apartment map from \cref{apartment_5} (with 1/2 in the place of 1/3) for the set $S$, and the standard apartment map from \cref{apartment_2} for $S^c$ in $\Sigma\!\ST(V\ominus \langle S \rangle)$.

    This is almost enough to show that the square commutes. It remains to check that the composite of collapse maps
    \[ S^V \to \frac{C(P)}{\partial C(P)} \cong \frac{I^n}{\partial I^n} \to \frac{I^S}{\partial(I^S)} \sma 
		\frac{I^{S^c}}{\partial(I^{S^c})} \]
    has the same degree as the smash product of the collapse maps
    \[ S^V \cong S^U \sma S^{V \ominus U} \to \frac{I^S}{\partial(I^S)} \sma 
		\frac{I^{S^c}}{\partial(I^{S^c})}. \]
    The argument for this last step is the same as in the proof of \cref{coproduct_pi0_EH}.
\end{proof}

\subsection{Checking the algebra structure on apartments}

Next we check that our topological algebra structures defined on the spaces $\PT(X)$ agree on homology with the classical product structures defined on the groups $\Pt(X)$.

\begin{proposition}
\label{product_pi0_E}
    For two Euclidean geometries with ambient vector spaces $U$ and $V$, the algebra structure fits into the commuting diagram
    
\begin{tikzcd}[column sep=4em]
        \Pt(U) \otimes \Pt(V) 
        	\ar[d,"\textup{product (\cref{e_product})}"'] 
			\ar{r}{[\apt] \otimes [\apt]} 
			& 
		H_0(\Sigma^{-U_0}\ST(U)) \otimes H_0(\Sigma^{-V_0}\ST(V))
			\ar[d,"\textup{\cref{cor-extended-PTE-commalg}}"] 
			\\
        \Pt(U \times_\R V)
			\ar{r}{[\apt]} 
			& 
				H_0(\Sigma^{-(U_0 \oplus V_0)}
					\ST(U \times_\R V)).
    \end{tikzcd}
    
\end{proposition}

\begin{proof}
    For any two convex polytopes $P \subseteq E^m$ and $Q \subseteq E^n$, the product $P \times Q$ is a convex polytope in which every face is of the form $F \times F'$ for faces $F \subseteq P$ and $F' \subseteq Q$, and we have the equality $\langle F \times F' \rangle = \langle F \rangle \times_{\mathbb{R}} \langle F' \rangle$.
    
    From this we can deduce that for any two suspended apartment-like maps $P \to |0 \subsetneq A \subseteq U|$ and $Q \to |0 \subsetneq B \subseteq V|$ from \cref{apartment_1}, the product map $P \times Q \to |0 \subsetneq C \subseteq U \times_{\mathbb{R}} V|$ is also apartment-like: every face $F \times F'$ goes to the subspace $|0 \subsetneq C \subseteq \langle F \times F' \rangle|$ as required.

    Quotienting out by the boundary gives a commuting diagram of spaces of the form
    \[ \begin{tikzcd}[column sep=4em]
        \frac{P}{\partial P} \sma \frac{Q}{\partial Q}
        	\ar[d,"\cong"'] 
			\ar{r}{\apt \sma \apt} 
			& 
		\ST(U) \sma \ST(V)
			\ar[d,"\textup{\cref{cor-extended-PTE-commalg}}"] 
			\\
        \frac{P \times Q}{\partial(P \times Q)}
			\ar{r}{\apt} 
			& 
		\ST(U \times_\R V)
    \end{tikzcd} \]
    where the horizontal maps give the apartment classes for $P$, $Q$, and $P \times Q$. To handle the orientations introduced in \cref{apartment_oriented}, we observe that the collapse map $S^{U_0 \oplus V_0} \to \frac{P \times Q}{\partial(P \times Q)}$ is the smash product of the collapse maps $S^{U_0} \to \frac{P}{\partial P}$ and $S^{V_0} \to \frac{Q}{\partial Q}$.
\end{proof}

\begin{proposition}\label{product_pi0_S}
    For two spherical geometries with ambient vector spaces $U$ and $V$, the algebra structure fits into the commuting diagram
    
\begin{tikzcd}[column sep=4em]
        \Pt(U) \otimes \Pt(V) 
        	\ar[d,"\textup{product (\cref{s_product})}"'] 
			\ar{r}{[\apt] \otimes [\apt]} 
			& 
		H_0(\Sigma^{-U}\Sigma\!\PT(U)) \otimes H_0(\Sigma^{-V}\Sigma\!\PT(V))
			\ar[d,"\textup{\cref{constr:mult_spherical}}"] 
			\\
        \Pt(U \oplus V)
			\ar{r}{[\apt]} 
			& 
				H_0(\Sigma^{-(U \oplus V)}
					\Sigma\!\PT(U \oplus V)).
    \end{tikzcd}
    
\end{proposition}

\begin{proof}
    We make the same choice of apartment map as in \cref{coproduct_pi0_S} for each of the tuples $[u_1,\ldots,u_m]$ and $[v_1,\ldots,v_n]$, and observe that if we compose two of them with the direct sum functor $\Sigma\!\PT(U) \sma \Sigma\!\PT(V) \to \Sigma\!\PT(U \oplus V)$, we get exactly the same map as the apartment map for $[u_1,\ldots,u_m,v_1,\ldots,v_n]$. Then we compare orientations from \cref{apartment_oriented_2} and see that both routes give the same orientation; this is much easier than in the preceding proofs.
\end{proof}

All together, this finishes the proof of \cref{mainthm_pi0}.

\section{Two generalizations}

There are several ways to generalize the algebraic structures we constructed in this paper, and it would be difficult to give an exhaustive treatment. Two generalizations in particular seem worth mentioning though.

\subsection{Passing to a subfield of $\R$}
The first generalization is fairly simple: we could let $k \subseteq \R$ be any subfield and consider scissors congruence of polytopes whose vertices can be rescaled in the ambient vector space to points with coordinates in the ambient vector space lying in the subfield $k$, equivalently those cut out by hyperplanes defined by equations in $k$.

To be more precise, we let $X^n_k \subseteq X^n$ denote the subset of each geometry consisting only of those points with coordinates in $k$ in the ambient vector space $k^n$ or $k \oplus k^n$. Let $I(X^n;k) \leq I(X^n)$ be the subgroup of isometries that are defined over $k$ (so as a matrix all of the coefficients are in $k$). This leads to a modified definition of $\ST(X^n_k)$ and $\PT(X^n_k)$, where we only allow subspaces that came from linear subspaces of $k^n$ or $k \oplus k^n$. In the case of $\PT(S^n_k)$, each space in the diagram is still the entire sphere $S(U)$, not the subspace of those points defined over $k$. We also let $\Dip^X_k \subseteq \Dip^X$ be the wide subcategory consisting of those morphisms which, as linear maps on the ambient vector space, are defined over the field $k$. 

\begin{definition}\label{spectral_sah_module_k}
    For each type of geometry (Euclidean, hyperbolic, or spherical) we define the corresponding scissors congruence spectrum over $k$ by the wedge sum
    \[ \mathcal{E}_k = \bigvee_{n \geq 0} (EI(E^n;k)_+ \wedge_{I(E^n;k)} (\Sph^{-(\R \oplus \R^n)} \wedge \Sigma\!\ST(E^n_k))), \]
    \[ \mathcal{H}_k = \bigvee_{n \geq 0} (EI(H^n;k)_+ \wedge_{I(H^n;k)} (\Sph^{-(\R \oplus \R^n)} \wedge \Sigma\!\ST(H^n_k))), \]
    \[ \mathcal{S}_k = \bigvee_{n \geq 0} (EO(n;k)_+ \wedge_{O(n;k)} (\Sph^{-\R^n} \wedge \Sigma\!\PT(k^n))), \]
    \[ \widetilde{\mathcal{S}}_k = \bigvee_{n \geq 0} (EO(n;k)_+ \wedge_{O(n;k)} (\Sph^{-\R^n} \wedge \Sigma\!\ST(k^n))). \]
    By the same proof as in \cite{scissors_thom}, see also \cite{KLMMS-2} and \cite{KLMMS-3}, each of these agrees with the direct sum of scissors congruence $K$-theory spectra of the appropriate geometry in all dimensions:
    \[ \mathcal{E}_k \simeq \bigvee_{n \geq 0} K(\Pol{E^n_k}{I(E^n;k)}), \quad \mathcal{H}_k \simeq \bigvee_{n \geq 0} K(\Pol{H^n_k}{I(H^n;k)}), \quad \mathcal{S}_k \simeq \bigvee_{n \geq 0} K(\Pol{S^{n-1}_k}{O(n;k)}), \quad \widetilde{\mathcal{S}}_k \simeq \bigvee_{n \geq 0} \widetilde{K}(\Pol{S^{n-1}_k}{O(n;k)}), \]
\end{definition}

The proofs in this paper give a version of the main theorem in this case:

\begin{theorem}\label{mainthm_k}
    In the $\infty$-category of spectra,
\begin{itemize}
    \item $\widetilde{\mathcal{S}}_k$ is a Hopf algebra,
    \item $\mathcal{E}_k$ and $\mathcal{S}_k$ are comodule algebras over $\widetilde{\mathcal{S}}_k$, and
    \item $\mathcal{H}_k$ is a comodule over $\widetilde{\mathcal{S}}_k$.
\end{itemize}
    On $\pi_0$, these operations are given by the generalized Dehn invariants and the Cartesian and join products, just as before.
\end{theorem}

\subsection{Restricting the isometry group}\label{sec:restrict_isometry_group}

The other generalization concerns restricting the isometry group. For any symmetric monoidal subcategory $J \subseteq \Dip^X$, we could take the homotopy colimit of any of the scissors congruence functors along $J$ to get a type of scissors congruence spectrum where the isometry group has been restricted.

For instance, let $\Dip^{\{1\}} \subseteq \Dip^S$ be a subcategory constructed by choosing an orthonormal basis for each vector space $V$, and only allowing those isometries that preserve the choice of basis. Equivalently, $\Dip^{\{1\}}$ has just one vector space of each dimension, and no nontrivial automorphisms.

We may also choose $\Dip^T \subseteq \Dip^E$ by selecting an orthonormal basis for just the kernel $V_0$ of each of the maps $x_0\colon V \to \R$, and then restricting the morphisms to those that preserve the orthonormal basis on $V_0$. In effect, this means that we only take the translation subgroup of the isometry group $\R^n \leq I(E^n)$.

\begin{definition}\label{spectral_sah_module_translation}
    Define functors on these categories $\funSah^{\{1\}}$ and $\funSah^T$ by desuspending the functors $(*,\Sigma\!\ST)_E$ and $(\ST,*)_E$, as in \cref{def-spectral-sah}. Let $\widetilde{\mathcal{S}}_{\{1\}}$ and $\mathcal{T}$ denote the homotopy colimit of the results along $\Dip^{\{1\}}$ and $\Dip^T$, respectively. The homotopy colimit along $\Dip^{\{1\}}$ just takes a wedge sum of the spectra in the diagram for each dimension $n$, giving
    \[ \widetilde{\mathcal{S}}_{\{1\}} = \bigvee_{n \geq 0} \Sph^{-n} \wedge \Sigma\!\ST(\R^n). \]
    The homotopy colimit along $\Dip^T$ gives
    \[ \mathcal{T} = \bigvee_{n \geq 0} E\R^n_+ \wedge_{\R^n} (\Sph^{-\R^n} \sma \ST(E^n)). \]
    By \cite{scissors_thom}, $\mathcal{T}$ is the direct sum of all of the translational scissors congruence $K$-theory spectra, while $\widetilde{\mathcal{S}}_{\{1\}}$ is the direct sum of all the no-moving reduced spherical scissors congruence spectra:
    \[ \mathcal{T} \simeq \bigvee_{n \geq 0} K(\Pol{E^n}{\R^n}), \quad \widetilde{\mathcal{S}}_{\{ 1 \}} \simeq \bigvee_{n \geq 0} \widetilde{K}(\Pol{S^{n-1}}{1}). \]
\end{definition}

    The formula for $\mathcal T$ can be simplified by the following lemma, which uses the same argument as in \cref{switch_which_one_is_suspended} and \cite{scissors_thom}:
    \begin{lemma}\label{translation_action_trivial}
        The spectrum with $\R^n$-action $\Sph^{-\R^n}$ is equivalent to $\Sph^{-n}$ with trivial $\R^n$-action.
    \end{lemma}
    
    \begin{proof}
    We form a commuting diagram of spaces with $\R^n$-action
    \[ \begin{tikzcd}
        S^n & S^{\R^n} \\
        \R^n_+ \sma S^n \uar{\simeq} \rar[<->]{\simeq} & \R^n_+ \sma S^{\R^n} \uar{\simeq}
    \end{tikzcd} \]
    The bottom map has the formula $(u,v) \mapsto (u,u+v)$ for $u,v\in \R^n$. This shows that the two top spheres are equivalent as spaces with $\R^n$-action, and therefore their tensor inverses $\Sph^{-n}$ and $\Sph^{-\R^n}$ are equivalent as spectra with $\R^n$-action.
    \end{proof}

    So up to equivalence, the translation group acts trivially on the desuspension $\Sph^{-\R^n}$, giving the simpler formula
    \[ \mathcal{T} = 
    \bigvee_{n \geq 0} \Sph^{-n} \sma (E\R^n_+ \wedge_{\R^n} \ST(E^n)). \]

\begin{proposition}\label{comult_on_translational}
    The no-moving reduced spherical scissors congruence algebra $\widetilde{\mathcal{S}}_{\{1\}}$ is a (non-commutative) algebra in spectra with a comultiplication map of the form
    \[ \Sph^{-\R^n} \wedge \Sigma\!\ST(\R^n) \to \bigvee_{U \subseteq \R^n} \Sph^{-U} \sma \Sigma\!\ST(U) \sma \Sph^{-(\R^n \ominus U)} \sma \Sigma\!\ST(\R^n \ominus U). \]
    The translational scissors congruence algebra $\mathcal{T}$ is a (non-commutative) algebra in spectra with a coaction map of the form
    \begin{eqnarray*}
    \Sph^{-n} \sma \ST(E^n)_{h\R^n} \to \bigvee_{U \subseteq \R^n} \Sph^{-U} \sma \ST(U)_{hU} \sma \Sph^{-(\R^n \ominus U)} \sma \Sigma\!\ST(\R^n \ominus U).
    \end{eqnarray*}
\end{proposition}
\begin{proof}
    The algebra structure on $\mathcal{T}$ is obtained by composing the algebra structure on $\funSah^E$ from \cref{cor-ESah-multandcomod} with the canonical map of Day convolutions
    \[ \funSah^E \boxwedge_{\Dip^T} \funSah^E \to \funSah^E \boxwedge_{\Dip^E} \funSah^E. \]
    This satisfies the associativity and unitality conditions, but not a commutativity condition. In fact $\Dip^T$ is only a monoidal subcategory of $\Dip^E$, not a symmetric monoidal subcategory, because the symmetry isomorphism $E^m \times E^n \cong E^n \times E^m$ is not a translation. Taking the homotopy colimit along $\Dip^T$ as before gives an algebra structure on $\mathcal{T}$. The algebra structure on $\widetilde{\mathcal{S}}_{\{1\}}$ is constructed the same way.

    The comultiplication and coaction maps are obtained by taking a homotopy colimit along $\Dip^T$ and $\Dip^{\{1\}}$ of the comultiplication maps from \cref{constr:d_1-coalg_comod}. Note that before when we took the homotopy colimit along all of $\Dip^E$, the wedge sums over all subspaces $U \subseteq E^n$ were identified, giving a single smash product. But when we take the colimit only along the translations $\Dip^T$, every affine subspace is identified with all of its translates, but that still leaves one equivalence class of subspaces for every \emph{linear} subspace $U \subseteq \R^n$. So we get a wedge of smash products of homotopy orbits
    \[ \bigvee_{U \subseteq \R^n} \left(\Sph^{-U} \sma \ST(U) \sma \Sph^{-(\R^n \ominus U)} \sma \Sigma\!\ST(\R^n \ominus U)\right)_{hU} \]
    The action of the vector space $U$ on the smash product is trivial on $\Sph^{-(\R^n \ominus U)} \sma \Sigma\!\ST(\R^n \ominus U)$. It is equivalent to a trivial action on $\Sph^{-U}$ by the same argument as in \cref{translation_action_trivial}.
    Therefore the homotopy orbits pass to the $\ST(U)$ term, the only term where the action is nontrivial.
    
    The case of $\widetilde{\mathcal{S}}_{\{1\}}$ is similar except that there is no automorphism group to mod out by.
\end{proof}
One can deduce compatibilities for this comultiplication map from the same compatibilities in $\Fun(\Dip^E,\Sp)$, for instance the following coassociativity statement (with the desuspensions removed for readability):
\[ \begin{tikzcd}
    \displaystyle
    \ST(E^n)_{h\R^n} \rar \dar &
    \displaystyle
    \bigvee_{U_0 \subseteq \R^n} \ST(U_0)_{hU_0} \sma \Sigma\!\ST(\R^n \ominus U_0) \dar \\
    \displaystyle
    \bigvee_{U_1 \subseteq \R^n} \ST(U_1)_{hU_1} \sma \Sigma\!\ST(\R^n \ominus U_1) \rar &
    \displaystyle
    \bigvee_{U_0 \subseteq U_1 \subseteq \R^n} \ST(U_0)_{hU_0} \sma \Sigma\!\ST(U_1 \ominus U_0) \sma \Sigma\!\ST(\R^n \ominus U_1).
\end{tikzcd} \]
Although this is very useful, this is not exactly a comodule structure because it involves \emph{wedge sums} of smash products and not just smash products. In other words, the ``tensor product'' being taken at each vertex of this diagram is not the smash product of spectra, but rather a more complicated operation that remembers that the spectrum came from some diagram on $\Dip^E$.

\subsection{The higher Hadwiger invariant}
As an application, we show how to recover the classical Hadwiger invariant from \cite{hadwiger_invts,jessen_thorup,sah_79} and give a generalization to the higher translational scissors congruence groups. Continuing to work in Euclidean geometry, for each $0 \leq k \leq n$ and each string of vector subspaces $V^k \subsetneq V^{k+1} \subsetneq \cdots \subsetneq V^{n-1} \subsetneq V^n = \R^n$, we may apply the comultiplication map from \cref{comult_on_translational} a total of $(n-k)$ times, picking off just the codimension-one subspace $V^{n-1}$ the first time, then just $V^{n-2}$ the second time, and so on. This gives a string of maps of spectra (with desuspensions suppressed)
\begin{eqnarray*}
    \ST(E^n)_{h\R^n} &\to & \ST(V^{n-1})_{hV^{n-1}} \sma \Sigma\!\ST(\R^n \ominus V^{n-1}) \\
    &\to & \ST(V^{n-2})_{hV^{n-2}} \sma \Sigma\!\ST(V^{n-1} \ominus V^{n-2}) \sma \Sigma\!\ST(\R^n \ominus V^{n-1}) \\
    &\to & \cdots \\
    &\to & \ST(V^{k})_{hV^{k}} \sma \Sigma\!\ST(V^{k+1} \ominus V^{k}) \sma \cdots \sma \Sigma\!\ST(\R^n \ominus V^{n-1}).
\end{eqnarray*}
We may then compose these with a map of spectra on the first factor
\[ \Sph^{-V^k} \sma \ST(V^k) \to H\R \]
that truncates and maps $\pi_0 \cong \Pt(V^k)$ forward to $\R$ by taking every $k$-dimensional Euclidean polytope in $V^k$ to its $k$-dimensional volume. This respects the action of the translation group $V^k$, which acts trivially on the target $\R$, giving a map of homotopy orbit spectra
\[ \Sph^{-V^k} \sma \ST(V^k)_{hV^k} \to H\R_{hV^k}. \]
(In fact, this last map is exactly the volume trace from \cite{bgmmz}, by \cite[Thm 1.3]{bgmmz}.) Choosing an orientation for each of the quotients $V^{i+1} \ominus V^i$, we simplify down to a map of spectra
\[ \Sigma^{-n} \ST(E^n)_{h\R^n} \to H\R_{hV^k} \sma (\Sph^{-1} \sma \ST(\R^1))^{\sma (n-k)} \simeq H\R_{hV^k}. \]
Altogether, this gives us a map of spectra of the form
    \begin{equation}\label{eq:hadwiger_higher}
    \Sigma^{-n} \ST(E^n)_{h\R^n} \to \bigvee_{0 \leq k \leq n} \left( \bigvee_{V^k \subsetneq \cdots \subsetneq V^{n-1} \subsetneq V^n = \R^n} H\R_{hV^k} \right)
    \end{equation}
that we call the \emph{higher Hadwiger invariant}.

More conceptually, we can think of the higher Hadwiger invariant as the generalized Dehn invariant of the form
\[ K(\Pol{E^n}{\R^n}) \to \bigvee_{0 \leq k \leq n} \left( \bigvee_{V^k \subsetneq \cdots \subsetneq V^{n-1} \subsetneq V^n = \R^n} K(\Pol{V^k}{V^k}) \sma \widetilde K(\Pol{S^0}{1})^{\sma (n-k)} \right) \]
composed with choices of equivalences $\widetilde{K}(\Pol{S^0}{1}) \simeq \Sph$ and the volume trace from \cite{bgmmz}.
\begin{proposition}
    The higher Hadwiger invariant on $\pi_0$ gives the classical Hadwiger invariant from \cite{hadwiger_invts,sah_79}.
\end{proposition}

\begin{proof}
    The classical Hadwiger invariant is the map
    \begin{equation}\label{eq:hadwiger_classical}
        K_0(\Pol{E^n}{\R^n}) \to \bigoplus_{0 \leq k \leq n} \left( \bigoplus_{V^k \subsetneq \cdots \subsetneq V^{n-1} \subsetneq V^n = \R^n} \R \right)
    \end{equation}
    that takes each $n$-dimensional polytope $P$, identifies each string of faces $P^k \subsetneq \cdots \subsetneq P^{n-1} \subsetneq P$ of consecutive dimensions from $k$ to $n$, and takes the $k$-dimensional volume of $P^k$ as an element of $\R$. This is placed in the summand corresponding to the vector subspaces parallel to the affine spans of the faces $P^k$ through $P^{n-1}$. Finally, a sign is added to each of these terms by fixing any orientation for each of the quotients $V^{i+1} \ominus V^i$, and adding a sign of $\pm 1$ according to which side of $V^i$ each face $P^{i+1}$ lies on, relative to $P^i$. These signs are multiplied together for each value of $i$, giving a single sign of $\pm 1$ for each string of faces $P^k \subsetneq \cdots \subsetneq P^{n-1} \subsetneq P$.

    Our invariant can be calculated on $\pi_0$ by \cref{coproduct_pi0_EH}, which identifies it with a string of maps of abelian groups
\begin{eqnarray*}
    \Pt(E^n)_{h\R^n} &\to & \Pt(V^{n-1})_{V^{n-1}} \otimes \tPt(\R^n \ominus V^{n-1}) \\
    &\to & \Pt(V^{n-2})_{V^{n-2}} \otimes \tPt(V^{n-1} \ominus V^{n-2}) \otimes \tPt(\R^n \ominus V^{n-1}) \\
    &\to & \cdots \\
    &\to & \Pt(V^{k})_{V^{k}} \otimes \tPt(V^{k+1} \ominus V^{k}) \otimes \cdots \otimes \tPt(\R^n \ominus V^{n-1}) \\
    &\to & \R \otimes \Z \otimes \cdots \otimes \Z \\
    &\cong & \R.
\end{eqnarray*}
    Each map in this string is obtained by taking the $\R^m$-coinvariants (for some $m$) of a Dehn invariant as defined in \cref{pt_dehn}. Since the codimension is always one, the complementary polytope group $\tPt(V \ominus U)$ is always $\Z$, represented by either of the two points in the sphere $S(V \ominus U) \cong S^0$, which are negatives of each other. In this case the Dehn invariant simply selects a codimension-one face and remembers which side of the face the original polytope was on. Composing these together gives the same description as the classical Hadwiger invariant, for some choice of orientation convention for each of the groups $\tPt(V \ominus U)$.
\end{proof}

The classical Hadwiger invariant is known to be a complete invariant for translational scissors congruence. In other words, the map \eqref{eq:hadwiger_classical} is injective. This was proven by Hadwiger in dimensions up to three \cite{hadwiger_proof_of_invts} and by Jessen and Thorup in the general case \cite{jessen_thorup}. On the other hand, this higher version incorporates the higher version of the volume, and therefore by a result of Lemann \cite{lemann} the higher Hadwiger invariant is injective for one-dimensional Euclidean geometry $E^1$. It is therefore natural to conjecture:
\begin{conjecture}
    The higher Hadwiger invariant \eqref{eq:hadwiger_higher} is injective on the higher $K$-groups for every dimension $n$.
\end{conjecture}

Of course, these are not the only coproducts that we get on translational scissors congruence from \cref{comult_on_translational}, they are just the simplest ones. The other coproducts involve $\tPt(V \ominus U)$ when the codimension of $U$ in $V$ is larger than 1, which is a more complicated group.



\bibliographystyle{alpha}
\bibliography{references}

\end{document}